\documentclass[11pt]{amsart}

\usepackage{amsxtra,amssymb,amsmath,amscd,url,listings,stmaryrd}
\usepackage[alphabetic]{amsrefs}
\usepackage[utf8]{inputenc}
\usepackage{eucal}
\usepackage{fullpage}
\usepackage{scrtime}
\usepackage[colorlinks]{hyperref}
\usepackage{datetime}
\usepackage[all]{xy}
\usepackage{graphicx}
\usepackage{tikz-cd}

\shortdate

\renewcommand{\leq}{\leqslant}
\renewcommand{\geq}{\geqslant}
 
\numberwithin{equation}{section}

\newcommand{\uple}[1]{\text{\boldmath${#1}$}}

\def\stacksum#1#2{{\stackrel{{\scriptstyle #1}}
{{\scriptstyle #2}}}}

\newcommand{\bfalpha}{\uple{\alpha}}
\newcommand{\bfbeta}{\uple{\beta}}
\newcommand{\bfb}{\uple{b}}

\newcommand{\bfO}{\mathbf{0}}
\newcommand{\bfR}{\mathbf{R}}

\newcommand{\bfK}{\mathbf{K}}

\newcommand{\Cc}{\mathbf{C}}

\newcommand{\Aa}{\mathbf{A}}

\newcommand{\Zz}{\mathbf{Z}}
\newcommand{\Pp}{\mathbf{P}}
\newcommand{\Rr}{\mathbf{R}}
\newcommand{\Gg}{\mathbf{G}}
\newcommand{\Gm}{\mathbf{G}_{m}}

\newcommand{\Qq}{\mathbf{Q}}
\newcommand{\Fp}{{\mathbf{F}_p}}

\newcommand{\Fq}{{\mathbf{F}_q}}
\newcommand{\Fqt}{{\mathbf{F}^\times_q}}
\newcommand{\Fqd}{{\mathbf{F}_{q^d}}}

\newcommand{\Ff}{\mathbf{F}}

\newcommand{\bFq}{\overline{\Ff}_q}
\newcommand{\bQl}{\overline{\Qq}_{\ell}}

\newcommand{\mmu}{\boldsymbol{\mu}}

\newcommand{\mcV}{\mathcal{V}}
\newcommand{\mcO}{\mathcal{O}}
\newcommand{\mcN}{\mathcal{N}}
\newcommand{\mcH}{\mathcal{H}}
\newcommand{\HYPK}{\mathcal{K}\ell}
\newcommand{\KL}{\mathcal{K}\ell}

\newcommand{\mods}[1]{\,(\mathrm{mod}\,{#1})}

\newcommand{\what}{\widehat}

\DeclareMathOperator{\frob}{Fr}

\DeclareMathOperator{\hypk}{Kl}

\newcommand{\ra}{\rightarrow}
\newcommand{\lra}{\longrightarrow}

\newcommand{\injecte}{\hookrightarrow}

\newcommand{\fleche}[1]{\stackrel{#1}{\lra}}

\DeclareMathOperator{\Spec}{Spec}

\DeclareMathOperator{\rank}{rank}

\DeclareMathOperator{\Frob}{\mathrm{Fr}}

\DeclareMathOperator{\Kl}{\mathrm{Kl}}

\DeclareMathOperator{\tr}{\mathrm{Tr}}
\DeclareMathOperator{\nr}{\mathrm{Nr}}
\DeclareMathOperator{\Gal}{Gal}
\DeclareMathOperator{\Ind}{Ind}
\DeclareMathOperator{\Res}{Res}
\DeclareMathOperator{\supp}{supp}

\DeclareMathOperator{\Tr}{Tr}

\DeclareMathOperator{\swan}{Swan}
\DeclareMathOperator{\Sing}{Sing}
\DeclareMathOperator{\Drop}{drop}
\DeclareMathOperator{\sw}{Swan}

\DeclareMathOperator{\ft}{FT}
\DeclareMathOperator{\cond}{\mathbf{c}}

\newcommand{\eps}{\varepsilon}
\renewcommand{\rho}{\varrho}

\DeclareMathOperator{\SL}{SL}

\DeclareMathOperator{\GL}{GL}
\DeclareMathOperator{\PGL}{PGL}

\DeclareMathOperator{\rmB}{B}

\DeclareMathOperator{\Sp}{Sp}

\newcommand{\sheaf}[1]{\mathcal{{#1}}}

\DeclareMathSymbol{\gena}{\mathord}{letters}{"3C}
\DeclareMathSymbol{\genb}{\mathord}{letters}{"3E}

\def\multsum{\mathop{\sum\cdots \sum}\limits}

\theoremstyle{plain}
\newtheorem{theorem}{Theorem}[section]
\newtheorem*{theorem*}{Theorem}
\newtheorem{lemma}[theorem]{Lemma}
\newtheorem{cor}[theorem]{Corollary}

\newtheorem{proposition}[theorem]{Proposition}

\theoremstyle{remark}

\theoremstyle{definition}

\newtheorem{definition}[theorem]{Definition}

\newtheorem{remark}[theorem]{Remark}

\newcommand{\mcL}{\mathcal{L}}

\newcommand{\mcC}{\mathcal{C}}
\newcommand{\mcF}{\mathcal{F}}
\newcommand{\mcK}{\mathcal{K}}
\newcommand{\mcR}{\mathcal{R}}

\newcommand{\mcG}{\mathcal{G}}
\newcommand{\mcB}{\mathcal{B}}

\newcommand{\lf}{\lambda_f}

\newcommand{\vphi}{\varphi}

\renewcommand{\geq}{\geqslant}
\renewcommand{\leq}{\leqslant}

\newcommand{\refs}{\eqref}

\newcommand{\ov}[1]{\overline{#1}}

\newcommand\sumsum{\mathop{\sum\sum}\limits}
\newcommand\sumsumsum{\mathop{\sum\sum\sum}\limits}

\newcommand\rpfree{1/144}
\begin{document}

\title{Bilinear forms with Kloosterman sums and applications}

\author{Emmanuel Kowalski}
\address{ETH Z\"urich -- D-MATH\\
  R\"amistrasse 101\\
  CH-8092 Z\"urich\\
  Switzerland} \email{kowalski@math.ethz.ch}

\author{Philippe Michel} \address{EPFL/SB/TAN, Station 8, CH-1015
  Lausanne, Switzerland } \email{philippe.michel@epfl.ch}

\author{Will Sawin}
\address{ETH Institute for Theoretical Studies, ETH Zurich, 8092 Zürich}
\email{william.sawin@math.ethz.ch}

\begin{abstract}
  We prove non-trivial bounds for general bilinear forms in
  hyper-Kloosterman sums when the sizes of both variables may be below
  the range controlled by Fourier-analytic methods (P\'olya-Vinogradov
  range). We then derive applications to the second moment of cusp
  forms twisted by characters modulo primes, and to the distribution
  in arithmetic progressions to large moduli of certain
  Eisenstein-Hecke coefficients on $\GL_3$.  Our main tools are new
  bounds for certain complete sums in three variables over finite
  fields, proved using methods from algebraic geometry, especially
  $\ell$-adic cohomology and the Riemann Hypothesis.
\end{abstract}

\subjclass[2010]{11T23, 11L05, 11N37, 11N75, 11F66, 14F20, 14D05}

\keywords{Kloosterman sums, Kloosterman sheaves, monodromy, Riemann
  Hypothesis over finite fields, short exponential sums, moments of
  $L$-functions, arithmetic functions in arithmetic progressions}

\thanks{Ph.\ M.\ and E.\ K.\ were partially supported by a DFG-SNF
  lead agency program grant (grant 200021L\_153647). W. S. was
  supported by the National Science Foundation Graduate Research
  Fellowship under Grant No. DGE-1148900. A large portion of this
  paper was written while Ph.M. and W.S. where enjoying the
  hospitality of the Forschungsinstitut f\"ur Mathematik at ETH
  Zurich and it was continued during a visit of Ph. M.  at
  Caltech. We would like to thank both institutions for providing
  excellent working conditions. W.S. partially supported by Dr. Max
  R\"ossler, the Walter Haefner Foundation and the ETH Zurich
  Foundation. \today\ \currenttime}

\maketitle 

\begin{flushright}
  \textit{Dedicated to Henryk Iwaniec}
\end{flushright}

\setcounter{tocdepth}{1}
\tableofcontents

\section{Introduction}

\subsection{Statements of results} 

A number of important problems in analytic number theory can be
reduced to non-trivial estimates for bilinear forms
\begin{equation}\label{eq-bil}
  B(K,\bfalpha,\bfbeta)=\sum_{m}\sum_n \alpha_m \beta_n K(mn),
\end{equation}
for some arithmetic function $K$ and complex coefficients
$(\alpha_m)_{m\geq 1}$, $(\beta_n)_{n\geq 1}$.  A particularly
important case is when $K:\Zz\lra \Zz/q\Zz\ra\Cc$ runs over a sequence
of $q$-periodic functions, which are bounded independently of $q$, and
estimates are required in terms of $q$.
\par
In dealing with these sums, the challenges lie (1) in handling
coefficients $(\alpha_m)$, $(\beta_n)$ which are as general as
possible; and (2) in dealing with coefficients supported in intervals
$1\leq m\leq M$ and $1\leq n\leq N$ with $M$, $N$ as small as possible
compared with $q$.  In this respect, a major threshold is the
\emph{Fourier-theoretic range} (also called sometimes the
P\'olya-Vinogradov range), where $M$ and $N$ are both close to
$q^{1/2}$, and especially when they are slightly smaller in
logarithmic scale, so that applying the completion method and even
best-possible bounds for the Fourier transform gives trivial results.

In particular, when dealing with problems related to the analytic
theory of automorphic forms, one is often faced with the case where
$K(n)$ is a hyper-Kloosterman sum $\hypk_k(n;q)$. We recall that these
sums are defined, for $k\geq 2$ and $a\in (\Zz/q\Zz)^{\times}$, by
$$
\hypk_k(a;q)=\frac{1}{q^{(k-1)/2}} \sum_{\substack{x_1,\ldots,
    x_k\in\Zz/q\Zz\\x_1\cdots
    x_k=a}}e\Bigl(\frac{x_1+\cdots+x_k}{q}\Bigr).
$$
A deep result of Deligne shows that $|\hypk_k(a;q)|\leq k^{\omega(q)}$
for all $a\in (\Zz/q\Zz)^{\times}$. For any integer $c$ coprime to
$q$, we also denote by $[\times c]^*\hypk_k$ the function
$a\mapsto \hypk_k(ca;q)$.
\par
There are several intrinsic reasons why hyper-Kloosterman sums are
ubiquitous in the theory of automorphic forms:
\begin{itemize}
\item[-] they are closely related, via the Bruhat decomposition, to
  Fourier coefficients and Whittaker models of automorphic forms and
  representations, and therefore occur in the Kuznetsov-Petersson
  formula (see for instance the works of Deshouillers and
  Iwaniec~\cite{DesIw}, Bump--Friedberg--Goldfeld~\cite{BFG} or
  Blomer~\cite{Bl});
\item[-] the hyper-Kloosterman sums are the inverse Mellin transforms
  of certain monomials in Gauss sums, and therefore occur in
  computations involving root numbers in families of $L$-functions (as
  in the paper of Luo, Rudnick and Sarnak~\cite{LRS});
\item[-] the hyper-Kloosterman sums are constructed by iterated
  multiplicative convolution (see Katz's book~\cite{GKM} for the
  algebro-geometric version of this construction), which explains why
  they occur after applying the Voronoi summation formula on $\GL_k$.
\end{itemize}
\par
Our main results provide new bounds for general bilinear forms in
hyper-Kloosterman sums that go beyond the Fourier-theoretic range (see
Theorems~\ref{thmtypeII} and \ref{thmtypeI} below).  To illustrate the
potential of the results, we derive two applications of these bounds
in this paper. Both are related to the third source of
hyper-Kloosterman sums described above, but we believe that further
significant applications will arise from the other perspectives (as
well as from other directions).

\subsection{Bilinear forms with Kloosterman sums}

We will always assume that the sequences $\uple{\alpha}$ and
$\uple{\beta}$ have finite support. We denote
$$
\| \uple{\alpha} \|_1 = \sum_m |\alpha_m|,\quad\quad \| \uple{\alpha}
\|_2 = \Bigl(\sum_m |\alpha_m|^2\Bigr)^{1/2},
$$
the $\ell^1$ and $\ell^2$ norms.
\par
Our main result for general bilinear forms is the following:

\begin{theorem}[General bilinear forms]\label{thmtypeII}
  Let $q$ be a prime. Let $c$ be an integer coprime to $q$. Let $M$
  and $N$ be real numbers such that
$$
1\leq M\leq Nq^{1/4},\quad q^{1/4}<MN< q^{5/4}.
$$
\par
Let $\mcN\subset[1,q-1]$ be an interval of length $\lfloor N\rfloor$
and let $\uple{\alpha}=(\alpha_m)_{m\leq M}$ and
$\uple{\beta}=(\beta_n)_{n\in \mcN}$ be sequences of complex numbers.
\par
For any $\eps>0$, we have
\begin{equation}\label{typeIIgen}
  B([\times c]^*\hypk_k,\uple{\alpha},\uple{\beta})
  \ll q^\eps\|\uple{\alpha}\|_2\|\uple{\beta}\|_2(MN)^{\frac12}
  \Bigl(M^{-\frac12}+(MN)^{-\frac{3}{16}}q^{\frac{11}{64}}  \Bigr)
\end{equation}
where the implied constant depend only on $k$ and $\eps$.
\end{theorem}

\begin{remark}\label{remnontrivialII} 
  The bilinear form is easily bounded by
  $\|\uple{\alpha}\|_2\|\uple{\beta}\|_2(MN)^{\frac12}$, which we view
  as the trivial bound; a more elaborate treatment yields the bound of
  P\'olya-Vinogradov type (cf. \cite[Thm. 1.17]{FKM2})
  \begin{equation}\label{PVII}
    B([\times c]^*\hypk_k,\uple{\alpha},\uple{\beta})
    \ll_k \|\uple{\alpha}\|_2\|\uple{\beta}\|_2(MN)^{\frac12}
    \Bigl(q^{-\frac14}+M^{-\frac12}+N^{-\frac{1}{2}}q^{\frac{1}{4}}\log q  \Bigr),
\end{equation}
which improves the trivial bound as long as $M\gg 1$ and
$N\gg q^{1/2}\log^2q$. We then see that for $M=N$, the bound
\eqref{typeIIgen} is non-trivial as long as $M=N\geq q^{11/24}$, which
goes beyond the Fourier-theoretic range. In the special case
$M=N=q^{1/2}$, the saving factor is $q^{-1/64+\eps}$.
\end{remark}

When $\uple{\beta}$ is the characteristic function of an interval (or
more generally, by summation by parts, a ``smooth'' function; in
classical terminology, this means that the bilinear form is a ``type
I'' sum), we obtain a stronger result:

\begin{theorem}[Special bilinear forms]\label{thmtypeI} Let $q$ be a
  prime number.  Let $c$ be an integer coprime to $q$. Let $M$,
  $N\geq 1$ be such that
$$
1\leq M\leq N^2,\quad N<q,\quad MN<q^{3/2}.
$$ 
\par
Let $\uple{\alpha}=(\alpha_m)_{m\leq M}$ be a sequence of complex
numbers bounded by $1$, and let $\mcN\subset[1,q-1]$ be an interval of
length $\lfloor N\rfloor$.
\par
For any $\eps>0$, we have
\begin{equation}\label{typeIgen}
  B([\times c]^*\Kl_k,\uple{\alpha},1_\mcN)
  \ll q^{\eps} \|\uple{\alpha}\|_1^{1/2}
  \|\uple{\alpha}\|_2^{1/2}M^{\frac1{4}}
  N\times
  \Bigl(\frac{M^2N^5}{q^3}\Bigr)^{-1/12},
\end{equation}
where the implied constant depend only on $k$ and $\eps$.
\end{theorem}

\begin{remark}
  (1) A trivial bound in that case is
  $\|\uple{\alpha}\|_1^{1/2}\|\uple{\alpha}\|_2^{1/2}M^{1/4}N$, which
  explains why we stated the result in this manner. When $M=N$, we see
  that our bound \eqref{typeIgen} is non-trivial essentially when
  $M=N\geq q^{3/7}$, which goes even more significantly below the
  Fourier-theoretic range. In the special case $M=N=q^{1/2}$, the
  saving is $q^{-1/24+\eps}$.
\par
(2) For $k=2$, a slightly stronger result is proved by Blomer, Fouvry,
Kowalski, Michel and Mili{\'c}evi{\'c}~\cite[Prop. 3.1]{BFKMM}. This
builds on a method of Fouvry and Michel~\cite[\S VII]{FoMi}, which is
also the basic starting point of the analysis in this paper.
\par
(3) If $\uple{\alpha}$ and $\uple{\beta}$ are both characteristic
functions of intervals, a stronger result is proved by Fouvry,
Kowalski and Michel in~\cite[Th. 1.16]{FKM2} for a much more general
class of summands $K$, namely the trace functions of arbitrary
geometrically isotypic Fourier sheaves, with an implied constant
depending then on the conductor of these sheaves (for $M=N$, it is
enough there that $MN\geq q^{3/8}$, and for $M=N=q^{1/2}$, the saving
is $q^{-1/16+\eps}$).
\end{remark}


\subsection{Application 1: moments of twisted $L$-functions} 

Let $f$ and $g$ be fixed Hecke-eigenforms (of level $1$ say). 
A long-standing problem is the evaluation with power-saving error term
of the average
$$
\frac{1}{\vphi(q)}\sum_{\chi\mods{q}}L(f\otimes\chi,1/2)L(g\otimes\ov\chi,1/2),
$$
where $\chi$ runs over Dirichlet characters of prime conductor
$q$. When $f$ and $g$ are non-holomorphic Eisenstein series, the
problem becomes that of evaluating the fourth moment of Dirichlet
$L$-series at $1/2$. This was studied, for instance, by
Heath-Brown~\cite{HBfourth} and by Soundararajan~\cite{Sound}, and it
was solved by Young~\cite{Young}. For $f$ and $g$ cuspidal, this
question was studied by Gao, Khan and Ricotta \cite{GKR} and, with
different methods, by Hoffstein and Lee \cite{HL}.  Recently, the
problem was revisited in full generality by Blomer and
Mili{\'c}evi{\'c}~\cite{BlMi} and by Blomer, Fouvry, Kowalski, Michel
and Mili{\'c}evi{\'c}~\cite{BFKMM}. This last work solved the problem
when one of the two forms is non-cuspidal.  The general bilinear bound
of Theorem \ref{thmtypeII} (for $k=2$) is the final ingredient to the
resolution of this problem in the case where $f$ and $g$ are cuspidal.

\begin{theorem}[Moments of twisted cuspidal
  $L$-functions]\label{th-moments} Let $q$ be a prime
  number.  Let $f,g$ be cuspidal Hecke eigenforms (holomorphic forms
  or Maass forms) of level $1$ with respective root numbers $\eps(f)$
  and $\eps(g)$ (equal to $\pm 1$). If $f$ and $g$ are either both
  holomorphic forms or both Maass forms, assume also that
  $\eps(f)\eps(g)=1$.
\par
Let $\delta< \rpfree$. If $f\not=g$, we have
\begin{equation}\label{eqmixedmoment}
  \frac{1}{\vphi(q)}\sum_{\chi\mods q }L(f\otimes\chi,1/2)
  \overline{L(g\otimes\chi,1/2)}=\frac{2L(f\otimes
    g,1)}{\zeta(2)}+O(q^{-\delta}),	
\end{equation}
where $L(f\otimes g,1)\not=0$ is the value at $1$ of the
Rankin-Selberg convolution of $f$ and $g$, and the implied constant
depends only on $f$, $g$ and $\delta$.
\par
If $f=g$, then there exists a constant $\beta_f\in\Cc$ such that we
have
\begin{equation}\label{eqsecondmoment}
  \frac{1}{\vphi(q)}\sum_{\chi\mods q}|L(f\otimes\chi,1/2)|^2= \frac{2 L(
    \mathrm{sym}^2f,1)}{\zeta(2)}(\log
  q)+\beta_f+O(q^{-\delta}),
\end{equation}
where $L(\mathrm{sym}^2f,s)$ denotes the symmetric square $L$-function
of $f$, and the the implied constant depends only on $f$ and $\delta$.
\end{theorem}

\begin{proof}
  In~\cite[\S 7.2]{BFKMM}, Theorem~\ref{th-moments} (which is Theorem
  1.3 in loc. cit.) was shown to follow from a certain bound on a
  bilinear sum of Kloosterman sums (cf. the statement
  of~\cite[Prop. 3.1]{BFKMM}.)  That bound is exactly the case $k=2$
  and $c=1$ of Theorem~\ref{thmtypeII}.
\end{proof}

\begin{remark}
  (1) The assumption on the root number in Theorem~\ref{th-moments} is
  necessary, since otherwise the special values vanish and the sums
  are identically $0$.
\par
(2) 
It is well-established that an asymptotic formula with a power saving
error term for some moment in a family of $L$-functions typically
implies the possibility of evaluating asymptotically some additional
``twisted'' moments, in this case those of the shape
$$
\frac{1}{\vphi(q)}\sum_{\chi\mods q}L(f\otimes\chi,1/2)
\overline{L(g\otimes\chi,1/2)}\chi(\ell/\ell'),
$$
where $1\leq \ell,\ell'\leq L$ are coprime integers which are also
coprime with $q$ and $L=q^\eta$ for some fixed absolute constant
$\eta>0$.

Using such a formula for $f=g$, we may apply the \emph{mollification
  method} and the \emph{resonance method}, and obtain further results
on the special values for this family of $L$-functions (estimates for
the distribution of the order of vanishing at $s=1/2$, existence of
large values, for instance).  This will be taken up in the forthcoming
paper \cite{2for6} jointly with Blomer, Fouvry and Mili{\'c}evi{\'c}.
\end{remark}

\subsection{Application 2: arithmetic functions in arithmetic
  progressions}

In our second application, we use the bound for special bilinear forms
when $K=\Kl_3$ to study the distribution in arithmetic progressions to
large moduli of certain arithmetic functions which are closely related
to the triple divisor function.

\begin{theorem}\label{th-lfone}
  Let $f$ be a holomorphic primitive cusp form of level $1$ with Hecke
  eigenvalues $\lf(n)$, normalized so that $|\lf(n)|\leq d_2(n)$.
\par
For $n\geq 1$, let 
$$
(\lf\star 1)(n)=\sum_{d\mid n}\lf(d).
$$
\par
For $x\geq 2$, for any $\eta<1/102$, for any prime
$q\leq x^{1/2+\eta}$, for any integer $a$ coprime to $q$ and for any
$A\geq 1$, we have
$$
\sum_{\stacksum{n\leq x}{n\equiv a\mods{q}}} (\lf\star
1)(n)-\frac{1}{\varphi(q)} \sum_\stacksum{n\leq x}{(n,q)=1}(\lf\star
1)(n)\ll \frac{x}q(\log x)^{-A}
$$
where the implied constant depends only on $(f,\eta,A)$.
\end{theorem}

When $f$ is replaced by a specific non-holomorphic Eisenstein series,
we obtain as coefficients $(\lf\star 1)(n)=(d_2\star 1)(n)=d_3(n)$,
the triple (or ternary) divisor function. In that case, a result with
exponent of distribution $>1/2$ as above was first obtained (for
general moduli) by Friedlander and Iwaniec~\cite{FrIw}. This was
subsequently improved by Heath-Brown~\cite{HB} and more recently (for
prime moduli) by Fouvry, Kowalski and Michel~\cite{FKM3}.
\par
The approach of~\cite{FKM3} relied ultimately on bounds for the
bilinear sums $B(\Kl_3,\uple{\alpha},\uple{\beta})$ when \emph{both}
sequences $\uple{\alpha}$ and $\uple{\beta}$ are smooth. Indeed, as
already recalled, a very general estimate for
$B(K,\uple{\alpha},\uple{\beta})$ was proved in that case
in~\cite{FKM2}. Here, in the cuspidal case, the splitting
$d_2(n)=(1\star 1)(n)$ is not available and we need instead a bound
where only one sequence is smooth, which is given by
Theorem~\ref{thmtypeI} (we could of course also use
Theorem~\ref{thmtypeII}, with a slightly weaker result).
\par

The functions $n\mapsto d_3(n)=(1\star 1\star 1)(n)$ and
$n\mapsto (\lf\star 1)(n)$ are the Hecke eigenvalues of certain
non-cuspidal automorphic representation of $\GL_{3,\Qq}$, namely the
isobaric representations $1\boxplus 1\boxplus 1$ and
$\pi_f\boxplus 1$. The methods of \cite{FrIw,HB,FKM3} and of the
present paper can be generalized straightforwardly to show that the
$n$-th Hecke eigenvalue function of any fixed non-cuspidal automorphic
representation of $\GL_{3,\Qq}$ has exponent of distribution $>1/2$,
for individual prime moduli. Extending this further to cuspidal
$\GL_{3,\Qq}$-representation is a natural and interesting challenge.
\par
Theorem~\ref{th-lfone} is proved in section~\ref{secternary}.

\subsection{Further developments}\label{sec-extensions}

We describe here some possible extensions of our results, which will
be the subject of future papers.

\subsubsection{Extension to other trace functions}

A natural problem is to try to extend Theorems~\ref{thmtypeII}
and~\ref{thmtypeI} to more general trace functions $K$. In
\cite{FoMi}, Fouvry and Michel derived non-trivial bounds as in
Theorem~\ref{thmtypeI} (type I sums) when $\Kl_k$ is replaced by a
\emph{rational phase function} of the type
$$
K_f(n)=
\begin{cases}e_q(f(n))&\text{ if $n$ is not a pole of $f$}\\
0&\text{ otherwise},
\end{cases}
$$
where $q$ is prime, $e_q(x)=\exp(2\pi i\frac{x}q)$ and $f\in\Fq(X)$ is
some rational function which is not a polynomial of degree $\leq 2$.
They proved bounds similar to Theorem~\ref{thmtypeII} (type II sums)
for $K$ given by a \emph{quasi-monomial} phase, defined as above with
$$
f=aX^d+bX
$$
for some $a,b\in\Fq$, $a\not=0$ and $d\in\Zz-\{0,1,2\}$. While both
cases relied on arguments from algebraic geometry, they were
different, and far simpler, than those involved in the present work.

It is plausible that the methods developed in the present paper would
allow for an extension of Theorems \ref{thmtypeI} and \ref{thmtypeII}
to many of the families of exponential sums studied in great details
in the books of Katz (in particular in~\cite{GKM,ESDE}). 
Other potentially interesting variants that could be treated by the
methods presented here are bilinear sums of the shape
$$
\sum_{m,n}\alpha_m\beta_n K((m^dn)^{\pm 1}),\ d{\geq 1}\text{ fixed}.
$$ 
Again the case where $K$ is a hyper-Kloosterman sum (possibly
including multiplicative characters) seem particularly interesting for
number theoretic applications (see the recent work of
Nunes~\cite{Nunes}, for instance).

\subsubsection{Extension to composite moduli}

In this paper, we have focused our attention on bilinear forms
associated to functions $K$ which are periodic modulo a prime
$q$. This is in some sense the hardest case, but nevertheless it would
be very useful for many applications to have bounds similar to those
of Theorems~\ref{thmtypeI} and~\ref{thmtypeII} when the modulus $q$ is
arbitrary, or at least squarefree.

For instance, Blomer and Mili{\'c}evi{\'c}~\cite[Thm 1]{BlMi} proved the
analogue of the asymptotic formula in Theorem~\ref{th-moments} with
power saving error term when the modulus $q$ admits a factorization
$q=q_1q_2$ where $q_1$ and $q_2$ are neither close to $1$ (excluding
therefore the case when $q$ is prime, which is now solved by Theorem
\ref{th-moments}) nor to $q^{1/2}$. This excludes the case when $q$ is
a product of two distinct primes which are close to each other; it
would be possible to treat this using if a version of
Theorem~\ref{thmtypeII} for composite moduli was available.

Another direct application would be a version of
Theorem~\ref{th-lfone} for general moduli $q$. This would immediately
imply the following shifted convolution bound: there exists a constant
$\delta>0$, independent of $f$ and $h$, such that for all $N\geq 1$
and $h\geq 0$, we have
$$
\sum_{n\leq N}(\lf\star 1)(n)d_2(n+h)\ll_{f} N^{1-\delta}
$$
where the implied constant is independent of $h$. We refer to the
works of Munshi~\cite{Mu1,Mu2} and Topacogullari~\cite{Topa} for
related results.

Other potential applications are to problems involving the
Petersson-Kuznetsov trace formula (the first of the three items listed
in the beginning of this introduction) as well as to the study of
arithmetic functions (like the primes) in arithmetic progressions to
large moduli, as suggested by Theorem~\ref{th-lfone}.

\subsection{Structure of the proofs}

We now discuss the essential features of the proofs of our bounds for
bilinear sums, in the more difficult case of general coefficients
$\uple{\alpha}$ and $\uple{\beta}$.  Several aspects of the proof are
not specific to the case of hyper-Kloosterman sums. In view of
possible extensions to new cases, we describe the various steps in a
general setting and indicate those which are currently restricted to
the case of hyper-Kloosterman sums.

Let $q$ be a prime, and let $K$ be the $q$-periodic trace function of
some $\ell$-adic sheaf $\mcF$ on $\Aa^1_\Fq$, which we assume to be a
middle-extension pure of weight $0$, geometrically irreducible and of
conductor $\cond(\mcF)$. We think of $q$ varying, while the conductor
$\cond(\mcF)$ is bounded independently of $q$ (for the case of
hyper-Kloosterman sums, the sheaf $\mcF=\Kl_k$ is the Kloosterman
sheaf, defined by Deligne and studied by Katz~\cite{GKM}). We denote
by $\psi$ a fixed non-trivial additive character of $\Fq$.
\par
The problem of bounding the general bilinear sums
$B(K,\uple{\alpha},\uple{\beta})$, with non-trivial bounds slightly
below the Fourier-theoretic range, can be handled by the following
steps.
\par
\smallskip
\par
(1) We consider auxiliary functions $\bfK$ and $\bfR$, of the ``sum of
product'' type, defined by
$$ 
\bfK(r,s,\lambda,\bfb)=\psi(\lambda s)\prod_{i=1}^2 K( s (r+b_i ))
\overline{K(s (r+b_{i+2}))}
$$ 
and
$$
  \bfR(r,\lambda,\bfb):=\sum_{s\in\Fq}\bfK(r,s,\lambda,\bfb),
$$
where $r$, $s$ and $\lambda$ are in $\Fq$, and
$\bfb=(b_1,b_2,b_3,b_4)\in\Ff_q^4$. 
\par
Building on methods developed in~\cite{FoMi} (also inspired by the
work of Friedlander-Iwaniec~\cite{FrIw} and the ``shift by $ab$''
trick of Vinogradov and Karatsuba), we reduce the problem in
Section~\ref{sec-incomplete} to that of obtaining square-root
cancellation bounds for two complete exponential sums involving $\bfK$
and $\bfR$. Precisely, we need to obtain bounds of the type
\begin{equation}\label{Ssum}
\sum_{r\mods q}\bfK(r,s,0,\bfb)\ll q^{1/2},
\end{equation}
for $s\in\Fqt$, as well as \emph{generic} bounds
\begin{align}
\label{Rsumbound}
\sum_{r\mods q}\bfR(r,\lambda,\bfb)&\ll q,\\
\label{correlationsumbound}
\sum_{r\mods q}\bfR(r,\lambda,\bfb)\ov{\bfR(r,\lambda',\bfb)}
&=\delta(\lambda,\lambda')q^2+O(q^{3/2}).
\end{align}
\par
Here, ``generic'' means that the bounds should hold for every
$\lambda\in\Fq$ provided $\bfb$ does not belong to some proper
subvariety of $\Aa^4$. Of course, the implied constants in all these
estimates must be controlled by the conductor of $\mcF$, but this can
be achieved relatively easily in all cases using general arguments to
bound suitable Betti numbers independently of $q$.
\par
We will obtain the bounds~(\ref{Ssum}),~(\ref{Rsumbound})
and~(\ref{correlationsumbound}) from Deligne's general form of the
Riemann Hypothesis over finite fields~\cite{WeilII}. A crucial feature
is that we can interpret the functions $\bfK$ and $\bfR$ themselves as
trace functions of suitable $\ell$-adic sheaves denoted $\mcK$ (on
$\Aa^7$) and $\mcR$ (on $\Aa^6$) respectively. We call the latter
the \emph{sum-product transform sheaf} associated to the input sheaf
$\mcF$, to emphasize the structure of its trace functions and the
``$+ab$'' trick.
\par
Using the Grothendieck--Lefschetz trace formula and Deligne's form of
the Riemann Hypothesis, we see that the bounds will result if we can
show the following properties of these sheaves:
\begin{itemize}
\item The sheaf representing $r\mapsto \bfK(r,s,0,\bfb)$ is
  geometrically irreducible and geometrically non-trivial;
\item The sheaf $\mcR_{\lambda,\bfb}$ with trace function
  $r\mapsto \bfR(r,\lambda,\bfb)$ is geometrically irreducible, and
  $\mcR_{\lambda,\bfb}$ is not geometrically isomorphic to
  $\mcR_{\lambda',\bfb}$ if $\lambda'\not=\lambda$.
\end{itemize}
\par
This is a natural and well-established approach, but the
implementation of this strategy will require very delicate geometric
analysis of the $\ell$-adic sheaves involved.
\par
\smallskip
\par
(2) The first bound~\eqref{Ssum} is proved in great generality in
Section~\ref{sec-complete} using the ideas of Katz around the
Goursat-Kolchin-Ribet criterion (see~\cite[Prop. 1.8.2]{ESDE})
following the general discussion of sums of products by Fouvry,
Kowalski and Michel in~\cite{FKMSP}. Indeed, it is sufficient that the
original sheaf $\mcF$ with trace function $K$ be a ``bountiful'' sheaf
in the sense of \cite[Def. 1.2]{FKMSP}, a class that contains many
interesting sheaves in analytic number theory (in particular,
Kloosterman sheaves).
\par
\smallskip
\par
(3) To prove that the sheaf representing $r\mapsto
\bfR(r,\lambda,\bfb)$ is geometrically irreducible is much more
involved. As a first step, we prove (also in
Section~\ref{sec-complete}) a weaker generic irreducibility property,
where both $\bfb$ and $\lambda$ are variables. Indeed, using Katz's
diophantine criterion for irreducibility~\cite[\S 7]{rigid}), it
suffices to evaluate asymptotically the second moment of the relevant
trace function over all finite extensions $\Fqd$ of $\Fq$, and to
prove that
\begin{gather*}
  \frac{1}{(q^d)^{5}}\sum_{(r,\bfb)\in\Ff_{q^d}^5}
  |\bfR(r,0,\bfb;\Fqd)|^2=q^d(1+o(1)),\\
  \frac{1}{(q^d)^2}\sum_{(r,\lambda)\in\Ff_{q^d}^2}
  |\bfR(r,\lambda,\bfb;\Fqd)|^2=q^d(1+o(1))),
\end{gather*}
as $d\ra +\infty$.  Again, the methods are those of~\cite{FKMSP} and
require only that $\mcF$ be a bountiful sheaf.
\par
\smallskip
\par
(5) The next and final step is the crucial one, and is the deepest
part of this work. In the very long Section~\ref{sec-irreducible}, we
show that one can ``upgrade'' the generic irreducibility of $\mcR$
from the previous step to \emph{pointwise} irreducibility of the sheaf
deduced from $\mcR$ by fixing the values of $\lambda$ and $\bfb$,
where only $\bfb$ is required to be outside some exceptional set.
This step uses such tools as Deligne's semicontinuity theorem and
vanishing cycles. It requires quite precise information on the
ramification properties of $\mcK$ and $\mcR$. At this stage, we need
to build on the precise knowledge of the local monodromy of
Kloosterman sheaves $\KL_k$, which is again due to Katz~\cite{GKM}.
We will give some indications of the ideas involved in
Section~\ref{sec-irreducible}.

\subsection*{Notation} 

We write $\delta(x,y)$ for the Kronecker delta symbol.

For any prime number $\ell$, we assume fixed an isomorphism
$\iota\,:\, \bQl\to \Cc$.  Let $q$ be a prime number. Given an
algebraic variety $X_{\Fq}$, a prime $\ell\not=q$ and a constructible
$\bQl$-sheaf $\mcF$ on $X$, we denote by $t_{\mcF}\,:\, X(\Fq)\lra
\Cc$ its trace function, defined by
$$
t_{\mcF}(x)=\iota(\Tr(\frob_{x,\Fq}\mid \mcF_{x})),
$$
where $\mcF_x$ denotes the stalk of $\mcF$ at $x$. More generally, for
any finite extension $\Fqd/\Fq$, we denote by $t_{\mcF}(\cdot;\Fqd)$
the trace function of $\mcF$ over $\Fqd$, namely
$$
t_\mcF(x;\Fqd)=\iota(\tr(\Frob_{x,\Fqd}\mid \mcF_x)).
$$
\par
We will usually omit writing $\iota$; in any expression where some
element $z$ of $\bQl$ has to be interpreted as a complex number, we
mean to consider $\iota(z)$.
\par
We denote by $\cond(\mcF)$ the conductor of a constructible
$\ell$-adic sheaf $\sheaf{F}$ on $\Aa^1_{\Fq}$ as defined
in~\cite{FKM1} (with adaptation to deal with sheaves which may not be
middle-extensions). Recall that this is the non-negative integer given
by
$$
\cond(\mcF)=\rank(\mcF)+|\Sing(\mcF)|+\sum_{x\in
  \Sing(\mcF)}\swan_x(\mcF)+ \dim H^0_c(\Aa^1_{\bFq},\sheaf{F}),
$$
where $\Sing(\mcF)\subset \Pp^1(\bFq)$ is the set of ramification
points of $\mcF$ and $\swan_x(\mcF)$ is the Swan conductor at $x$.


For convenience, we recall the general version of the Riemann
Hypothesis over finite fields that will be the source of our
estimates.

\begin{proposition}\label{pr-recall-rh}
  Let $\Fq$ be a finite field with $q$ elements. Let $\mcF$ and $\mcG$
  be constructible $\ell$-adic sheaves on $\Aa^1_{\Fq}$ which are
  geometrically irreducible, mixed of weights $\leq 0$ and pointwise
  pure of weight $0$ on a dense open subset. We have
$$
\sum_{x\in\Fq} t_{\mcF}(x;\Fq)\overline{t_{\mcG}(x;\Fq)}\ll 
\sqrt{q}
$$
unless $\mcF$ is geometrically isomorphic to $\mcG$, and
$$
\sum_{x\in\Fq} |t_{\mcF}(x;\Fq)|^2=q+O(\sqrt{q}).
$$
\par
The implied constants depend only on the conductors of $\mcF$ and
$\mcG$.
\end{proposition}

We denote by $\sheaf{F}^{\vee}$ the dual of a constructible sheaf
$\mcF$; if $\mcF$ is a middle-extension sheaf, we will use the same
notation for the middle-extension dual.
\par
Let $\psi$ (resp. $\chi$) be a non-trivial additive
(resp. multiplicative) character of $\Fq$. We denote by $\mcL_\psi$
(resp. $\mcL_\chi$) the associated Artin-Schreier (resp. Kummer) sheaf
on $\Aa^1_\Fq$ (resp. on $(\Gg_m)_{\Fq}$), as well (by abuse of
notation) as their middle extension to $\Pp^1_\Fq$. The trace
functions of the latter are given by
\begin{gather*}
  t_\psi(x;\Fqd)=\psi(\tr_{\Fqd/\Fq}(x))\quad\text{ if } x\in\Fqd,\quad
  t_{\psi}(\infty;\Fqd)=0,\\
  t_\chi(x;\Fqd)=\chi(\nr_{\Fqd/\Fq}(x))\quad\text{ if }
  x\in\Ff_{q^d}^{\times},\quad
  t_{\chi}(0;\Fqd)=t_{\chi}(\infty;\Fqd)=0
\end{gather*}
(which we denote also by $\psi_{q^d}(x)$ and by $\chi_{q^d}(x)$,
respectively). For the trivial additive or multiplicative character,
the trace function of the middle-extension is the constant function
$1$.
\par
Given $\lambda\in\Fqd$, we denote by $\mcL_{\psi_\lambda}$ the
Artin-Schreier sheaf of the character of $\Fqd$ defined by
$x\mapsto \psi(\tr_{\Fqd/\Fq}(\lambda x))$.
\par
If $q\geq 3$, we denote by $\chi_2$ the Legendre symbol on $\Fq$.
\par
If $X_{\Fq}$ is an algebraic variety, $\psi$ (resp. $\chi$) is an
$\ell$-adic additive character of $\Fq$ (resp. $\ell$-adic
multiplicative character) and $f\,:\, X\lra \Aa^1$ (resp.
$g\,:\, X\lra \Gg_m$) is a morphism, we denote by either
$\sheaf{L}_{\psi(f)}$ or $\sheaf{L}_{\psi}(f)$ (resp. by
$\sheaf{L}_{\chi(g)}$ or $\sheaf{L}_{\chi}(g)$) the pullback
$f^*\sheaf{L}_{\psi}$ of the Artin-Schreier sheaf associated to $\psi$
(resp. the pullback $g^*\sheaf{L}_{\chi}$ of the Kummer sheaf). These
are lisse sheaves on $X$ with trace functions $x\mapsto \psi(f(x))$
and $x\mapsto \chi(g(x))$, respectively. The meaning of the notation
$\sheaf{L}_{\psi}(f)$, which we use when putting $f$ as a subscript
would be typographically unwieldy, will always be unambiguous, and no
confusion with Tate twists will arise.
\par
Given a variety $X/\Fq$, an integer $k\geq 1$ and a function $c$ on
$X$, we denote
by $\mcL_{\psi}(c s^{1/k})$ the sheaf on $X\times \Aa^1$ (with
coordinates $(x,s)$) given by
$$
\alpha_* \mcL_{\psi (c(x)t)}
$$ 
where $\alpha$ is the covering map $(x,s,t)\mapsto (x,s)$ on the
$k$-fold cover
$$
\{(x,s,t)\in X\times\Aa^1\times\Aa^1\,\mid\, t^k=s\}.
$$
\par
Given a field extension $L/\Fp$, and elements $\alpha\in L^{\times}$
and $\beta\in L$, we denote by $[\times\alpha]$ the scaling map
$x\mapsto \alpha x$ on $\Aa^1_L$, and by $[+\beta]$ the additive
translation $x\mapsto x+\beta$. For a sheaf $\mcF$, we denote by
$[\times\alpha]^*\mcF$ (resp. $[+\alpha]^*\mcF$) the respective
pull-back operation. More generally given an element
$\gamma=\begin{pmatrix}a&b\\c&d
\end{pmatrix}\in\PGL_2$, we denote by $\gamma^*\mcF$ the
pullback under the fractional linear transformation on $\Pp^1$ given
by
$$
\gamma\cdot x=\frac{ax+b}{cx+d}.
$$ 
We usually omit to mention any necessary base change to $L$ if the
matrix involved is in $\PGL_2(L)$ for some extension $L/\Fq$.
\par
We will usually not indicate base points in \'etale fundamental
groups; whenever this occurs, it will be clear that the properties
under consideration are independent of the choice of a base point.
\par


As mentioned above, a large portion of our argument is valid for a
more general class of functions $K$ than hyper-Kloosterman sums. We
now state the definition of the relevant class of sheaves, which is a
slight extension of~\cite[Def. 1.2]{FKMSP}. Let $\mcG$ be a
middle-extension sheaf on $\Aa^1$ of rank $k\geq 2$, which is pure of
weight $0$. Let $U=\Aa^1-S_\mcG$ denote the maximal open subset where
$\mcG$ is lisse, and let $\cond(\mcG)$ be the conductor of $\mcG$.
Let $\mcF$ be \emph{either} $\mcG$ or the extension by zero to $\Aa^1$
of $\mcG|U$.  
\par

\begin{definition}\label{def-type}
We say that $\mcF$ is \emph{bountiful} (resp. \emph{bountiful with
  respect to the upper-triangular Borel subgroup $\rmB\subset\PGL_2$})
if
\begin{itemize}
\item[--] The geometric and arithmetic monodromy groups of the lisse
  sheaf $\mcF|U$, or equivalently of $\mcG|U$, coincide and are equal
  either $\SL_k$ if $k\geq 3$ or to $\Sp_k$. Accordingly, we will
  say that $\mcF$ (or $\mcG$) is of $\Sp$ or $\SL$ type.
\item[--] For any non-trivial element $\gamma\in \PGL_2(\bFq)$
  (resp. in $\rmB(\bFq)$), the sheaf $\gamma^*\mcG$ is not
  geometrically isomorphic to $\mcG\otimes\mcL$ for any rank $1$ sheaf
  $\mcL$.
\item[--] If $\mcF$ is of $\SL$-type, there is at most one $\xi\in
  \PGL_2(\bFq)$ (resp. $\xi\in \rmB(\bFq)$) such that we have a
  geometric isomorphism
$$
\xi^*\mcG\simeq \mcG^\vee\otimes\mcL
$$
for some rank $1$ sheaf $\mcL$. If the element $\xi$ exists, it is
called the special involution of $\mcF$.  It is {\em exactly} of order $2$ and in the
Borel case, is of the shape
$$
\xi_\mcF=\begin{pmatrix} -1&b_\mcF\\&1
\end{pmatrix}.
$$
\end{itemize}
\end{definition}
\begin{remark} We take this occasion to address a minor slip in
  \cite{FKMSP} pointed by one of the referees: the original definition
  of a bountiful sheaf should have required the rank of the sheaf to
  be $\geq 3$ in the $\SL$ case, since $\SL_2$ should be viewed as a
  symplectic group in this context (because its standard
  representation is self-dual). Correspondingly \cite[Thm 1.5]{FKMSP}
  should include this condition as well. This has no impact on
  applications since the resulting corollaries all included that
  condition in their statement.
\end{remark}

\begin{remark} Another difference with \cite{FKMSP} is that we allow
  the possibility that $\mcF$ be the extension by zero of $\mcG$, and
  do not require that $\mcF$ be necessarily a middle-extension. It is
  immediate that the results of~\cite{FKMSP} that we use extend to
  this slightly more general class of sheaves: the arguments there are
  either performed on a dense open subset where all sheaves involved
  are lisse, or only depend on the bound $|t_{\mcG}(x)|\leq \rank(x)$
  for a middle-extension sheaf $\mcG$ (see, e.g.,~\cite[p. 21, proof
  of Prop. 1.1]{FKMSP}). We refer to Remark~\ref{rm-kl-def} for a
  justification of this change in the definition of~\cite{FKMSP}.
\end{remark}

The Kloosterman sheaves $\KL_k$ (defined here as extension by zero of
the Kloosterman sheaves on $\Gg_m$) are examples of bountiful sheaves.
They are of $\Sp$-type if $k$ is even and of $\SL$-type if $k$ is odd
(cf.~\cite{GKM,FKMSP}), and in that case, there is a special
involution given by $\xi=\begin{pmatrix}-1&\\&1
\end{pmatrix}$, and indeed $\xi^*\KL_k\simeq \KL_k^{\vee}$. All this
will be recalled with references in Section~\ref{sec-kl-sheaves}.

\subsection*{Acknowledgments}

We acknowledge the deep influence of \'E. Fouvry on this work.  The
ideas of our collaborators concerning the problem of averages of
twisted $L$-functions in~\cite{BFKMM} (\'E. Fouvry, V. Blomer and
D. Mili{\'c}evi{\'c}) were also of great importance in motivating our
work on this paper. We also thank P. Nelson and I. Petrow for many
discussions.

We are extremely thankful to the referee who read
Sections~\ref{sec-complete} and~\ref{sec-irreducible} and pointed out
many minor slips and a few more significant issues in the first and second version of this paper.  We thank him or her in particular for giving very
useful references to certain papers of L. Fu that corrected and
simplified some of our local monodromy computations.


\section{Reduction to complete exponential sums}\label{sec-incomplete}

In this section, we perform the first step of the proof of Theorems
\ref{thmtypeI} and \ref{thmtypeII}: the reduction to estimates for
complete sums over finite fields. The two subsections below are
essentially independent; the first one concerns special bilinear forms
(``type I'', as in Theorem~\ref{thmtypeI}) and the second discusses
the case of general bilinear forms (``type II'') as in
Theorem~\ref{thmtypeII}. 

\subsection{Special bilinear forms}

We follow the method of \cite{FoMi}, as generalized in \cite[\S
6.2]{BFKMM}. Let $q$ be a prime number and let $\mcF$ be a bountiful
sheaf on $\Aa^1_{\Fq}$ (with respect to the Borel subgroup). Let
$k\geq 2$ be the rank of $\mcF$ and $\cond(\mcF)$ its conductor.
\par
We fix some $c\in\Fqt$, and denote $K_c=[\times c]^*K$.  We consider
the special bilinear form
$$
B(K_c,\uple{\alpha},\mcN)= \sumsum_{m\leq M,\,n\in\mcN}\alpha_m K(cmn)
$$
where $\mcN$ is an interval in $[1,q-1]$ of length $\lfloor N\rfloor $
and $\uple{\alpha}=(\alpha_m)_{m\leq M}$ with
\begin{equation}\label{MNcond}
  1\leq M\leq N^2,\ N<q,\ MN<q^{3/2}.
\end{equation}

\begin{remark}
  The condition $MN<q^{3/2}$ is somewhat restrictive. It arises from
  the estimate of the possible ``bad'' parameter $\uple{b}$ (see the
  proof of Theorem~\ref{thm1Klk} below). However, for
  $MN\geq q^{3/2}$, other methods lead to non-trivial estimates for
  these bilinear forms (e.g., the bound~(\ref{PVII})).
\end{remark}

Given auxiliary integral parameters $A,B\geq 1$ such that
\begin{equation}\label{ABbounds}
  2B<q,\quad AB\leq N,\quad AM<q, 
\end{equation}
we have
\begin{align*}
  B(K_c,\uple{\alpha},\mcN)
  &=\frac{1}{AB}
    \sumsum_{\substack{A < a \leq 2A\\ B < b \leq 2B}}
  \sum_{m\leq M}\alpha_m\sum_{n+ab\in\mcN} K_c(m(n+ab))\\
  &=\frac{1}{AB}\sumsum_{\substack{A < a \leq 2A\\ B < b \leq 2B}}
  \sum_{m\leq M}\alpha_m\sum_{n+ab\in\mcN} K_c(am(\ov an+b)).
\end{align*}
We get
$$
B(K_c,\uple{\alpha},\mcN)\ll_{\eps}
\frac{q^\eps}{AB}\sumsum_{\substack{r\mods q\\s\leq
    2AM}}\nu(r,s)\Bigl|\sum_{B < b \leq 2B }\eta_bK_c(s(r+b))\Bigr|
$$ 
where
$$
\nu(r,s)=\sumsumsum_\stacksum{A < a \leq 2A,\ m\leq M,\
  n\in\mcN}{am=s,\ \ov an\equiv r\mods q}|\alpha_m|
$$
and $(\eta_b)_{B < b \leq 2B}$ are some complex numbers such that
$|\eta_b|\leq 1$.  We have clearly
$$
\sum_{r,s}\nu(r,s)\ll AN\sum_{m\leq M}|\alpha_m|.
$$
\par
We also have
$$
\sum_{r,s}\nu(r,s)^2=\multsum_{\substack{a,m,n,a',m',n'\\
    am=a'm'\\a'n=an'\mods q}} |\alpha_m||\alpha_{m'}|.
$$
\par
Observe that, once $a$ and $m$ are given, the equation $am=a'm'$
determines $a'$ and $m'$ up to $O(q^\eps)$ possibilities; furthermore,
for each such pair $(a,m)$ and each $n\in\mcN$, the congruence
$a'n=an'\mods q$ determines $n'$ uniquely, as $n'$ varies over an
interval of length $\leq q$. Therefore we get
$$
\sum_{r,s}\nu(r,s)^2\ll \sum_{a,m}|\alpha_m|^2
\multsum_{\substack{n,a',m',n'\\am=a'm'\\a'n=an'\mods q}}1 \ll_\eps
q^\eps AN\sum_{m}|\alpha_m|^2,
$$
where we have used the inequality
$|\alpha_m||\alpha_{m'}|\leq |\alpha_m|^2+|\alpha_{m'}|^2$.

We next apply H\"older's inequality in the form
\begin{align*}
  \sumsum_{\substack{r\mods q\\1\leq s\leq 2AM}} \nu(r,s)
  \Bigl|\sum_{B < b \leq 2B}\eta_bK_c(s(r+b))\Bigr|
&
  \leq
  \Bigl(\sum_{r,s}\nu(r,s)\Bigr)^{\frac{1}{2}}
  \Bigl(\sum_{r,s}\nu(r,s)^2\Bigr)^{\frac1{4}}
  \\
&\quad\quad\quad\quad\times \Bigl(\sum_{r,s}\Bigl|\sum_{B < b \leq 2B
  }\eta_bK_c(s(r+b))\Bigr|^{4}\Bigr)^{\frac{1}{4}}
  \\
& \ll_{\eps} q^\eps
  (AN)^{\frac{3}{4}}\|\uple{\alpha}\|_1^{\frac{1}{2}}\|\uple{\alpha}\|_2^{\frac{1}{2}}
  \Bigl(\sum_{r,s}\Bigl|\sum_{B < b \leq 2B }\eta_b
  K_c(s(r+b))\Bigr|^{4}\Bigr)^{\frac{1}{4}}.
\end{align*}

Expanding the fourth power, we have
\begin{equation}\label{bSigmasum} 
  \sum_{r,s}\Bigl|\sum_{B < b \leq 2B }\eta_b
  K_c(s(r+b))\Bigr|^{4}
  \leq \sum_{\bfb\in\mcB}\bigl|\Sigma(K_c, \bfb;AM)\bigr|
\end{equation}
where $\mcB$ denotes the set of tuples $\bfb=(b_1,b_2,b_3,b_4)$
of integers satisfying $B < b_i \leq 2B$ ($i=1,\cdots,4$), and
\begin{equation}\label{eq-def-sigma}
  \Sigma(K_c,\bfb;AM)=
  \sumsum_\stacksum{r\mods
    q}{1\leq s\leq 2AM}\prod_{i=1}^2 K_c(s(r+b_i))\ov{K_c(s(r+b_{i+2}))}.
\end{equation}
This is a sum over $r$ and $s$ of a product of four values of the
trace function $K$, which we will later specialize to
hyper-Kloosterman sums. At this stage, we have proved the bound
\begin{equation}\label{eq-intermediate}
  B(K_c,\uple{\alpha},\mathcal{N})
  \ll q^{\eps}
  \frac{N^{3/4}}{A^{1/4}B}\|\uple{\alpha}\|_1^{1/2}\|\bfb\|_2^{1/2}
  \Bigl(\sum_{\bfb\in\mcB}\bigl|\Sigma(K_c, \bfb;AM)\bigr|\Bigr)^{1/4}
\end{equation}
for any $\eps>0$, where the implied constant depends on $\eps$ and on
the conductor of $\sheaf{F}$.
\par
To continue, we first define the ``diagonal'' in the space of the
parameters $\bfb\in\mcB$. We recall that sheaves of $\Sp$-type or of
$\SL$-type were introduced in Definition~\ref{def-type}.

\begin{definition}\label{multidiagonaldef} 
  Let $\mcV^\Delta$ be the affine variety of $4$-uples
$$
\bfb=(b_1,b_2,b_3,b_4)\in\Aa_\Fq^{4}
$$
defined by the following conditions:
\begin{itemize}
\item[--] if $\mcF$ is of $\Sp$-type, then for any $i\in\{1,\cdots,
  4\}$, the cardinality
$$
|\{j=1,\ldots, 4\,\mid\, b_j=b_i\}|
$$
is even.
\item[--] if $\mcF$ is of $\SL$-type, then for any $i\in\{1,2\}$, we
  have
$$
|\{j=1,2\,\mid\, b_j=b_i\}|-|\{j=3,4\,\mid\, b_j=b_i\}|=0.
$$
\end{itemize}
\end{definition}

We now denote by $\mcB^{\Delta}$ the subset of tuples of integers
$\bfb\in\mcB$ such that
$$
\bfb\mods q\in\mcV^\Delta(\Fq).
$$

Since $k\geq 2$ and $2B<q$ (by~(\ref{ABbounds})), we have
$|\mcB^{\Delta}|=O(B^2)$.  For $\bfb\in\mcB^{\Delta}$, we estimate
$\Sigma(K_c, \bfb; AM)$ trivially using the bound
$|K(cx)|\leq \cond(\mcF)$. The contribution to \refs{bSigmasum} of all
$\bfb\in\mcB^\Delta$ satisfies
\begin{equation}\label{bDelta}
  \sum_{\bfb\in\mcB^{\Delta}}|\Sigma(K_c, \bfb;AM)|\ll AB^2M q,
\end{equation}
where the implied constant depends only on the conductor of $\mcF$.

In Section~\ref{sec-complete}, we will establish two estimates
concerning the contribution of $\bfb\not\in \mcB^{\Delta}$.  For the
first argument, we fix the value of $s$ with $1\leq s\leq 2AM$ and we
average  over $r$.

\begin{lemma}\label{lemonevariableI} 
  For $\bfb\in \mcB\backslash\mcB^\Delta$ and any  $s\in\Fqt$, we have
$$
\sum_{r\mods q}\prod_{i=1}^2
K_c(s(r+b_i))\ov{K_c(s(r+b_{i+2}))}\ll_{k} q^{1/2}
$$
where the implied constant depends only on $\cond(\mcF)$.
\par
In particular 
for any subset $\mcB'\subset \mcB\backslash\mcB^\Delta$, we have
\begin{equation}\label{onevariableI}
  \sum_{\bfb\in\mcB'}|\Sigma(K_c, \bfb; AM)|\ll_k AM|\mcB'|q^{1/2}
\end{equation}
where the implied constant depends only on $\cond(\mcF)$.
\end{lemma}

This result gives a saving of a factor $q^{1/2}$ over the trivial
bound.  We refer to Section~\ref{sec-one-variable} for the proof.

The second argument is much deeper, and we can only bring it to
completion for hyper-Kloosterman sums. We apply the discrete
Plancherel formula to complete the sum with respect to the variable
$s$ (see for instance~\cite[Lemma 12.1 and following]{IwKo}). Recall
that $\psi$ is a fixed non-trivial additive character of $\Fq$.  For
any function $L\colon \Fq\to \Cc$, define
$$
\hat{\Sigma}(L,\bfb,\lambda)=\sum_{r\in\Fq}\bfR(L,r,\lambda,\bfb)
$$
with
\begin{equation}\label{defSrblambdasum}
  \bfR(L,r,\lambda,\bfb)=\sum_{s\in\Fqt}
  \psi(\lambda s)\prod_{i=1}^2 L(s(r+b_i))\ov{L(s(r+b_{i+2}))}.
\end{equation}
Then, observing that for any $c\in\Fqt$, we have
\begin{equation}\label{eq-rem-c}
\bfR(K_c,r,\lambda,\bfb)=\bfR(K,r,\lambda/c,\bfb),\quad\quad
\hat{\Sigma}(K_c,\bfb,\lambda)=\hat{\Sigma}(K,\bfb,\lambda/c),
\end{equation}
the completion yields the bound
$$
\Sigma(K_c,\bfb;AM)\ll (\log q)
\max_{\lambda\in\Fq}|\hat{\Sigma}(K,\bfb,\lambda)|
$$
where the implied constant is absolute.
 

Taking $\mcF$ to be the Kloosterman sheaf with trace function
$K=\Kl_k$, we will obtain, an additional saving of $q^{1/2}$ in
comparison with Lemma~\ref{lemonevariableI}, from the cancellation in
the completed variable $s$, leading to a net saving $AMq^{1/2}$.

\begin{theorem}\label{thm1Klk} 
  Let $k\geq 2$ and let $K=\Kl_k$. There exists a codimension one
  subvariety $\mcV^{bad}\subset\Aa^4_\Fq$ containing $\mcV^{\Delta}$,
  with degree bounded independently of $q$, such that for any
  $\lambda\in\Fq$ and any $\bfb\not\in\mcV^{bad}(\Fq)$, we have
$$
\Sigma([\times c]^*\Kl_{k},\bfb;AM)\ll q\log q
$$
for any $c\in\Fqt$.  The implied constant depends only on $k$.
\end{theorem}

This follows from Theorem \ref{thmRbounds} in
Section~\ref{sec-irreducible}.
\par
\smallskip
\par
Now, assuming Lemma~\ref{lemonevariableI} and Theorem~\ref{thm1Klk},
we can conclude the proof of Theorem~\ref{thmtypeI}.  Indeed, set
$$
\mcB^{bad}=\mcB\cap\{\bfb\in\mcB\,\mid\,
\bfb\mods{q}\in\mcV^{bad}(\Fq)\},\quad\quad
\mcB^{gen}=\mcB\backslash\mcB^{bad}.
$$
Since $\mcV^{bad}$ has degree bounded in terms of $k$ only,
independently of $q$, we have $|\mcB^{bad}|=O_{k}(B^{3})$ (in fact,
$|\mcB^{bad}|\leq (\deg\mcV^{bad})|B|^{3}$ by the so-called
Schwarz-Zippel Lemma).
\par
Hence, applying Theorem \ref{thm1Klk} for $\bfb\in\mcB^{gen}$, the
bound \refs{onevariableI} from Lemma~\ref{lemonevariableI} for
$\bfb\in\mcB^{bad}-\mcB^{\Delta}$, and finally~\eqref{bDelta} for
$\bfb\in\mcB^{\Delta}$, we obtain
$$
\sum_{\bfb\in\mcB}\bigl|\Sigma([\times c]^*\Kl_k, \bfb;AM)\bigr|\ll_k (B^{4}q+
AB^{3}Mq^{1/2}+AB^2M q)(\log q).
$$
\par
Upon choosing
$$
A=M^{-\frac1{3}}N^{\frac{2}{3}},\quad B=(MN)^{\frac1{3}},
$$
(which satisfy \eqref{ABbounds} by \eqref{MNcond}), we see that the
first and third terms in parenthesis coincide and are equal to
$(MN)^{4/3}q$, while the second term is equal to
$$
(MN)^{4/3}q\times(MNq^{-3/2})^{1/3}\leq (MN)^{4/3}q
$$ 
by \eqref{MNcond}.  Therefore we deduce from~(\ref{eq-intermediate})
that
\begin{align*}
  B([\times c]^*\Kl_k,\uple{\alpha},\mcN)
  &\ll_{k,\eps} 
    \frac{q^\eps}{AB}(AN)^{\frac{3}{4}}
    \|\uple{\alpha}\|_1^{\frac{1}{2}}\|\uple{\alpha}\|_2^{\frac1{2}}Bq^{1/4}\\
  &\ll_{k,\eps}q^\eps\|\uple{\alpha}\|_1^{\frac{1}{2}}
    \|\uple{\alpha}\|_2^{\frac1{2}}M^{1/4}N\Bigl(\frac{M^2N^5}{q^3}\Bigr)^{-1/12}.
\end{align*}
\par
This proves Theorem \ref{thmtypeI}, subject to the proof of
Lemma~\ref{lemonevariableI} and of Theorem~\ref{thm1Klk}.

\subsection{General bilinear forms} 

We now consider the situation of Theorem~\ref{thmtypeII}. Again we
begin with a prime $q$ and a bountiful sheaf $\mcF$ on $\Aa^1_{\Fq}$
with respect to the Borel subgroup.  Let $k\geq 2$ be the rank of
$\mcF$ and $\cond(\mcF)$ its conductor.
\par
Given $M,N\geq 1$ satisfying
\begin{equation}\label{MNcondII}
1\leq M\leq Nq^{1/4},\ q^{1/4}<MN< q^{5/4},
\end{equation}
an interval $\mcN\subset[1,q-1]$ of length $\lfloor N\rfloor$ and
sequences $\uple{\alpha}=(\alpha_m)_{m\leq M}$ and
$\uple{\beta}=(\beta_n)_{n\in \mcN}$, we consider the general bilinear
form
$$
B(K_c,\uple{\alpha},\uple{\beta})=\sumsum_{m\leq M,\,n\in
  \mcN}\alpha_m\beta_n K(cmn).
$$
\par
We begin once more as in \cite{FoMi,BFKMM}. We choose auxiliary
parameters $A,B\geq 1$ satisfying \eqref{ABbounds}. The argument of
\cite[\S 5.5]{BFKMM} leads
to the estimate
\begin{equation}\label{eqBKholder}
  |B(K_c,\uple{\alpha},\uple{\beta})|^2\ll 
  \|\uple{\alpha}\|_2^2\|\uple{\beta}\|_2^2
  \Bigl(N+\frac{q^{\eps}}{AB}(AN)^{3/4}M^{1/2}
  \Bigl(\sum_\bfb|\Sigma^{\not=}(K_c,\bfb;AM)|\Bigr)^{1/4}\Bigr)	
\end{equation}
for any $\eps>0$, where the implied constant depends only on
$\cond(\mcF)$ and $\eps$, and where
\begin{multline*}
\Sigma^{\not=}(K_c,\bfb;AM)=\sum_{r\mods q}\sumsum_\stacksum{1\leq
  s_1,s_2\leq AM}{s_1\not = s_2\mods
  q}\prod_{i=1}^2K_c(s_1(r+b_i))\ov
K_c(s_2(r+b_i))
\\
\ov{K_c(s_1(r+b_{i+2}))\ov K_c(s_2(r+b_{i+2}))}
\end{multline*}
for $\bfb$ running over the set $\mcB$ of quadruples of integers
$(b_1,b_2,b_3,b_4)$ satisfying $B < b_i \leq 2 B$. Note that, in the
case $K=\hypk_k$, we have now a sum, over the three variables
$(r,s_1,s_2)$, of a product of eight hyper-Kloosterman sums.

We will estimate the inner triple sum over $r, s_1, s_2$ in different
ways depending on the value taken by $\bfb$.

First, for $\bfb\in\mcB^\Delta$ (as defined in Definition
\ref{multidiagonaldef}) we use the trivial bound from
$|K(cx)|\leq \cond(\sheaf{F})$ and obtain
\begin{equation}\label{diagonalbound}
  \sum_{\bfb\in \mcB^{\Delta}}|\Sigma^{\not=}(K_c,\bfb;AM)|\ll qA^2B^2M^2,
\end{equation}
where the implied constant depends only on $\cond(\mcF)$.

We next have an analogue of Lemma \ref{lemonevariableI}, where we sum
over the variable $r$ for fixed $(s_1,s_2)$:

\begin{lemma}\label{lemonevariableII} 
  For $\bfb\in \mcB\backslash\mcB^\Delta$ and any $s_1$, $s_2\in\Fqt$ with
  $s_1\not=s_2$, we have
  \begin{equation}\label{onevariableII}
    \sum_{r\mods q}\prod_{i=1}^2K_c(s_1(r+b_i))\ov K_c(s_2(r+b_i))
    \ov{K_c(s_1(r+b_{i+2}))\ov K_c(s_2(r+b_{i+2}))}\ll q^{1/2},
  \end{equation}
  where the implied constant depends only on $\cond(\mcF)$.
\par
In
particular 
for any subset $\mcB'\subset\mcB\backslash\mcB^\Delta$, we have
$$
\sum_{\bfb\in\mcB'}|\Sigma^{\not=}(K_c,\bfb;AM)| \ll
(AM)^2|\mcB'|q^{1/2},
$$
where the implied constant depends only on $\cond(\mcF)$.
\end{lemma}

This is proved in Section~\ref{sec-one-variable}.
\par
\medskip
\par
Finally, we use discrete Fourier analysis. We detect the condition
$s_1\not\equiv s_2\mods q$ using additive characters:
$$
1-\frac1q\sum_{\lambda\mods q}
e_q(\lambda(s_1-s_2))=
\begin{cases}
1&\text{ if } s_1\not=s_2\text{ in }\Fq,\\
0&\text{ otherwise.}
\end{cases}
$$
\par
We further complete the sums over $s_1$ and $s_2$ using additive
characters. 
For any $L\colon \Fq\to\Cc$, we define
$$
\mcC(L,\lambda_1,\lambda_2,\bfb)=\sum_{r\mods
  q}\bfR(L,r,\lambda_1,\bfb)\ov{\bfR(L,r,\lambda_2,\bfb)}
$$
where $\bfR(L,r,\lambda,\bfb)$ is the sum defined in
\eqref{defSrblambdasum}. Then let
$$
\hat{\Sigma}(L,\bfb,\lambda_1,\lambda_2,)=
\mcC(L,\lambda_1,\lambda_2,\bfb)-\frac{1}{q} \sum_{\lambda\mods q}
\mcC(L,\lambda_1+\lambda,\lambda_2+\lambda,\bfb).
$$
Observing, as in~(\ref{eq-rem-c}), that for $c\in\Fqt$ we have
$$
\hat{\Sigma}(K_c,\bfb,\lambda_1,\lambda_2)=\hat{\Sigma}(K,\lambda_1/c,\lambda_2/c,\bfb),
$$
the completion leads to the bound
$$
\Sigma^{\not=}(K_c,\bfb;AM)\ll (\log q)^2
\max_{\lambda_1,\lambda_2\in\Fq}|\hat{\Sigma}(K,\bfb,\lambda_1,\lambda_2)|
$$
for any $c\in\Fqt$, where the implied constant is absolute.


We must now assume  as before that $\mcF=\KL_k$ is the Kloosterman
sheaf of rank $k$ with trace function $K=\Kl_k$. We will prove below
our final bound:

\begin{theorem}\label{thm2Kl2} 
  Let $k\geq 2$ and let $K=\Kl_k$. There exists a codimension one
  subvariety $\mcV^{bad}\subset\Aa^4_\Fq$ containing $\mcV^{\Delta}$,
  with degree bounded independently of $q$, such that for any
  $\bfb\not\in\mcV^{bad}(\Fq)$ and every distinct $\lambda_1$,
  $\lambda_2\in\Fq$, we have
\begin{equation}\label{SigmaIIbound}
  |\hat{\Sigma}(\Kl_k,\bfb,\lambda_1,\lambda_2)|\ll q^{3/2}	
\end{equation}
where the constant depends only on $k$.
\end{theorem}

This follows from Theorem \ref{thmRbounds} in
Section~\ref{sec-irreducible}. In fact, the subvariety $\mcV^{bad}$ is
the same as in Theorem~\ref{thm1Klk}.

Assuming these results, we conclude the proof of
Theorem~\ref{thmtypeII} in the same manner as in the previous
section. For
$$
\mcB^{bad}=\mcB\cap\{\bfb\in\mcB\,\mid\,
\bfb\mods{q}\in\mcV^{bad}(\Fq)\},\quad\quad
\mcB^{gen}=\mcB\backslash\mcB^{bad},
$$
we have the estimate $|\mcB^{bad}|=O_{k}(B^{3})$ since $\mcV^{bad}$
has degree bounded independently of $q$.
\par
We apply Theorem \ref{thm2Kl2} for $\bfb\in\mcB^{gen}$, the bound
\refs{onevariableII} of Lemma~\ref{lemonevariableII} for
$\bfb\in\mcB^{bad}-\mcB^{\Delta}$ and finally \eqref{diagonalbound}
for  $\bfb\in\mcB^{\Delta}$. This gives
$$
\sum_\bfb|\Sigma^{\not=}([\times c]^*\Kl_k,\bfb;AM)|\ll (\log
q)^2(B^4q^{3/2}+A^2B^3M^2q^{1/2}+A^2B^2M^2q),
$$
where the implied constant depends only on $k$.
\par
We select
$$A=q^{\frac18}M^{-\frac12}N^{\frac12},\
B=q^{-\frac18}M^{\frac12}N^{\frac12},\ $$ which satisfy
\eqref{ABbounds} by \eqref{MNcondII}. Then $AB=N$ and the first and
third terms on the right-hand side are equal to $(MN)^2q$. The second
term is $(MN)^{\frac{5}2}q^{\frac38}\leq (MN)^2q$ by \eqref{MNcondII}.
Therefore we have
$$
\sum_\bfb|\Sigma^{\not=}([\times c]^*\Kl_k,\bfb;AM)|\ll (MN)^2q(\log q)^2
$$
and consequently we obtain from~(\ref{eqBKholder}) the bound
\begin{align*}
  |B([\times c]^*\Kl_k,\uple{\alpha},\uple{\beta})|^2
  &\ll
    \|\uple{\alpha}\|_2^2\|\uple{\beta}\|_2^2
    \Bigl(N+\frac{q^{\eps}}{N}(AN)^{3/4}M^{1/2}q^{1/4}(MN)^{1/2}\Bigr)\\
  &\ll
    q^\eps\|\uple{\alpha}\|_2^2\|\uple{\beta}\|_2^2\Bigl(N+(MN)^{\frac{5}{8}}q^{\frac{11}{32}}
    \Bigr)\\
  &\ll
    q^\eps\|\uple{\alpha}\|_2^2\|\uple{\beta}\|_2^2MN\Bigl(M^{-1}+(MN)^{-\frac{3}{8}}q^{\frac{11}{32}}
    \Bigr),
\end{align*}
for any $\eps>0$, where the implied constant depends only on $k$ and
$\eps$.
\par
This concludes the proof of Theorem \ref{thmtypeII} modulo the proof
of Lemma~\ref{lemonevariableII} and of Theorem \ref{thm2Kl2}.

\begin{remark} \label{rmkimprove}
  As in \cite{FoMi} it is possible to apply the H\"older inequality
  that leads to \eqref{eqBKholder} with higher exponent than
  $2l=4$. Doing this leads to sums involving products of the shape
$$
(r,s_1,s_2)\mapsto\prod_{i=1}^{l}K(s_1(r+b_i))\ov
K(s_2(r+b_i))\ov{K(s_1(r+b_{i+l}))\ov K(s_2(r+b_{i+l}))}
$$ 
for
$$
\bfb_l=(b_1,\cdots,b_l,b_{l+1},\cdots,b_{2l})\in ]B,2B]^{2l}.
$$
\par
Except for heavier notational complexity, some of the arguments of
this section (and of the next) do carry over and (assuming that
\eqref{MNcondII} holds), one obtains for $l\geq 3$ and
$MN\geq q^{7/8}$ the bound
$$
|B(\Kl_k,\uple{\alpha},\uple{\beta})|^2\ll_{\eps,k}
q^\eps\|\uple{\alpha}\|_2^2\|\uple{\beta}\|_2^2MN\Bigl(M^{-1}+
(q^{l+4}(MN)^{-8})^{\frac{1}{4l(l+2)}}\Bigr).
$$
\par
This bound is only interesting  when $l=3$ and yields a
non-trivial estimate in the range
$$
MN\geq q^{\frac{7}8+\delta},\quad\quad \delta>0
$$ 
compared  with $MN\geq q^{\frac{11}{12}+\delta}$ in Remark
\ref{remnontrivialII}. 
\par
In order for the H\"older inequality with higher exponents to give
better estimates, one needs to improve the lower bound on the
codimension of the variety $\mcV_l^{bad}\subset \Aa_\Fq^{2l}$ in the
corresponding generalization of Theorem~\ref{thm2Kl2}. At the moment,
we only know that this codimension is at least $1$, but if one could
prove that this variety has codimension $2$, one could take $l=5$ and
obtain a non-trivial bound in the range $MN\geq q^{\frac56+\delta}$.
\par
The best possible result which might be achieved using this method
would be if $\mcV_l^{bad}$ had codimension $l$. This would lead to
non-trivial bounds as long as
$$
MN\geq q^{\frac{3l+5}{4(l+1)}+\delta},\quad\quad \delta>0.
$$
Although we expect that the codimension of $\mcV_l^{bad}$ is indeed
$l$, this seems a difficult geometric problem when $l$ is large
(indeed, the method used in Section~\ref{sec-irreducible} do not seem
to be sufficient, as they amount to ``estimating'' $\mcV_{l}$ by
showing that it is a subvariety of another variety whose codimension
we estimate, and one expects that the codimension of this auxiliary
variety is exactly $1$).
\par
By taking $l$ very large, we thus see that the limit of the method is
the range $MN\geq q^{\frac34+\delta}$. Interestingly, this is the same
range achieved in~\cite[Th. 1.6]{FKM2} for the case where
$\uple{\alpha}$ and $\uple{\beta}$ are both smooth.

Some of the claims made in this remark have now been verified by D. Bejleri, A. Christensen, B. Kadets, C.-Y. Hsu and Z. Yao while pursuing a research project during the 2016 Arizona Winter School.
\end{remark}

\section{Bounds for complete exponential sums}
\label{sec-complete}

In this section we use methods from $\ell$-adic cohomology to prove
Lemmas \ref{lemonevariableI} and \ref{lemonevariableII}, and we make
the first steps towards Theorems~\ref{thm1Klk} and \ref{thm2Kl2}.  The
proof of these last two theorems will be finished in
Section~\ref{sec-irreducible}.
\par
All results in this section apply for bountiful sheaves (with respect
to the Borel subgroup). Thus we fix a prime $q$ and such a sheaf
$\mcF$ on $\Aa^1_{\Fq}$. We denote by $K$ the trace function of
$\mcF$, and by $\Sing({\mcF})\subset \Pp^1(\bar{\Ff}_q)$ the set of
ramification points of $\mcF$.
\par
For any finite extension $\Fqd/\Fq$, for $\bfb\in (\Fqd)^4$,
$\lambda\in\Fqd$, and $(r,s)\in \Fqd\times\Fqd$, we denote
$$
\bfK(r,s,\lambda,\bfb;\Fqd)=\psi_{\Fqd}(\lambda s) \prod_{i=1}^2 K( s
(r+b_i );\Fqd) \overline{K(s (r+b_{i+2});\Fqd)}.
$$
\par
For $d=1$, we write simply
$\bfK(r,s,\lambda,\bfb)=\bfK(r,s,\lambda,\bfb;\Fq)$. 

\subsection{One variable bounds}\label{sec-one-variable}

The next proposition is a restatement of Lemma~\ref{lemonevariableI}
and \ref{lemonevariableII}, where the variable $c\in\Fqt$ is taken to
be equal to $1$; since we may replace $s$ by $cs$ (resp. $(s_1,s_2)$
by $(cs_1,cs_2)$), this implies the case where $c\in\Fqt$ is
arbitrary.

\begin{proposition}\label{weil} 
  Assume $q\not=2$.  Let $\mcV^\Delta\subset \Aa_\Fq^{4}$ be the
  affine variety given in Definition \ref{multidiagonaldef}.
\par
For all $\bfb=(b_1,b_2,b_3,b_4)\in \Fq^4-\mcV^\Delta(\Fq)$ and for all
$s$, $s_1$, $s_2\in \Fqt$, with $s_1\not=s_2$, we have
\begin{gather}\label{weil1}
\sum_{r\in\Fq}\bfK(r,s,0,\bfb)\ll q^{1/2},\\
\label{weil2}
\sum_{r\in\Fq}\bfK(r,s_1,0,\bfb)\ov{\bfK(r,s_2,0,\bfb)}\ll q^{1/2}
\end{gather}
where the constant implied depend only on the conductor of $\mcF$.
\end{proposition}

\begin{proof}
  This follows from the techniques surveyed in~\cite{FKMSP}.
  Precisely, for fixed $s\in\Fqt$ and $\bfb\notin\mcV^{\Delta}(\Fq)$,
  the sum in~(\ref{weil1}) is of the type discussed
  in~\cite[Cor. 1.6]{FKMSP} with $k=4$, $h=0$, the $4$-tuple
$$ 
\uple{\gamma}=(\gamma_{s,1},\cdots,\gamma_{s,4})\in\PGL_2(\Fq)^{4}
$$
such that
$$
\gamma_{s,i}=\begin{pmatrix} s &sb_i\\ &1
\end{pmatrix},\ i=1,\ldots,4.
$$
and (if $\mcF$ is of $\SL$-type) the $4$-tuple
$$
\uple{\sigma}=(\sigma_i)_{i=1,\cdots, 4}\in \mathrm{Aut}(\Cc/\Rr)^{4}
$$
where
$$
\sigma_1=\sigma_2=\mathrm{Id_{\Cc}},\quad \sigma_{3}=\sigma_4=c,\quad
c=\text{complex conjugation}.
$$
\par
If $\mcF$ is of $\Sp$ type, the fact that $\bfb$ is not contained in
$\mcV^\Delta(\Fq)$ implies that the tuple $\uple{\gamma}$ is
\emph{normal} in the sense of \cite[Definition 1.3]{FKMSP}. 
\par
Similarly, if $\mcF$ is of $\SL$-type with $\rank(\mcF)=r\geq 3$, and
$\bfb\not\in\mcV^\Delta(\Fq)$, the pair of tuples $(\gamma,\sigma)$ is
$r$-\emph{normal}, including with respect to the special involution
$\xi_\mcF$ of $\mcF$, if the latter exists. Indeed, because $q\not=2$,
$\gamma_{s,i}\gamma^{-1}_{s,j}$ is not an involution unless $b_i=b_j$, so
can only be equal to $\xi_\mcF$ if $b_i=b_j$. This means that conditions
(2) and (3) of \cite[Def. 1.3]{FKMSP} are equivalent in our situation.
Thus the bound \eqref{weil1} follows from \cite[Cor. 1.6]{FKMSP}.
\par
We now consider the bound~(\ref{weil2}). We are again in the situation
of~\cite[Cor. 1.6]{FKMSP} with $h=0$, $k=8$, 
the $8$-tuple
$$
\uple{\gamma}=(\gamma_{s_1,1},\ldots,\gamma_{s_1,4},
\gamma_{s_2,1},\ldots,\gamma_{s_2,4})
$$ 
and (in the $\SL$-type case) the $8$-tuple
$$
\uple{\sigma}= (\mathrm{Id}_{\Cc},
\mathrm{Id}_{\Cc},c,c,c,c,\mathrm{Id}_{\Cc},\mathrm{Id}_{\Cc}).
$$
\par
For $s_1\not=s_2$ and $\bfb\not\in\mcV^\Delta(\Fq)$, the $8$-tuple
$\uple{\gamma}$ is normal for $\mcF$ of $\Sp_k$-type while for $\mcF$
of $\SL_r$-type with $r\geq 3$, the tuples
$(\uple{\gamma},\uple{\sigma})$ are again $r$-normal (also possibly
with respect to the special involution $\xi_\mcF$, if it
exists). Indeed, the fact that $s_1\not=s_2$ implies that the
multiplicities involved in checking~\cite[Def. 1.3]{FKMSP} are either
multiplicities from the $4$-tuple associated to $s_1$, or from that
associated to $s_2$, and we are reduced to the situation
in~(\ref{weil1}).  Hence we obtain \eqref{weil2} by
\cite[Cor. 1.6]{FKMSP}.
\end{proof}

By definition, the bound~(\ref{weil1}) gives
Lemma~\ref{lemonevariableI}, and~(\ref{weil2}) gives
Lemma~\ref{lemonevariableII}. 

\begin{remark}
  In the case of hyper-Kloosterman sums ($K=\Kl_k$), the statements we
  use are special cases of the bounds stated in~\cite[Cor. 3.2,
  Cor. 3.3]{FKMSP}.
\end{remark}

\subsection{Second moment computations}

We now consider second moment averages. These estimates will be used
in the next section to prove irreducibility of various sheaves.
\par
For any finite extension $\Fqd/\Fq$, any $\bfb\in (\Fqd)^4$ and
$r\in\Fqd$, we define
\begin{equation}\label{Rdef}
  \bfR(r,\lambda,\bfb;\Fqd)=\sum_{s\in\Fqd}\bfK(r,s,\lambda,\bfb;\Fqd).	
\end{equation}
Note that, as a function of $\lambda$, this is the discrete Fourier
transform of $s\mapsto \bfK(r,s,0,\bfb;\Fqd)$.

\begin{lemma}\label{lm-Rcorrelation}
  Suppose that the $b_i$, $1\leq i\leq 4$, are pairwise distinct in
  $\Fq$. For any $d\geq 1$, we have
\begin{equation}\label{Rcorrelation}
  \frac{1}{(q^d)^2}\sumsum_{r,\lambda\in\Fqd}|\bfR(r,\lambda,\bfb;\Fqd)|^2=
  q^d+O(q^{d/2}),
\end{equation}
where the implied constant depends only on the conductor of $\mcF$.
\par
If $\mcF$ is of $\SL$-type and admits the special involution
\begin{equation}
\label{eq-special-invol}
\xi=
\begin{pmatrix} -1&0\\
  0 &1
\end{pmatrix}
\end{equation}
then we have
\begin{equation}\label{Rnoncorrelation}
  \frac{1}{(q^d)^2}\sum_{r,\lambda\in\Fqd}
  \bfR(r,\lambda,\bfb;\Fqd)\ov{\bfR(r,-\lambda,\bfb;\Fqd)}=O(q^{d/2}),
\end{equation}
where the implied constant depends only on the conductor of $\mcF$.
\end{lemma}


\begin{proof} 
  We abbreviate simply $\psi=\psi_{\Fqd}$ and $K(x)=K(x;\Fqd)$ in the
  computations.  Opening the squares in the lefthand sides of
  \eqref{Rcorrelation} and \eqref{Rnoncorrelation} and averaging over
  $\lambda$, we obtain
\begin{align*}
  q^{-d}\sum_{r,s\in\Fqd}|\bfK(r,s,0,\bfb;\Fqd)|^2&=
  q^{-d}\sum_{r,s\in\Fqd}\prod_{i=1}^4|K(s(r+b_i))|^2\\
&  =q^{-d}\sum_\stacksum{r\in\Fqd}{r+b_i\not=0,\
    i=1,\ldots, 4}\sum_{s\in\Fqd}\prod_{i=1}^4|K(s(r+b_i))|^2+O(1)
\end{align*}
and
\begin{multline*}
  q^{-d}\sum_{r,s\in\Fqd}\bfK(r,s,0,\bfb;\Fqd)\ov{\bfK(r,-s,0,\bfb;\Fqd)}
  \\=q^{-d}\sum_{r,s\in\Fqd}
  \prod_{i=1}^2K(s(r+b_i))\overline{K(s(r+b_{i+2}))} \ov{K(-s(r+b_i))}
  K(-s(r+b_{i+2}))\\
  =q^{-d}\sum_\stacksum{r\in\Fqd}{r+b_i\not=0,\ i=1,\ldots, 4}\sum_{s\in\Fqd}
  \prod_{i=1}^2K(s(r+b_i))\overline{K(s(r+b_{i+2}))} \ov{K(-s(r+b_i))}
  K(-s(r+b_{i+2}))+O(1)
\end{multline*}
respectively, where the implied constant depends only on the conductor
of $\mcF$. 
\par
Since $\xi^*\mcF$ is geometrically isomorphic to the tensor product of
the dual of $\mcF$ with a rank $1$ sheaf $\sheaf{L}$, by assumption,
it follows that $K(-x)=\chi(x)\overline{K(x)}$ for some function
$\chi$ with $|\chi(x)|=1$ for all $x$ such that $\mcF$ is lisse at
$x$. Hence the last sum is equal to
$$
q^{-d}\sum_\stacksum{r\in\Fqd}{r+b_i\not=0,\ i=1,\ldots, 4}\sum_{s\in\Fqd}
L(s)
\prod_{i=1}^2K(s(r+b_i))^2\ov{K(s(r+b_{i+2}))}^2+O(1).
$$
where 
$$
L(s)=\prod_{i=1}^2\overline{\chi(s(r+b_i))} \chi(s(r+b_{i+2}))
$$
is the trace function of a rank $1$ sheaf. Using the relation
$\xi^*\mcF\simeq \mcF^{\vee}\otimes\sheaf{L}$, we see that the
conductor of $\sheaf{L}$ is bounded in terms of the conductor of
$\mcF$ only.
\par
We proceed to evaluate the sum over $s$ using again~\cite{FKMSP} (more
precisely, the final estimates follow from the extension to all finite
fields of these results, which is immediate).
\par
For each $i$, let 
$$
\gamma_{r+b_i}=
\begin{pmatrix}r+b_i&0\\0&1
\end{pmatrix}.
$$
\par
In the $\Sp$-type case, since the $r+b_i$ are pairwise distinct for
$1\leq i\leq 4$, the 
$8$-tuple
$$
\uple{\gamma}=(\gamma_{r+b_1},\ldots,\gamma_{r+b_4},\gamma_{r+b_1},\ldots,
\gamma_{r+b_4})
$$
consists of $4$ distinct pairs $(\gamma,\gamma)$; by~\cite[Cor. 1.7
(1)]{FKMSP}, it follows that for each $r$ distinct from the $-b_i$ for
$1\leq i\leq 4$, we have
$$
\sum_{s\in\Fqd}\prod_{i=1}^4|K((r+b_i)s)|^2= q^d+O(q^{d/2}),
$$
and summing over $r$ gives~(\ref{Rcorrelation}).
\par
In the $\SL$-type case with $r=\rank(\mcF)\geq 3$, the components of
the pair of $8$-tuples
\begin{gather*}
\uple{\gamma}=(\gamma_{r+b_1},\ldots,\gamma_{r+b_4},\gamma_{r+b_1},\ldots,
\gamma_{r+b_4})\\
\uple{\sigma}=(\mathrm{Id}_{\Cc},
\mathrm{Id}_{\Cc},
\mathrm{Id}_{\Cc},
\mathrm{Id}_{\Cc},c,c,c,c)
\end{gather*}
satisfy the final assumption of~\cite[Cor. 1.7 (2)]{FKMSP}, and hence
$$
\sum_{s\in\Fqd}\prod_{i=1}^4|K((r+b_i)s)|^2= q^d+O(q^{d/2}).
$$
also follows if $r+b_i$ is non-zero for each $i$. We therefore
derive~(\ref{Rcorrelation}) again.
\par
Finally we prove \eqref{Rnoncorrelation}: recall we are in the
$\SL$-type with the special involution $\xi$ as
in~(\ref{eq-special-invol}) and with the pair of $8$-tuples
\begin{gather*}
  \uple{\gamma}=(\gamma_{r+b_1},\ldots,\gamma_{r+b_4},\gamma_{r+b_1},\ldots,
  \gamma_{r+b_4})\\
  \uple{\sigma}=(\mathrm{Id}_{\Cc},\mathrm{Id}_{\Cc},c,c,\mathrm{Id}_{\Cc},\mathrm{Id}_{\Cc},c,c).
\end{gather*}
This pair is $r$-normal with respect to $\xi$ (because the multiplicity of any
element in the tuple is either $0$ or $2$). Arguing as in the proof
of~\cite[Th. 1.5]{FKMSP} (p. 20--21, loc. cit.), we deduce that for
each $r$ distinct from the $-b_i$ for $1\leq i\leq 4$, we have
$$
\sum_{s\in\Fqd}L(s) \prod_{i=1}^2K(s(r+b_i))^2\ov{K(s(r+b_{i+2}))}^2\ll q^{d/2},
$$
where the implied constant depends only on the conductor of $\mcF$.
\end{proof}

Finally, we consider one more averaging over the $r$ and $\bfb$
variables in the case  when $\lambda=0$.

\begin{lemma}\label{lm-more-average}
For any $d\geq 1$, we have
$$
\frac{1}{(q^d)^5}
\sumsum_{(r,\bfb)\in\Ff^5_{q^d}}|\bfR(r,0,\bfb;\Fqd)|^2=q^d+O(q^{d/2})
$$
where the implied constant depends only on the conductor of $\mcF$.
\end{lemma}

\begin{proof}
  By a change of variables, we see that the sum is given by
$$
q^{-4d}\sumsum_{b_1,b_2,b_3,b_4}\Bigl| \sum_{s\in\Ff^\times_{q^d}}
\prod_{i=1}^2K(sb_i;\Fqd)\ov{K(sb_{i+2};\Fqd)} \Bigr|^2=
\sum_{s,s'\in\Ff^\times_{q^d}}|\mcC(K,s,s';\Fqd)|^2|\mcC(K,s',s;\Fqd)|^2
$$
where $\mcC(K,s,s';\Fqd)$ denote the correlation sum
$$
\mcC(K,s,s';\Fqd)=q^{-d}\sum_{b\in\Ff_{q^d}}K(sb;\Fqd) \ov
K(s'b;\Fqd)=q^{-d}\sum_{b\in\Ff_{q^d}}K((s/s')b;\Fqd)\ov K(b;\Fqd).
$$
\par
By assumption, the sheaf $\mcF$ is geometrically irreducible and is
such that $[\times (s/s')]^*\mcF$ is geometrically isomorphic to
$\mcF$ if and only if $s=s'$. Therefore by the usual application of
the Riemann Hypothesis (see Proposition~\ref{pr-recall-rh}), we have
$$
\mcC(K,s,s';\Fqd)=\delta(s,s')+O(q^{-d/2}),
$$
where the implied constant depends only on the conductor of
$\mcF$. It follows that
$$
\sum_{s,s'\in\Ff^\times_{q^d}}
|\mcC(K,s,s';\Fqd)|^2|\mcC(K,s',s;\Fqd)|^2=
{q^d}+O({q}^{d-d/2})+O({q}^{2d-4d/2})
=q^d+O({q}^{d/2}),
$$
were the implied constant depends only on the conductor of $\mcF$.
\end{proof}


\section{Irreducibility of sum-product transform
  sheaves}\label{sec-irreducible}

The goal of this long section, which is the most difficult of the
paper, is to prove Theorems~\ref{thm1Klk} and~\ref{thm2Kl2}.  In the
whole section, we fix a prime $q$ and a non-trivial 
additive character $\psi$ of $\Fq$. We fix also an integer $k\geq
2$. We will also assume that $q$ is sufficiently large depending on
$k$. In particular, unless stated otherwise, we always assume that
$$q>k\geq 2.$$
\par
We first begin by outlining the argument. The $7$-variable function
$\bfK$ and its sum $\bfR$ associated to the trace function of a sheaf
$\mcF$ are first interpreted as trace functions of suitable sheaves in
Section~\ref{sec-aux}. The goal is then to prove that various
specializations of these sheaves, which we call \emph{sum-product}
sheaves, are geometrically irreducible. This we can do when $\mcF$ is
a Kloosterman sheaf. To do so requires quite delicate properties of
these sheaves, which are recalled in Section~\ref{sec-kl-sheaves}. It
also requires some relatively general tools which are stated for
convenience in Section~\ref{sec-prelim-geo}. The argument splits in
two parts, depending on whether we specialize with $\lambda=0$ or with
$\lambda\not=0$, and these are handled separately in
Sections~\ref{sec-zero} and~\ref{sec-not-zero}.

\subsection{Sum-product sheaves}\label{sec-aux}

Let $\mcF$ be an $\bQl$-sheaf on $\Aa^1_{\Fq}$, lisse of rank $k$ and
pure of weight $0$ on a dense open subset, and mixed of weight
$\leq 0$ on $\Aa^1$. (Examples of this include the extension by zero
of a lisse and pure sheaf from an open subset or the middle extension
of a lisse and pure sheaf \cite[Corollary 1.8.9]{WeilII}.)
\par
On the affine space $\Aa^7=\Aa^2\times\Aa^1\times\Aa^4$, with
coordinates denoted $(r,s,\lambda,\bfb)$, we define the projection
$p_{2,3}\,:\, \Aa^7\lra \Aa^1$ by
$$
p_{2,3}(r,s,\lambda,b_1,\ldots,b_4)=\lambda s
$$ 
and morphisms $f_i\,:\, \Aa^7\lra \Aa^1$ for $1\leq i\leq 4$ by
\begin{equation}\label{eq-fi}
  f_i(r,s,\lambda,b_1,\ldots,b_4)=s(r+b_i).
\end{equation}

Let $\mcK$ be the $\bQl$-sheaf on $\Aa^7$ defined by
\begin{equation}\label{eq-sheaf-k}
 \mcK=p_{2,3}^*\mcL_{\psi}\otimes\bigotimes_{i=1}^{2}  (f_i^* \mcF
 \otimes f_{i+2}^* \mcF^\vee ).
\end{equation}
\par
The sheaf $\mcK$ is a constructible $\bQl$-sheaf of rank $k^4$ on
$\Aa^7$, pointwise mixed of weights $\leq 0$. It is lisse and
pointwise pure of weight $0$ on the dense open set $U_{\mcK}$ which is
the complement of the union of the divisors given by the equations
$$
\quad \{s(r+b_i)=\mu\},\quad\text{ for }\mu\in
S_\mcF\text{ and } i=1,\ldots,4,
$$
where $S_{\mcF}$ is the set of ramification points of $\mcF$ in
$\Aa^1$. The trace function of $\mcK$ is
$$
t_{\mcK}(r,s,\lambda,\bfb)=\bfK(r,s,\lambda,\bfb)
$$
for $(r,s,\lambda,\bfb)\in U_{\mcK}(\Fq)$.
\par
Now we consider the projection $\pi^{(2)}\,:\, \Aa^7\lra \Aa^6$ given
by
$$
\pi^{(2)}(r,s,\lambda,\bfb)=(r,\lambda,\bfb),
$$
and the compactly-supported higher-direct image sheaves
$R^i\pi^{(2)}_{!}\mcK$.  Since the fibers of $\pi^{(2)}$ are curves,
these sheaves are zero unless $0\leq i\leq 2$.

\begin{lemma}\label{lm-support}
Assume that the sheaf $\mcF$ is bountiful with respect to the Borel
subgroup.
\par
\emph{(1)} For $0\leq i\leq 2$, the sheaf $R^i\pi^{(2)}_!\mcK$ on
$\Aa^6_{\Fq}$ is mixed of weights $\leq i$.
\par
\emph{(2)} Let $\mcV^{\Delta}$ be the subvariety of $\Aa^4$ given in
Definition~\ref{multidiagonaldef}. The sheaves $R^0\pi^{(2)}_!\mcK$
and $R^2\pi^{(2)}_!\mcK$ are supported on
$\Aa^1\times\Aa^1\times\mcV^{\Delta}$.
\par
\emph{(3)} For $(r,\lambda,\bfb)$ such that
$\bfb\notin \mcV^{\Delta}$, the geometric monodromy representation of
the sheaf $\mcK_{r,\lambda,\bfb}$ does not contain a trivial
subrepresentation on a dense open subset of $\Aa^1$ where it is lisse.
\end{lemma}

\begin{proof}
  The first part is an application of Deligne's main
  theorem~\cite[Theorem 1]{WeilII}. For the second part, by the proper base
  change theorem, the stalk of $R^i\pi^{(2)}_!\mcK$ at
  $x=(r,\lambda,\bfb)\in \Aa^6$ is
$$
H^i_c(\Aa^1_{\bFq}, \sheaf{L}_{\psi(s\lambda)}\otimes
\bigotimes_{i=1}^{2}[\times (r+b_i)]^*\mcK\otimes [\times
(r+b_{i+2})]^*\mcK^{\vee})
$$
where $s$ is the coordinate on $\Aa^1$.
\par
This cohomology group vanishes for $i=0$ and any $x$. For $i=2$ and
$x\notin \mcV^{\Delta}$, its vanishing is given by~\cite[Theorem
1.5]{FKMSP} using (only) the assumption that $\mcF$ is bountiful in
the sense of our definition.
\par
For the last part, we first consider a closed point
$x=(r,\lambda,\bfb)$. Then the vanishing of the stalk
$$
H^2_c(\Aa^1_{\bFq}, \sheaf{L}_{\psi(s\lambda)}\otimes
\bigotimes_{i=1}^{2}[\times (r+b_i)]^*\mcK\otimes [\times
(r+b_{i+2})]^*\mcK^{\vee})
$$
of $R^2\pi^{(2)}_!\mcK$ implies that the geometric monodromy
representation of $\mcK_{r,\lambda,\bfb}$ has no trivial
subrepresentation where it is lisse (by the co-invariant formula and the
semisimplicity that holds because the sheaf is pure of weight
$0$). The statement then extends to all points by Pink's
Specialization Theorem~\cite[Th. 8.18.2]{ESDE}.
\end{proof}

The sheaf $R^1\pi^{(2)}_!\mcK$, which is mixed of weights at most $1$,
is almost the sheaf we want to understand. However, some cleaning-up
is required to facilitate the later arguments. Precisely, recall
(see~\cite[Th. 3.4.1 (ii)]{WeilII}) that a lisse sheaf which is mixed
of weight $\leq w$ is an extension of a lisse sheaf which is pure of
weight $w$ by a mixed sheaf of weight $\leq w-1$. Thus the following
definition makes sense:

\begin{definition}[Sum-product sheaf]\label{def-sp}
  Let $\mcF$ be a bountiful sheaf on $\Aa^1_{\Fq}$, and let $\mcK$ be
  the sheaf~(\ref{eq-sheaf-k}) and $\mcR=R^1\pi^{(2)}_!\mcK$.
  Consider the stratification $(X_i)_{1\leq i\leq m}$ of $\Aa^6_{\Fq}$
  such that
\begin{itemize}
\item $X_1$ is the maximal open subset of $\Aa^6$ on which $\mcR$ is
  lisse;
\item for $i\geq 2$, $X_i$ is the maximal open subset of
  $\Aa^6\setminus (X_1\cup\cdots\cup X_{i-1})$ on which $\mcR$ is
  lisse.
\end{itemize}
 We define the
  \emph{sum-product transform sheaf} $\mcR^*$ associated to $\mcF$ as
  the constructible sheaf given as the sum over $X_i$ of the maximal
  quotient of $\mcR|X_i$ which is pure of weight $1$ extended by zero
  to all of $\Aa^6_{\Fq}$, so that $\mcR^*|X_i$ is the maximal
  quotient of $\mcR|X_i$ pure of weight $1$.
\par
For any $(\lambda,\bfb)\in \Aa^5$, we denote by
$\mcR^*_{\lambda,\bfb}$ the pullback of $\mcR^*$ to the affine line
given by the morphism $r\mapsto (r,\lambda,\bfb)$, and we call
$\mcR^*_{\lambda,\bfb}$ a \emph{specialized sum-product sheaf}.
\end{definition}

By construction, the sum-product sheaf is punctually pure of weight
$1$. A first property of this sheaf is as follows:

\begin{proposition}\label{corRrbcorrelation}
For any $d\geq 1$, we have
$$
\frac{1}{(q^d)^5}
\sumsum_{(r,\bfb)\in\Ff^5_{q^d}}|t_{\mcR^*}(r,0,\bfb;\Fqd)|^2=q^d+O(q^{d/2}).
$$
\end{proposition}

\begin{proof}
Since
$$
t_{\mcR^*}(r,0,\bfb;\Fqd)=t_{\mcR}(r,0,\bfb;\Fqd)+O(1),
$$
by construction, it is enough to prove that
$$
\frac{1}{(q^d)^5}
\sumsum_{(r,\bfb)\in\Ff^5_{q^d}}|t_{\mcR}(r,0,\bfb;\Fqd)|^2=q^d+O(q^{d/2}).
$$
Let $\mcV^{\Delta}$ be the subvariety of $\Aa^4$ of
Definition~\ref{multidiagonaldef}. 
We have
$$
\sumsum_{(r,\bfb)\in\Ff^5_{q^d}}|\bfR(r,0,\bfb;\Fqd)|^2=
\sumsum_\stacksum{(r,\bfb)\in\Ff^5_{q^d}}{\bfb\not\in\mcV^{\Delta}(\Fqd)}
|\bfR(r,0,\bfb;\Fqd)|^2+
\sumsum_\stacksum{(r,\bfb)\in\Ff^5_{q^d}}{\bfb\in\mcV^{\Delta}(\Fqd)}
|\bfR(r,0,\bfb;\Fqd)|^2.
$$
Since $\mcV^\Delta$ has codimension $2$ and
$\bfR(r,0,\bfb;\Fqd)\ll_k q^d$, the second sum is bounded by
$\ll_k q^{5d}$. On the other hand, the first sum equals
$$
\sumsum_\stacksum{(r,\bfb)\in\Ff^5_{q^d}}{\bfb\not\in\mcV^{\Delta}(\Fqd)}
|t_{\mcR}(r,0,\bfb;\Fqd)|^2.
$$
\par
By the same argument we get
$$
\sumsum_\stacksum{(r,\bfb)\in\Ff^5_{q^d}}{\bfb\not\in\mcV^{\Delta}(\Fqd)}
|t_{\mcR}(r,0,\bfb;\Fqd)|^2= \sumsum_{(r,\bfb)\in\Ff^5_{q^d}}
|t_{\mcR}(r,0,\bfb;\Fqd)|^2 +O(q^{5d}),
$$
and the result then follows from Lemma~\ref{lm-more-average}.
\end{proof}



The following lemma is a fairly standard one.

\begin{lemma}
  Let $\sheaf{F}_1$ and $\sheaf{F}_2$ be two lisse $\ell$-adic sheaves
  on a smooth geometrically connected scheme $X/\Fq$. Assume that
  $\sheaf{F}_1$ and $\sheaf{F}_2$ are both pure of some weight $w$ and
  that for any $d\geq 1$ and any $x\in X(\Ff_{q^d})$, we have
$$
t_{\sheaf{F}_1}(x;\Fqd)=t_{\sheaf{F}_2}(x;\Fqd)+O(q^{d(w-1)/2}),
$$
where the implied constant is absolute. Then the semisimplifications
of $\sheaf{F}_1$ and $\sheaf{F}_2$ are isomorphic.
\end{lemma}

\begin{proof}
  By the Chebotarev density theorem, it suffices to prove that the
  trace functions of $\sheaf{F}_1$ and $\sheaf{F}_2$ actually coincide
  (see, e.g.,~\cite[Prop. 1.1.2.1]{laumon87}). After applying a
  suitable Tate twist, we may assume that $w=0$. Let $d\geq 1$ and let
  $x\in X(\Ff_{q^d})$. Denote by $(\alpha_i)$ (resp. $(\beta_j)$) the
  (complex) eigenvalues of the Frobenius at $x$ relative to $\Fqd$ on
  $\sheaf{F}_1$ (resp. $\sheaf{F}_2$). By assumption, for any integer
  $k\geq 1$, we have
$$
\sum_{i} \alpha_i^k=\sum_{j}\beta_j^k+O(q^{-k/2}).
$$
We multiply this by $z^k$ and sum over $k\geq 1$, getting
$$
\sum_i\frac{\alpha_i z}{1-\alpha_i z}=
\sum_{j}\frac{\beta_j z}{1-\beta_j z}+R(z)
$$
where $R(z)$ is holomorphic for $|z|<q^{1/2}$. Comparing poles, we
deduce that the $\alpha_i$'s are a permutation of the $\beta_j$'s,
hence the result.
\end{proof}


We deduce from this an important duality property.

\begin{lemma}\label{lm-dual-sheaf}
  For $\bfb=(b_1,b_2,b_3,b_4)\in \Aa^4$, let
  $\tilde{\bfb}=(b_3,b_4,b_1,b_2)$.  For any $\lambda$ and
  $\bfb\notin\mcV^{\Delta}$, the arithmetic semisimplifications of
  $\mcR^{*\vee}_{\lambda, \bfb}$ and
  $ \mcR^*_{-\lambda, \tilde{\bfb}}(1)$ are isomorphic on any dense
  open subset where $\mcR^*_{\lambda,\bfb}$ is lisse.
\end{lemma}

\begin{proof}
  Let $U$ be a dense open subset where $\mcR^*_{\lambda,\bfb}$ is
  lisse. We will check that the sheaves
  $\mcR^{*,\vee}_{\lambda,\bfb}$ and
  $\mcR^*_{-\lambda, \tilde{\bfb}}(1)$, which are both pure of weight
  $-1$, satisfy the conditions of the previous lemma.  Indeed, let
  $d\geq 1$ and $x\in U(\Fqd)$ be given. We observe that
\begin{multline*}
  t_{\mcR^*_{-\lambda,\tilde{\bfb}}}(x;\Fqd) =t_{\mcR_{-\lambda,
      \tilde{\bfb}}}(x;\Fqd)+O(1)
  =- \bfR(x, -\lambda, \tilde{\bfb}) +O(1)\\
  =- \overline{\bfR(x, \lambda,\bfb)} +O(1)
  =\overline{t_{\mcR_{\lambda,\bfb}}(x;\Fqd)} +O(1)
  = \overline{t_{\mcR^*_{\lambda,\bfb}}(x;\Fqd)}+O(1).
\end{multline*}
Since $\mcR^*_{\lambda,\bfb}$ is pure of weight $1$ on $U$, we have
further
$$
t_{\mcR^{*\vee}_{\lambda,\bfb}}(x;\Fqd)=\frac{1}{q^d}\overline{
  t_{\mcR^*_{\lambda,\bfb}}(x;\Fqd)}=
\frac{1}{q^d}t_{\mcR^*_{-\lambda,\tilde{\bfb}}}(x;\Fqd)+O(q^{-d}).
$$
The conclusion now follows.
\end{proof}

\subsection{Properties of Kloosterman sheaves}\label{sec-kl-sheaves}

We will study the sum-product transform of the Kloosterman sheaves.
We first summarize the basic properties of the Kloosterman sheaves,
which were originally defined by Deligne.

\begin{proposition}[Kloosterman sheaves]\label{pr-kl}
  Let $q>2$ be a prime number, $\ell\not=q$ an auxiliary prime number
  and $\psi$ a non-trivial $\ell$-adic additive character of
  $\Fq$. Let $k\geq 2$ be an integer.
\par
There exists a constructible $\bQl$-sheaf $\HYPK_k=\HYPK_{\psi,k}$ on
$\Pp^1_{\Fq}$, with the following properties:
\begin{enumerate}
\item For any $d\geq 1$ and any $x\in \Gm(\Fqd)$, we have
$$
t_{\KL}(x;\Fqd)=\Kl_k(x;\Fqd)=
\frac{(-1)^{k-1}}{q^{d(k-1)/2}}\sum_{x_1\cdots
  x_k=x}\psi_{\Ff_{q^d}}(x_1+\cdots+x_k).
$$
\item The sheaf $\HYPK_{\psi,k}$ is lisse of rank $k$ on $\Gm$.
\item On $\Gm$, the sheaf $\HYPK_{\psi,k}$ is geometrically
  irreducible and pure of weight $0$.
\item The sheaf $\HYPK_{\psi,k}$ is tamely ramified at $0$ with
  unipotent local monodromy with a single Jordan block.
\item The sheaf $\HYPK_{\psi,k}$ is wildly ramified at $\infty$, with
  a single break equal to $1/k$, and with Swan conductor equal to $1$.
\item There is an arithmetic isomorphism 
$$
\HYPK_{\psi,k}^\vee\simeq [x\mapsto (-1)^kx]^*\HYPK_{\psi,k},
$$
and in particular $\HYPK_{\psi,k}$ is arithmetically self-dual if $k$
is even.
\item If $k\geq 2$, then the arithmetic and geometric monodromy groups
  of $\HYPK_{\psi,k}$ are equal; if $k$ is even, they are equal to
  $\Sp_k$ and if $k$ is odd, then they are equal to $\SL_k$.
\item The stalks of $\HYPK_{\psi,k}$ at $0$ and $\infty$ both vanish.
\item If $\gamma\in \PGL_2(\bFq)$ is non-trivial, there does not exist
  a rank $1$ sheaf $\sheaf{L}$ such that we have a geometric
  isomorphism over a dense open set
$$
\gamma^*\HYPK_{\psi,k}\simeq \HYPK_{\psi,k}\otimes \sheaf{L}.
$$
\end{enumerate}
\end{proposition}

\begin{proof}
  All this is essentially \emph{mise pour m\'emoire}
  from~\cite{GKM}. The sheaf $\HYPK_k$ is the sheaf denoted
  $\mathrm{Kl}_n(\psi)((k-1)/2)$ in~\cite[11.0.2]{GKM}; precisely,
  properties (1) to (5) are stated with references
  in~\cite[11.0.2]{GKM}, property (6) is found in~\cite[Cor. 4.1.3,
  Cor. 4.1.4]{GKM}, and the crucial property (7) is~\cite[Th. 11.1,
  Cor. 11.3]{GKM}. The sheaf constructed in~\cite{GKM} is on $\Gm$,
  and we extend by zero from $\Gm$ to $\Pp^1$, making property (8)
  true by definition. The last property is explained, e.g.,
  in~\cite[\S 3, (b), (c)]{FKMSP}.
\end{proof}

\begin{remark}\label{rm-kl-def}
  (1) As a matter of definition, one possibility is to define
  $\HYPK_{\psi,k}$ as $k$-fold (Tate-twisted) multiplicative
  convolution of the basic Artin-Schreier sheaf $\mcL_\psi$, namely
$$
\HYPK_k=(\mcL_{\psi}\star \cdots\star \mcL_{\psi})((k-1)/2),
$$
see~\cite[5.5]{GKM}. 
\par
(2) Katz has also shown (see~\cite[Cor. 4.1.2]{GKM}) that the property
(1) characterizes $\HYPK_{\psi,k}$ as a lisse sheaf on $\Gm$, up to
arithmetic isomorphism.
\par
(3) It might seem more natural to define $\HYPK_{\psi,k}$ as the
middle extension from $\Gm$ to $\Pp^1$ of the sheaf constructed by
Katz. However, the property of being a middle extension is not
preserved by tensor product, so we would not be able to use directly
any of the properties of middle extension sheaves when studying the
sum-product transform sheaves. On the other hand, having stalk zero is
preserved by tensor product, and it will turn out that this property
simplifies certain technical arguments.
\end{remark}

\begin{cor} 
  For $k\geq 2$, the sheaf $\HYPK_{\psi,k}$ is bountiful with respect
  to the full group $\PGL_2$; it is of $\Sp$-type if $k$ is even, and
  of $\SL$-type if $k$ is odd. In the second case, $\HYPK_{\psi,k}$
  has the special involution $\begin{pmatrix} -1&0\\0&1
  \end{pmatrix}$.  Moreover, the conductor of $\HYPK_{\psi,k}$ is
  bounded in terms of $k$ only.
\end{cor}
\begin{proof}
  This is clear from Proposition~\ref{pr-kl} using the definition of
  bountiful sheaves and of the conductor of a sheaf.
\end{proof}

For convenience, we will most often simply denote
$\HYPK_k=\HYPK_{\psi,k}$ since we assume that $\psi$ is fixed.
\par
\medskip
\par
The following lemma computes precisely the local monodromy of
$\HYPK_k$ at $\infty$. This is a special case of a formula of
L. Fu~\cite[Prop. 0.8]{Fu} (which also describes $\HYPK_k$ as a
representation of the decomposition group, not just the inertia
group).

\begin{lemma}\label{lm-kl-infty} 
  Assume $q>k\geq 2$. Denote by $\tilde{\psi}$ the additive character
  $x\mapsto \psi(kx)$ of $\Fq$. Then, as representations of the
  inertia group $I(\infty)$ at $\infty$, there exists an isomorphism
$$
\HYPK_k\simeq [x\mapsto
x^k]_*(\mcL_{\chi_2^{k+1}}\otimes\mcL_{\tilde{\psi}}),
$$
where we recall that
$\chi_2$ is the unique non-trivial character of order $2$ of $\Fqt$.
\end{lemma}

\begin{proof} According to the remark in~\cite[10.4.5]{GKM}, we have
  an isomorphism
$$
\HYPK_k\simeq [x\mapsto
x^k]_*(\mcL_{\tilde{\psi}}),
$$
as a representation of the wild inertia subgroup $P(\infty)\subset
I(\infty)$. On the other hand, by~\cite[\S 1.18]{GKM}
and~\cite[Theorem 8.6.3]{ESDE}, an
$I(\infty)$-representation which is totally wild with Swan conductor
$1$ is determined, up to scaling, by its rank and its determinant
(i.e., if two such representations $\pi_1$ and
$\pi_2$ have same rank and determinant, then there exists a non-zero
$c$ such that $\pi_2\simeq [\times c]^*\pi_1$).
\par
Since $\det
\HYPK_k$ is trivial (see Proposition~\ref{pr-kl} (7)), it is therefore
sufficient to check that the determinant of the
$I(\infty)$-representation $[x\mapsto
x^k]_*(\mcL_{\chi_2^{k+1}}\otimes\mcL_{\tilde{\psi}})$ is trivial.
But for any multiplicative character
$\chi$, we have a geometric isomorphism
$$
\det([x\mapsto x^k]_*(\mcL_{\chi}\otimes\mcL_{\tilde{\psi}})) \simeq
\chi\chi_2^{k+1}
$$
and this is geometrically trivial if $\chi=\chi_2^{k+1}$
(this follows, e.g., from the Hasse-Davenport relations as
in~\cite[Proposition 5.6.2]{GKM}, or from the block-permutation matrix
representation of an induced representation, similarly to the argument
that appears later in Lemma~\ref{lm-inductionformula}).
\end{proof}

Finally, we can state our main theorem concerning the sum-product
sheaves associated to Kloosterman sheaves.

\begin{theorem}[Irreducibility of sum-product
  sheaves]\label{thmirreducibility}
  Let $k\geq 2$ be an integer. 
  Let $\ell$ be a prime $\not=q$ and let $\mcR^*$ be the $\ell$-adic
  sum-product transform sheaf of $\KL_k$ over $\Fq$.
  \par
  If $q$ is sufficiently large with respect to $k$, there exists a
  closed subset $\mcV^{bad} \subset \Aa_\Fq^4$ containing
  $\mcV^{\Delta}$, of codimension $1$ and of degree bounded
  independently of $q$, stable under the automorphism
  $(b_1,b_2,b_3,b_4)\mapsto (b_3,b_4,b_1,b_2)$, such that for all
  $\bfb=(b_1,b_2,b_3,b_4)$ not in $\mcV^{bad}$, the following
  properties hold:
\begin{enumerate}
\item For all $\lambda$, the specialized sum-product sheaf
  $\mcR^*_{\lambda,\bfb}$ is lisse and geometrically irreducible on a
  dense open subset of $\Aa^1$;
\item For all $\lambda$, there does not exist a dense open subset $U$
  of $\Aa^1$ such that $\mcR^*_{\lambda,\bfb}|U$ is geometrically
  trivial;
\item If $\lambda\not=\lambda'$, then there does not exist a dense
  open subset $U$ of $\Aa^1$ such that $\mcR^*_{\lambda,\bfb}|U$ is
  geometrically isomorphic to $\mcR^*_{\lambda',\bfb}|U$ .
\item For all $\lambda_1,\lambda_2, \bfb_1$, $\bfb_2$, the dimensions
  of the stalks of the sheaf $\mcR_{\lambda_i,\bfb_i}$, and the
  dimensions of the cohomology groups
  $H^i_c(\Aa^1_{\bFq}, \mcR_{\lambda_1,\bfb_1})$ and
  $H^i_c (\Aa^1_{\bFq}, \mcR_{\lambda_1,\bfb_1} \otimes
  \mcR_{\lambda_2, \bfb_2} )$ are bounded in terms of $k$ only, in
  particular independently of $q$ for $k$ fixed.
\end{enumerate}
\end{theorem}

After some preliminaries, the proof splits into two cases: the case
$\lambda=0$ in Section~\ref{sec-zero} and the case $\lambda\not=0$ in
Section~\ref{sec-not-zero}.

First, let us recall how this theorem implies our desired
Theorems~\ref{thm1Klk} and~\ref{thm2Kl2}.

\begin{theorem}\label{thmRbounds}  Let $k\geq 2$ be an integer. 
  Let $\bfR(r,\lambda,\bfb)$ be the function on $\Aa^6(\Fq)$ defined
  in \eqref{Rdef}. For any $\bfb\in\Ff_q^4-\mcV^{bad}(\Fq)$ and any
  $\lambda,\lambda'\in\Fq$, we have
\begin{align*}
  \sum_{r\in\Fq}\bfR(r,\lambda,\bfb)
  &\ll q,\\
  \sum_{r\in\Fq}\bfR(r,\lambda,\bfb)
  \ov{\bfR(r,\lambda',\bfb)}
  &=
    \delta(\lambda,\lambda')q^2+O(q^{3/2}),
\end{align*}
where the implied constant depends only on $k$.
\end{theorem} 

\begin{proof} It is sufficient to prove the theorem when $q$ is
  sufficiently large with respect to $k$, since we can handle any
  finite set of primes by replacing the implied constant by a larger
  one using trivial bounds for the sums.
\par
First of all, note that by the proper base change theorem and the
Grothendieck-Lefschetz trace formula, we have
\begin{equation}\label{eq-trace-1}
t_{\mcR}(r,\lambda,\bfb)=-\sum_{s\in\Fq} t_{\mcK}(r,s,\lambda,\bfb)=
-\bfR(r,\lambda,\bfb)
\end{equation}
for $\bfb\notin \mcV^{\Delta}(\Fq)$, where the implied constant depends only on
$k$. Since $\mcR$ is mixed of weights $\leq 1$ and of rank bounded in
terms of $k$ only, we have
$$
t_{\mcR}(r,\lambda,\bfb)\ll q^{1/2}
$$
for $\bfb\notin \mcV^{\Delta}(\Fq)$.
\par
We begin the proof of the second bound. Thus let
$\bfb\in\Ff_q^4-\mcV^{bad}(\Fq)$ (in particular
$\bfb\not\in\mcV^\Delta(\Fq)$) and $\lambda$, $\lambda'\in\Fq$ be
given.
First, 
we have
$\overline{\bfR(r, \lambda,\bfb)}= \bfR(r,-\lambda,\tilde{\bfb})$,
where $\bfb=(b_3,b_4,b_1,b_2)\in\Ff_q^4-\mcV^{bad}(\Fq)$.
Thus the
relation~(\ref{eq-trace-1}) and the Grothendieck-Lefschetz trace
formula imply that
$$
\sum_{r\in\Fq}\bfR(r,\lambda,\bfb)\ov{\bfR(r,\lambda',\bfb)} =
\sum_{i=0}^{2} (-1)^i \Tr\left(\Frob_q\,\mid\, H^i_c (\Aa^1_{\bFq},
  \mcR_{\lambda, \bfb} \otimes \mcR_{-\lambda',\tilde{\bfb}}
  )\right).
$$
Let $\mcF= \mcR_{\lambda, \bfb} \otimes \mcR_{-\lambda',\tilde{\bfb}}$
and
$\mcF^*=\mcR^*_{\lambda, \bfb} \otimes
\mcR^*_{-\lambda',\tilde{\bfb}}$.  Since $\mcR$ is mixed of weight
$\leq 1$, the tensor product sheaf $\mcF$ is mixed of weight $\leq 2$,
so the $i$-th compactly supported cohomology group with coefficient in
$\mcF$ is mixed of weight $\leq i + 2$ by Deligne's
Theorem~\cite{WeilII}.

The dimension of these cohomology groups are bounded in terms of $k$
only by Theorem \ref{thmirreducibility} (4). Thus we have
$$
\sum_{r\in\Fq}\bfR(r,\lambda,\bfb)\ov{\bfR(r,\lambda',\bfb)}=
\Tr(\Frob_q\,\mid\, W_{\lambda,\lambda'})+O(q^{3/2})
$$
where $W_{\lambda,\lambda'}$ is the subspace of weight $4$ in
$H^2_c(\Aa^1_{\bFq},\mcF)=H^2_c(U_{\bFq},\mcF)$, and the implied constant
depends only on $k$ (here $U$ is any dense open set where $\mcF$ is
lisse).
\par
We have by definition a short exact sequence
$$
0\lra \sheaf{G}\lra \sheaf{F}\lra \sheaf{F}^*\lra 0
$$
of lisse sheaves on $U$ where $\sheaf{G}$ is mixed of weights
$<2$. Taking the long cohomology exact sequence and applying again
Deligne's Theorem, we see that
$W_{\lambda,\lambda'}\simeq W^*_{\lambda,\lambda'}$, where
$W^*_{\lambda,\lambda'}$ is the subspace of weight $4$ in
$H^2_c(U_{\bFq},\mcF^*)$.
\par
By the coinvariant formula, we have
$$
H^2_c (U_{\bFq}, \mcF^*)=(\mcF^*_{\bar{\eta}})_{\pi_1(U_{\bFq})}(-1),
$$ 
so it is sufficient to prove that the weight $2$ part of the
$\pi_1(U_{\bFq})$-coinvariants of $\mcF^*$ has dimension
$\delta(\lambda, \lambda')$, and that the action of $\Frob_q$ is
multiplication by $q$ when $\lambda=\lambda'$.
\par
The sheaves $\mcR^*_{\lambda,\bfb}$ and $\mcR^*_{\lambda',\bfb}$ are
geometrically irreducible by Theorem~\ref{thmirreducibility} (1), in
particular they are arithmetically semisimple. By
Lemma~\ref{lm-dual-sheaf}, we have arithmetic isomorphisms
$$
\mcR^*_{-\lambda', \tilde{\bfb}} \simeq \mcR^{*\vee}_{\lambda',
  \bfb}(-1),\ \mcF^*\simeq \mcR^{*}_{\lambda, \bfb}\otimes
\mcR^{*\vee}_{\lambda', \bfb}(-1)
$$
on $U$. 
Again by geometric irreducibility of $\mcR^*_{\lambda,\bfb}$ and
$\mcR^*_{\lambda',\bfb}$, the monodromy coinvariants of that tensor
product is one-dimensional if $\mcR^*_{\lambda, \bfb}$ and
$\mcR^{*}_{\lambda',\bfb}$ are geometrically isomorphic and is zero
otherwise. By Theorem~\ref{thmirreducibility} (3), the sheaves are
geometrically isomorphic if and only if $\lambda=\lambda'$. In that
later case the space of (geometric) coinvariants of
$\mcR^{*}_{\lambda, \bfb}\otimes \mcR^{*\vee}_{\lambda', \bfb}$ is
one-dimensional, generated by the trace, and $\Frob_q$ acts trivially
on it; therefore $\Frob_q$ acts by multiplication by $q$ on $\mcF^*$.

The argument for the first bound is similar but simpler. We work with
the cohomology groups $H^i_c (\Aa^1_{\bFq}, \mcR_{\lambda, \bfb})$,
which are mixed of weights $\leq i+1$. It is sufficient to show that
the weight $3$ part of $H^2_c (\Aa^1_{\bFq}, \mcR_{\lambda, \bfb})$
vanishes, and thus sufficient to show that the weight $1$ part of the
monodromy coinvariants of $\mcR_{\lambda,\bfb}$ vanishes. Because
$\mcR^*$ is the weight $1$ part of $\mcR$, this is the same as showing
that the monodromy coinvariants of $\mcR^*_{\lambda,\bfb}$
vanishes. But $\mcR^*_{\lambda,\bfb}$ is irreducible and nontrivial as
a monodromy representation, by Theorem~\ref{thmirreducibility} (2)
(3), so it has no coinvariants.
\end{proof}


\subsection{Preliminaries}\label{sec-prelim-geo}

We collect in this section a number of results and definitions that we
will use in the proof of our results. In a first reading, it might be
easier to only survey the statements before going to the next section.

We will derive the irreducibility statement of
Theorem~\ref{thmirreducibility} for $\lambda\not=0$ from the second of
the following criteria.

\begin{lemma}\label{lemcriterion}
  Let $X_0$ and $Y_0$ be normal varieties
  over 
  $\Ff_q$. Let $f\,:\,Y_0 \lra X_0$ be a smooth proper morphism whose
  fibers are curves and whose geometric fibers are
  geometrically connected.  Let $D_0 \subset Y_0$ be a divisor. Write
  $X$, $Y$ and $D$ for the corresponding varieties over $\bFq$.
\par
For a lisse $\bQl$-sheaf $\mcF$ on $Y_0-D_0$, consider the three
following conditions:
\begin{enumerate}
\item The sheaf $\mcF$ is geometrically irreducible and pure of some
  weight;
\item For the generic point $\eta$ of $X$, there exists a point
  $z$ of $D_\eta$ defined over the function field $\kappa(\eta)$
  of $X$ such that there exists an irreducible component of
  multiplicity one of the restriction of the monodromy representation
  of $\mcF_{\eta}$ to the inertia group at $z$ whose isomorphism class
  is preserved by the action of the Galois group of $\kappa(\eta)$ by
  conjugation on representations of the inertia group;
\item The divisor $D$ is finite and flat over $X$, and
the function
$$
x\mapsto \sum_{y\in Y_x-D_x} (\swan_y(\mcF\otimes
\mcF^\vee) +\rank(\mcF\otimes\mcF^\vee))
$$ 
is locally constant on $X$.
\end{enumerate}
\par
Then the following statements are true:
\par
\emph{(a)} If \emph{(1)} and \emph{(2)} hold, then for all $x$ in a
dense open subset of $X$, the restriction
$\mcF_x=\mcF|(Y_x-D_x)$ to a fiber $Y_x-D_x$
is geometrically irreducible.
\par
\emph{(b)} If \emph{(1)}, \emph{(2)} and \emph{(3)} hold, then for all
$x$ in $X$, the restriction $\mcF|(Y_x-D_x)$ to a fiber $Y_x-X_x$ is
geometrically irreducible.
\end{lemma}

\begin{proof} 
  We assume that conditions (1) and (2) hold.
  
  Let $\eta'$ be the generic point of $Y$. By~\cite[V, Proposition
  8.2]{sga1}, the natural homomorphism $\pi_1(\eta')\lra \pi_1(Y-D)$
  is surjective. Since it factors through the natural homomorphism
 $$
\pi_1(Y_\eta-D_{\eta}) \to \pi_1(Y-D),
$$
it follows that the latter is also surjective. In particular,
condition (1) shows that the restriction of $\mcF$ to
$Y_{\eta}-D_{\eta}$ corresponds to an irreducible representation of
$\pi_1(Y_\eta-D_{\eta})$. Thus $\mcF_{\eta}=\mcF|(Y_{\eta}-D_{\eta})$
is an irreducible lisse sheaf on $Y_\eta-D_\eta$.

  Consider now a geometric point $\bar{\eta}$ over $\eta$, the
  geometric fibers $Y_{\bar{\eta}}$ and $D_{\bar{\eta}}$ and the
  pullback $\mcF_{\bar{\eta}}$ of $\mcF_{\eta}$ to
  $(Y-D)_{\bar{\eta}}$. We will show that condition (2) implies that
  $\mcF_{\bar{\eta}}$ is irreducible.

  Indeed, the representation of $\pi_1(Y_{\bar{\eta}}-D_{\bar{\eta}})$ corresponding
  to $\mcF_{\bar{\eta}}$ is semisimple, as the restriction to a normal
  subgroup of an irreducible, hence semisimple, representation.
Let
$$
\mcF_{\bar{\eta}}=\bigoplus_{i\in I} n_i V_i
$$
be a decomposition of this representation of $\pi_1(Y_{\bar{\eta}}- D_{\bar{\eta}})$
into irreducible subrepresentations, where $n_i\geq 1$ and the $V_i$
are pairwise non-isomorphic. The quotient
$$
G=\pi_1(Y_{\eta}-D_{\eta})/\pi_1(Y_{\bar{\eta}}-D_{\bar{\eta}}),
$$
is isomorphic to the Galois group of the function field $\kappa(\eta)$
of $X$ since $f$ has geometrically connected generic fiber. It acts on
the set $\{V_i\}$ of irreducible subrepresentations of
$\mcF_{\bar{\eta}}$. Since $\mcF_{\eta}$ is an irreducible
representation of $\pi_1(Y_{\eta}-D_{\eta})$, this action is
transitive. Hence, for any point $y$ of $D_{\eta}$ defined over
$\kappa(\eta)$ , the restriction of $\mcF_{\eta}$ to the inertia group
at $y$ has the property that it is a direct sum of $n=\sum n_i$
subrepresentations which are $G$-conjugates (but not necessarily
irreducible or even indecomposable). In particular, any irreducible
subrepresentation of the inertia group whose isomorphism class is
fixed by $G$ appears with multiplicity divisible by $n$. By condition
(2), this means that $n=1$, so that $\mcF_{\bar{\eta}}$ is
irreducible.

By Pink's Specialization Theorem (see~\cite[Th. 8.18.2]{ESDE}), it
follows that $\mcF_{x}$ is geometrically irreducible for all $x$ in
some dense open subset, which gives (a).


Now assume further that condition (3) holds. For a closed point
$x \in X$, the fiber $\mcF_x$ is geometrically irreducible if and only
if the cohomology group
$H^2_c ( (Y - D)_{x,\bFq}, \mcF_x\otimes\mcF_x^\vee)$ is
one-dimensional, by the coinvariant formula for the second cohomology
group on a curve (see, e.g.,~\cite[2.0.4]{GKM}) and the fact that
$\mcF_x$, being pure by condition (1), is geometrically semisimple
(see~\cite[Th. 3.4.1 (iii)]{WeilII}). Equivalently, by the proper base
change theorem, the specialized sheaf $\mcF_x$ is geometrically
irreducible if and only if the stalk of
$R^2 f_!  ( \mcF\otimes\mcF^{\vee})$ at $x$ is one-dimensional. 
Condition (3) and Deligne's semicontinuity theorem~\cite[Corollary
2.1.2]{LaumonSMF} imply that the sheaf
$R^2 f_!  ( \mcF \otimes \mcF^{\vee})$ is lisse on $X$. Since it has
rank $1$ at all closed points in an open set, by what we proved
before, it has rank $1$ on all of $X$, which means that $\mcF_x$ is
geometrically irreducible for all closed points $x$ in $X$. By Pink's
Specialization Theorem (see~\cite[Th. 8.18.2]{ESDE}), $\mcF_x$ is
geometrically irreducible for all points in $X$.
\end{proof}

\begin{remark} 
  Our proof of condition (1) below generalizes to quite general
  (bountiful) sheaves, but the proofs of conditions (2) and (3)
  involve careful calculations that depend on specific properties of
  the Kloosterman sheaves. This means that our results do not easily
  generalize to other sheaves.
\par
However, condition (2) is a ``generic'' condition that should hold for
a ``random'' sheaf. Thus it should be possible to prove it in a number
of different concrete cases. The last condition (3) is more subtle;
although is always true on a dense open subset (hence is generic in
that sense), the closed complement where it fails will usually have
codimension $1$. However, it should often be possible to compute
explicitly that subset, and to use this information for further study
(cf. Remark~\ref{rmkimprove} for instance).
\end{remark}

In this paper, we will only use the second criterion of
Lemma~\ref{lemcriterion} in the proof of
Theorem~\ref{thmirreducibility}, to show that for all $\bfb$ outside
of a proper subvariety, the specialized sheaves
$\mcR_{\lambda,\bfb}^*$ are geometrically irreducible for \emph{every}
non-zero $\lambda$. However, the first criterion might be useful in
other applications (in the first draft of this paper, we used it to
deal with sum-product sheaves where $\lambda=0$, but we later found a
simpler argument to deal with this case).
\par
To verify the first condition of the lemma, we will use Katz's
diophantine criterion for geometric irreducibility
(compare~\cite[Lemma 7.0.3]{rigid}).

\begin{lemma}[Diophantine criterion for
  irreducibility]\label{diophantine} 
  Let $Y$ be a normal variety over $\Fq$, $U\subset Y$ a dense open subset
  and $\mcF$ a sheaf on $Y$ that is lisse on $U$. Assume moreover that
  $\mcF|U$ is pure of some weight $w$, and that $\mcF$ is mixed
  of weights $\leq w$ on $Y$. Then $\mcF|U$ is geometrically
  irreducible if
$$
\frac{1}{q^{d\dim Y}}\sum_{y\in Y(\Fqd)}|t_\mcF(y)|^2 =q^{dw}(1+o(1))
$$
as $d$ tends to infinity.
\end{lemma}

\begin{proof}
  Using a Tate twist, we may assume that $w=0$.  Let $n$ be the
  dimension of $Y$ and $D=Y-U$.  We have
  $$
  \frac{1}{q^{nd}}\sum_{y\in
    Y(\Fqd)}|t_\mcF(y)|^2=\frac{1}{q^{nd}}\sum_{y\in
    U(\Fqd)}|t_\mcF(y)|^2+ \frac{1}{q^{nd}}\sum_{y\in
    D(\Fqd)}|t_\mcF(y)|^2.
$$
The second sum is bounded by $O(q^{-d})=o(1)$ using our assumption on
the weights of $\mcF$ on $Y$ (and the reduction to $w=0$), and hence
the assumption implies that
$$
\frac{1}{q^{nd}}\sum_{y\in U(\Fqd)}|t_\mcF(y)|^2 \ra 1
$$
as $d\ra +\infty$. On the other hand, the Grothendieck--Lefschetz
Trace Formula and the Riemann Hypothesis imply that
$$
\sum_{y\in U(\Fqd)}|t_\mcF(y)|^2 =\Tr(\frob_{\Fqd}\mid
H^{2n}_c(Y_{\bFq},\mcF\otimes\mcF^{\vee})) +O(q^{d(n-1/2)}),
$$ 
and therefore
$$
\frac{1}{q^{nd}}\sum_{y\in U(\Fqd)}|t_\mcF(y)|^2 =
\Tr(\frob_{\Fqd}\mid H^{2n}_c(Y_{\bFq},\mcF\otimes\mcF^{\vee})(n))+o(1).
$$
By the semisimplicity of $\mcF$ (see~\cite[Th. 3.4.1 (iii)]{WeilII})
and the coinvariant formula
$$
H^{2n}_c(Y_{\bFq},\mcF\otimes\mcF^{\vee})
\simeq (\mcF\otimes\mcF^{\vee})_{\pi(U_{\bFq})}(-n),
$$
we deduce by combining these formulas that the geometric invariant
subspace of $\mcF\otimes\mcF^{\vee}$ is one-dimensional, which by
Schur's Lemma means that $\mcF$ is geometrically irreducible.
\end{proof}

We will use the following lemma from elementary representation theory
to describe the local monodromy of tensor products of Kloosterman
sheaves.

\begin{lemma}\label{lm-inductionformula} Let $G$ be a group and $E$
  an arbitrary field. Let $H$ be a normal subgroup of $G$ of finite index. Consider
  the usual action $\sigma\cdot V=\sigma(V)$ of $G/H$ on
  $E$-representations of $H$, where $x\in H$ acts on $\sigma(V)$ by
  the action of $\sigma^{-1}x\sigma$ on $V$.
\par
For any finite-dimensional $E$-representations $V_1, \dots, V_n$ of
$H$, we have a canonical isomorphism
$$
 \bigotimes_{i=1}^n \Ind_H^G V_i \simeq \bigoplus_{ (\sigma_2,
  \dots, \sigma_n) \in (G/H)^{n-1} } \Ind_{H}^G \Bigl( V_1 \otimes
  \bigotimes_{i=2}^n \sigma_i(V_i ) \Bigr).
$$
\end{lemma} 

\begin{proof} We proceed by induction on $n$. The case $n=1$ is a
  tautology. For $n=2$, we need to prove that
  \[ \Ind_H^G V_1 \otimes \Ind_H^G V_2 \simeq \bigoplus_{\sigma \in
    G/H} \Ind_H^G \left( V_1 \otimes \sigma(V_2)\right) \]
\par
To see this, first apply the projection formula
\[ \Ind_H^G ( V_1 \otimes \Res_G^H \Ind_H^G V_2) = \Ind_H^G V_1
\otimes \Ind_H^G V_2 \]
and then the fact that
\[ \Res_G^H \Ind_H^G V_2 = \bigoplus_{\sigma \in G/H} \sigma(V_2), \]
which follows from the definition of induction (see,
e.g.,~\cite[Prop. 2.3.15, Prop. 2.3.18]{k-repr} for these standard
facts).
\par
We easily complete the proof for $n\geq 3$ by induction using the case
$n=2$.
\end{proof}

As a corollary, we now obtain the local monodromy at infinity for the
sheaves $\mcK_{r,\lambda,\bfb}$. To state the result, we recall from
the introduction the notation $\mcL_{\psi}(c s^{1/k})$, for a variety
$X/\Fq$, an integer $k\geq 1$ and a function $c$ on $X$: this is the
sheaf on $X\times \Aa^1$ (with coordinates $(x,s)$) given by
$$
\mcL_{\psi}(c s^{1/k})=\alpha_* \mcL_{\psi (c(x)t)}
$$ 
where $\alpha$ is the map
$$
\begin{cases}
X\times\Aa^1\rightarrow X\times\Aa^1\\
(x,t)\mapsto (x,t^k).
\end{cases}
$$

\begin{lemma}\label{lm-local-product}  
  Assume $q>k\geq 2$ and denote by $\tilde{\psi}$ the character
  $x\mapsto \psi(kx)$. Fix $r,\bfb,\lambda$ such that $r+b_i\not=0$
  for all $i$.  Let $(r+b_i)^{1/k}$ be a fixed $k$-th root of $r+b_i$
  in $\bFq$.
\par
Then the local monodromy at $s=\infty$ of $ \mcK_{r,\lambda,\bfb}$ is
isomorphic to the local monodromy at $s=\infty$ of the sheaf
\begin{equation}\label{eq-local-goal}
\mcL_{\psi(\lambda s)} \otimes \bigoplus_{\zeta_2,\zeta_3,\zeta_4 \in
  \mmu_k} \mcL_{\tilde{\psi}} \left( \left((r+b_1)^{1/k}+ \zeta_2
    (r+b_2)^{1/k} - \zeta_3 (r+b_3)^{1/k} - \zeta_4
    (r+b_4)^{1/k}\right) s^{1/k} \right)
\end{equation}
where $\mmu_k$ is the group of $k$-th roots of unity in $\bFq$.
\par
More generally, for fixed $\lambda$ and $\bfb$, for any algebraic
variety $U_{\Fq}$, let $f\,:\, U\lra \Aa^1-\{-b_1,\ldots, -b_4\}$ be a
morphism, and assume there are morphisms $r_i\,:\, U\lra \Aa^1$ such
that $[x\mapsto x^k]\circ r_i=[x+b_i]\circ f$. Assume that $k$ is odd
or that there exist a constant $c$ and a function $g$ on $U$ such that
$r_1r_2r_3r_4=cg^2$.
Then the local monodromy of the sheaf
$(f\times\mathrm{Id})^*\mcK_{\lambda,\bfb}$ on $U\times\Aa^1$ along
the divisor $U\times\{\infty\}$ is isomorphic to the local monodromy
of the sheaf
\[ 
\mcL_{\psi(\lambda s)} \otimes \bigoplus_{\zeta_2,\zeta_3,\zeta_4 \in
  \mmu_k} \mcL_{\tilde{\psi}} \left( \left(r_1+ \zeta_2r_2 - \zeta_3
    r_3 - \zeta_4 r_4\right) s^{1/k} \right)
\]
along $U\times\{\infty\}$
\end{lemma}

\begin{proof} 
We have
$$
\mcK_{r,\lambda,\bfb}=\mcL_{\psi(\lambda s)} \otimes
\bigotimes_{i=1}^2 [\times (r+b_i)]^*\HYPK_k \otimes [\times
(r+b_{i+2})]^*\HYPK_k^{\vee},
$$
so that it is enough to treat the case $\lambda=0$. Furthermore, the
first statement is the special case of the second where $U$ is a
single point (the second assumption holds with $c=r_1r_2r_3r_4$,
$g=1$), so it is enough to handle the second case. By definition, we
have
$$
(f\times\mathrm{Id})^*\mcK_{\lambda,\bfb} = (f\times\mathrm{Id})^*
\bigotimes_{i=1}^{2} (f_i^* \HYPK_k \otimes f_{i+2}^* \HYPK_k^\vee ) =
\bigotimes_{i=1}^{2}\left( (f\times\mathrm{Id})^* f_i^* \HYPK_k
  \otimes (f\times\mathrm{Id})^* f_{i+1}^* \HYPK_k^\vee\right)
$$
where $f_i$ is the map $(r,s)\mapsto s(r+b_i)$.
 
Let $\alpha: \Aa^1 \to \Aa^1$ be the morphism $t \mapsto t^k$. By
Lemma \ref{lm-kl-infty}, the local monodromy of $\HYPK_k$ at $\infty$
is $\alpha_* (\mcL_{\chi_2^{k+1}}\otimes\mcL_{\tilde{\psi}})$.

Let $V=\Aa^1-\{-b_1,\ldots,-b_4\}$. For each $i$, we have the
Cartesian diagram
$$
\begin{tikzcd}
  U \times  \Aa^1 \arrow{d}{\mathrm{Id}_U \times \alpha} \arrow{rrr}{(u,t)\mapsto r_i(u)t}&  & & \Aa^1 \arrow{d}{\alpha}\\
  U \times \Aa^1 \arrow{rr}{f \times \mathrm{Id}_{\Aa^1}} & &V\times
  \Aa^1 \arrow{r}{f_i} & \Aa^1
  \end{tikzcd}
$$
By proper base change, this implies that the local monodromy at
$\infty$ of $(f\times\mathrm{Id})^* f_i^* \HYPK_k$ is the same as the
local monodromy at $\infty$ of
$(\mathrm{Id}_U\times \alpha)_* \left( \mcL_{\chi_2^{k+1}}( r_i t)
  \otimes\mcL_{\tilde{\psi}}(r_it)\right)$.
In terms of representation theory, this means that the local monodromy
representation at $\infty$ is induced from the normal subgroup $H$ of
$G=\pi_1((U\times\Gg_m)_{\bFq})$ corresponding to the covering
$\mathrm{Id}_U\times \alpha$ (which we will simply denote $\alpha$ by
slight abuse of notation).
\par
The quotient group $G/H$ is naturally isomorphic to the Galois group
of the covering, which is isomorphic to $\mmu_k$ by the homomorphism
sending a root of unity $\zeta\in \mmu_k$ to the maps
$(s,t)\mapsto (s,\zeta t)$. One checks easily that the action of
$\zeta$ on representations of $H$ is given by
$$
\zeta\cdot \mcL_{\chi_2^{k+1}}=\mcL_{\chi_2^{k+1}},\quad\quad
\zeta\cdot \mcL_{\tilde{\psi}}=[\times \zeta]^*\mcL_{\tilde{\psi}}.
$$
\par
Hence by Lemma~\ref{lm-inductionformula}, the local monodromy at
$\infty$ of $\mcK_{r,0,\bfb}$ is  isomorphic to that of
\begin{multline*}
  \bigoplus_{\zeta_2,\zeta_3,\zeta_4\in \mmu_k} \alpha_*\Bigl(
  \mcL_{\chi_2^{k+1}(r_1t)}\otimes\mcL_{\tilde{\psi}(r_1t)} 
  \otimes
  \mcL_{\chi_2^{k+1}(r_2t)}\otimes\mcL_{\tilde{\psi}(\zeta_2r_2t)}
  \otimes
  \\
  \mcL_{\chi_2^{k+1}(r_3t)}\otimes\mcL_{\tilde{\psi}(-\zeta_3r_3t)}
  \otimes
  \mcL_{\chi_2^{k+1}(r_4t)}\otimes\mcL_{\tilde{\psi}(-\zeta_4r_4t)}
  \Bigr)
  \\
  \simeq \bigoplus_{\zeta_2,\zeta_3,\zeta_4\in \mmu_k}
  \alpha_*\Bigl(\mcL_{\chi_2^{k+1}}(r_1r_2r_3r_4 t^4)
  \mcL_{\tilde{\psi}}(r_1t+\zeta_2r_2t-\zeta_3r_3t-\zeta_4r_4t)\Bigr),
\end{multline*}

If $k$ is odd, then $\chi_2^{k+1}$ is then trivial. Otherwise
$\chi_2^{k+1}=\chi_2$. Since $r_1,\ldots,r_4$ are nonvanishing on $U$,
the sheaf $\mcL_{\chi_2}(r_1r_2r_3r_4 t^4)$ is lisse on
$U \times \Gg_m \subseteq U \times \Aa^1$. By assumption, we have
$r_1r_2r_3r_4=cg^2$, so $r_1r_2r_3r_4 t^4=c(gt^2)^2$, and thus
$\mcL_{\chi_2}(r_1r_2r_3r_4 t^4)$ is geometrically trivial on
$U \times \Gg_m$.  Therefore we may ignore that term, and we
obtain~(\ref{eq-local-goal}).
\end{proof}

The following lemma about Kloosterman sheaves will prove useful to
compute the monodromy at $r=\infty$ of sum-product sheaves.

\begin{lemma}\label{lm-kloosterman-invariance} 
  Let $R$ be a strictly Henselian regular local ring of characteristic
  $q>2$ with fraction field $K$ and maximal ideal $\mathfrak{m}$. Assume
  that $q\nmid k$.
\par
\emph{(1)} If $a\in R-\{0\}$ and $b\in\mathfrak{m}$, then we have
$$
a^*\HYPK_k\simeq (a+ab)^*\HYPK_k,
$$
where we view $a$ and $a+ab$ as maps $\Spec(R)\lra \Aa^1_{\bFq}$. 
\par
\emph{(2)} If $a\in K^{\times}$ is such that $a^{-1}\in\mathfrak{m}$,
and $b\in R$, then we have
$$
a^*\HYPK_k\simeq (a+b)^*\HYPK_k
$$ 
where we view $a$ and $a+b$ as maps
$\Spec(R) \lra \Pp^1_{\bFq}$.
\end{lemma} 

\begin{proof} 
(1) There are two cases: either $a\in \mathfrak{m}$ or $a\in R^{\times}$.

If $a\in\mathfrak{m}$, we first observe that as $1+b \in R^\times$,
the ideals $(a)$ and $(a+ab)$ are the same, and hence
$$
Z=a^{-1}(\{0\})=(a+ab)^{-1}(\{0\})\subset \Spec(R).
$$
Let $U$ be the open complement of $Z$ in $\Spec(R)$. Let $j$ be the
open immersion $U \to \Spec R$. As $\HYPK_k$ is zero at $0$ according
to our definition, both $a^* \HYPK_k$ and $(a+ab)^* \HYPK_k$ are zero
on $Z$. Thus $a^* \HYPK_k$ is the extension by zero of
$j^* a^* \HYPK_k$, and $(a+ab)^* \HYPK_k$ is the extension by zero of
$j^* (a+ab)^* \HYPK_k$. So it is sufficient to check that
$j^* a^* \HYPK_k$ is isomorphic to $j^* (a+ab)^* \HYPK_k$ on $U$, and
then applying $j_!$ gives the isomorphism on $\Spec R$.

As $\HYPK_k$ is lisse on $\Gg_m$, the sheaves $j^*a^* \HYPK_k$ and
$j^* (a+ab)^* \HYPK_k$ are both lisse on $U$.  We next check that
these two sheaves are isomorphic as lisse sheaves on $U$, or
equivalently that they are isomorphic as representations of
$\pi_1(U)$.
\par
First, $a$ and $a+ab$, viewed as maps from $\Spec(R)$ to
$\Aa^1_{\bFq}$, both factor through the \'etale local ring at $0$. So
on the complement $U$ of the inverse image of zero, both maps factor
through the generic point.
\par
By Proposition~\ref{pr-kl}(4), the local monodromy representation
associated to $\HYPK_k$ at $0$ is tame, hence it factors through the tame
fundamental group
$$
\pi_1^t\simeq \lim_{\substack{\leftarrow\\(n,q)=1}}\mmu_n(\bFq),
$$
(see, e.g.,~\cite[Examples I.5.2(c)]{milne}) corresponding to
coverings obtained by adjoining $n$-th roots of the coordinate with
$(n,q)=1$.  
To show that $a^*\HYPK_k$ and $(a+ab)^*\HYPK_k$ are isomorphic on $U$,
it is therefore enough by the Galois correspondence to prove that, for
any $n$ with $(n,q)=1$, the pullbacks under $a$ and $a+ab$ of the
covers obtained by $n$-th roots of the coordinate are isomorphic. But
$1+b$ is a unit and $R$ is a strict Henselian local ring, so that $R$
contains an $n$-th root of $1+b$, and the equation
$$ 
(a+ab)^{n} = a^{1/n} (1+b)^{1/n}
$$
gives such an isomorphism.
\par
On the other hand, if $a\in R^{\times}$, then $a+ab\in
R^{\times}$.  Hence both $a$ and $a+ab$, as maps from
$\Spec(R)$ to $\Aa^1_{\bFq}$, send the special point to a point $y\in
\Gg_m$.  Therefore the pullbacks $a^*\HYPK_k$ and
$(a+ab)^*\HYPK_k$ are both locally constant on $\Spec
(R)$, hence correspond to representations of $\pi_1 (\Spec
(R))$. These are all trivial since $\pi_1(\Spec
(R))$ is trivial for
$R$ strictly Henselian (see, e.g.,~\cite[Ex. I.5.2(b)]{milne}), and
since $a^*\HYPK_k$ and
$(a+ab)^*\HYPK_k$ have the same rank, they are isomorphic.
\par
(2) Assume now that $a^{-1}\in\mathfrak{m}$. Then
$$
u=\frac{a+b}{a} = 1 + \frac{b}{a}\in R^{\times},
$$
and hence $(a+b)^{-1}=u^{-1}a^{-1}\in \mathfrak{m}$. So both
$a$ and $a+b$ (now viewed as maps $\Spec(R)\lra
\Pp^1_{\bFq}$) send the special point of
$\Spec(R)$ to
$\infty\in\Pp^1$.  Furthermore the inverse image $Z\subset
\Spec(R)$ of $\infty \in
\Pp^1_{\bFq}$ is the same under both maps, since multiplying by a unit
does not change whether a function is infinite at a point. Because the
sheaves $a^*\HYPK_k$ and $(a+b)^*\HYPK_k$ are $0$ on
$Z$ and lisse on the complement
$U=\Spec(R)-Z$, they are both the extensions by zero of their
restrictions to
$U$, so it is enough to check that they are isomorphic on
$U$ as lisse sheaves, or as representations of the fundamental group
$\pi_1(U)$.
\par
As representations of the fundamental group, both sheaves are
pullbacks of the local monodromy representation of $\HYPK_k$.  By
Lemma~\ref{lm-kl-infty}, the local monodromy of $\HYPK_k$ at $\infty$
is isomorphic to that of the sheaf
$$
[x\mapsto x^k]_*(\mcL_{\chi_2^{k+1}}\otimes\mcL_{\tilde{\psi}}),
$$
where
$\tilde{\psi}(x)=\psi(kx)$. It is therefore enough to show that the
pullbacks of this sheaf along $a$ and $a+b$ are isomorphic.
\par
Let $C_a=\Spec(R[a^{-1/k}])$ and $C_{a+b}=\Spec(R[(a+b)^{-1/k}])$,
viewed as \'etale covers of $U$. Then, because
$u=(a+b)/a\in R^{\times}$ is a unit congruent to $1$ modulo
$\mathfrak{m}$ (and $k$ is coprime to $q$), there exists a $k$-th root
(say $v$) of $u$ in $R^{\times}$ which is congruent to $1$ modulo
$\mathfrak{m}$. The two covers are isomorphic via the map
$$
C_{a+b}\lra C_{a}
$$
induced by $y\mapsto vy$. Let $f\,:\, C_a\lra U$ be the
covering map. We have then
\begin{align*}
  a^*([x\mapsto x^k]_*(\mcL_{\chi_2^{k+1}}\otimes\mcL_{\tilde{\psi}}))
  &\simeq
    f_*\Bigl((a^{1/k})^*\sheaf{L}_{\chi_2^{k+1}}
    \otimes(a^{1/k})^*\mcL_{\tilde{\psi}}\Bigr),\\
  (a+b)^*([x\mapsto x^k]_*(\mcL_{\chi_2^{k+1}}\otimes\mcL_{\tilde{\psi}}))
  &\simeq
    f_*\Bigl((va^{1/k})^*\sheaf{L}_{\chi_2^{k+1}}\otimes
    (va^{1/k})^*\mcL_{\tilde{\psi}}\Bigr).
\end{align*}
It is therefore sufficient to prove that
$$
(a^{1/k})^*\sheaf{L}_{\chi_2^{k+1}}\otimes(a^{1/k})^*\mcL_{\tilde{\psi}}
\simeq (va^{1/k})^*\sheaf{L}_{\chi_2^{k+1}}\otimes
(va^{1/k})^*\mcL_{\tilde{\psi}}.
$$
Indeed, since $q\neq 2$ and $v$ is a unit, we have first
$$
(a^{1/k})^*\sheaf{L}_{\chi_2^{k+1}}\simeq
(va^{1/k})^*\sheaf{L}_{\chi_2^{k+1}}
$$
since $v$ is a unit. Furthermore, since
$v-1=w$ belongs to $\mathfrak{m}$, we get
$$
(va^{1/k})^*\mcL_{\tilde{\psi}} \simeq (a^{1/k})^*\mcL_{\tilde{\psi}}
\otimes (wa^{1/k})^*\mcL_{\tilde{\psi}}.
$$
Now we claim that the second factor is trivial on $R[a^{-1/k}]$, which
concludes the proof. Indeed, $w$ is in the ideal generated by $a^{-1}$
(by the power series $v=1+b k^{-1} a^{-1} + \cdots$), so $wa^{1/k}$ is
in the ideal generated by $a^{-(k-1)/k}$ and thus in the maximal ideal
of $R[a^{-1/k}]$. Hence it sends $\Spec R[a^{-1/k}]$ to a neighborhood
of $0$ in $\Aa^1_{\bFq}$, where $\mcL_{\tilde{\psi}}$ is lisse and
hence locally trivial, so the pullback
$(wa^{1/k})^*\mcL_{\tilde{\psi}}$ is trivial.
\end{proof}

We will need some simple facts about hypergeometric sheaves in the
sense of Katz~\cite{ESDE}, more precisely about a particular
hypergeometric sheaf.

\begin{definition}\label{def-hypergeo}
  For $k\geq 2$ an integer, we denote by $\mcH_{k-1}$ the middle-extension to
  $\Aa^1$ with coordinate $\xi$ of the $\ell$-adic sheaf on ${\Gm}$
  given by
$$
\mcH_{k-1}=[\xi\mapsto \xi^{-1}]^*
j^*\ft_{\psi}(\mcL_{\tilde{\psi}}(x^{1/k})),
$$ 
where $j\,:\, \Gg_m\lra \Aa^1$ is the open immersion and we recall
that $\tilde\psi(x)=\psi(kx)$.
\end{definition}

It is important for later purpose to note the following lemma:

\begin{lemma}\label{lm-hypergeo} 
  The sheaf $\mcH_{k-1}$ is a multiplicative translate of a
  hypergeometric sheaf of type $(k-1,0)$ in the sense of Katz. More
  precisely, it is geometrically isomorphic to
$$
\mathrm{Hyp}_{(-1)^k}(!,\bar{\psi};\{\chi|\chi^k=1,\
\chi\not=1\};\emptyset),
$$
with the notation of~\cite[8.2.2, 8.2.13]{ESDE}. The inertia
representation of $\mcH_{k-1}$ at infinity is absolutely irreducible.
\end{lemma}

We thank the referee for giving a proof that is simpler than our
original. 

\begin{proof} 
  Since both $\mcH_{k-1}$ and hypergeometric sheaves are
  middle-extension sheaves (recall that $k\geq 2$), it is enough to
  prove the isomorphism after restriction to $\Gg_m$.  We compute
  \begin{align*}
    j^*\mcH_{k-1}
    &= [\xi\mapsto \xi^{-1}]^*j^*\mathrm{FT}_\psi(\mcL_{\tilde\psi}(x^{1/k}))\\ 	
    &\simeq  [\xi\mapsto
      \xi^{-1}]^*j^*\mathrm{FT}_\psi(j_*\mathrm{Hyp}_1(!,\psi;\{\chi|\chi^k=1\};\emptyset))
      \quad\quad \text{\cite[5.6.2]{GKM}}\\
    &\simeq  [\xi\mapsto
      \xi^{-1}]^*\mathrm{Hyp}_{(-1)^k}(!,\psi;\emptyset;\{\chi\not=1,\
      \chi^k=1\})) 
      \quad\quad \text{\cite[5.6.2]{GKM}}\\
    &\simeq  \mathrm{Hyp}_{(-1)^k}(!,\ov\psi;\{\chi\not=1,\
      \chi^k=1\};\emptyset)
  \end{align*}
where $\simeq$ always denotes geometric isomorphisms.
\par
The last assertion now follows from~\cite[Th. 8.4.2 (6)]{ESDE}, which
shows that the inertia representation at $\infty$ is of dimension
$k-1$ will unique break $1/(k-1)$ and~\cite[Prop. 1.14]{GKM}, which
shows that such a representation of the inertia group at $\infty$ is
absolutely irreducible.
\end{proof}

We will need some properties of the local monodromy at $\infty$ of
$\mcH_{k-1}$. To state them, we need the following definition.

\begin{definition}\label{def-reparameterization}
  Let $K$ be a local field and let $\sigma$ be an automorphism of
  $K$. Let $n\geq 1$ be an integer and let $\pi$ be a uniformizer of
  $K$. We say that $\sigma$ is \emph{a reparameterization of order
    $n$} if $\sigma(\pi)$ is a uniformizer of $K$ such that
$$
\sigma(\pi)\equiv \pi\mods{\pi^n}.
$$
\end{definition}

Note that since an order $n$ reparameterization acts on $K$, it also
defines an outer automorphism of the Galois group of $K$: each
extension $\bar{\sigma}$ of $\sigma$ to a separable closure $\bar{K}$
of $K$ gives an automorphism of $\Gal(\bar{K}/K)$, and the ambiguity
in the possible choices of this extension amounts to conjugating
$\bar{\sigma}$ with an element of $\Gal(\bar{K}/K)$, so that the
corresponding outer automorphism of the Galois group is well
defined. This outer automorphism defines an action of $\sigma$ on the
set of isomorphism classes of representations of the Galois
group. More abstractly, $\sigma$ defines an automorphism of the
category of finite \'{e}tale covers of $\Spec(K)$ by pullback, and
hence acts on the category of \'{e}tale sheaves on $\Spec(K)$, which
is equivalent to the category of Galois representations.

\begin{lemma}\label{lm-hypergeo-2}
  Assume that $q>k\geq 2$. 
\par
\emph{(1)} The local monodromy representation at infinity of
$\mcH_{k-1}$ is invariant under reparameterizations of order $2$.
\par
\emph{(2)} The local monodromy representation at infinity of
$\mcH_{k-1}$, restricted to the wild inertia group, is a direct sum of
pairwise non-isomorphic characters with multiplicity $1$. The tame
inertia group acts transitively on these characters.
\par
\emph{(3)} Let $\alpha_1$, $\alpha_2$ be elements of an algebraically
closed extension of $\Fq$ such that the wild local monodromy
representation at infinity of $[\times \alpha_1]^*\mcH_{k-1}$ and
$[\times \alpha_2]^*\mcH_{k-1}$ have a common irreducible
component. Then $\alpha_1=\alpha_2$.
\end{lemma}

\begin{proof} 
  The integer $q$ is coprime with $2(k-1)$ since $q>k\geq 2$. By
  \cite[Theorem 0.1, (iii)]{Fu} (which is more precise) we derive
  isomorphisms of $I(\infty)$-representations
  \begin{align}
    \mcH_{k-1|I(\infty)}
    &\simeq
      \ft_\psi(\mcL_{\tilde\psi}(x^{1/k}))_{|I(0)}
      \notag  \\
    &\simeq
      \ft_{\psi}\text{loc}(\infty, 0)
      ([t\mapsto t^k]_*\mcL_{\tilde\psi}) \simeq
      [t\mapsto -t^{k-1}]_*(\mcL_{\psi((k-1)t)}\otimes\mcL_{\chi_2})
\label{eq-local-hk}
\end{align}
where $\ft_{\psi}\text{loc}(\cdot,\cdot)$ denotes Laumon's local
Fourier transform functors (see, e.g,~\cite[7.4]{ESDE}).
\par
To prove (1), it is therefore enough to prove that for any additive
character $\eta$ and any multiplicative character $\chi$, the local
monodromy representation at $\infty$ of
$[t\mapsto t^{k-1}]_*(\sheaf{L}_{\eta}\otimes\sheaf{L}_{\chi})$ is
invariant under reparameterizations of order $2$.
\par
Let $V$ denote this representation. Let $R$ be the strict
henselization at $\infty$, let $K$ be its field of fractions and let
$\pi$ be a uniformizer of $R$.  Let $g\,:\, \Spec(K)\lra \Spec(K)$ be
the map corresponding to $t\mapsto t^{k-1}$. A representation obtained
from $V$ by applying a reparameterization of order $2$ is of the form
$\sigma^*V=(\sigma^{-1})_*V$, where $\sigma$ is an automorphism $K$
such that $\sigma(\pi)\equiv \pi\mods{\pi^2}$. We view $\sigma$ and
$\sigma^{-1}$ as automorphisms $\Spec(K)\lra \Spec(K)$.
\par
Let $W=\sheaf{L}_{\eta}\otimes\sheaf{L}_{\chi}$; we have
$V\simeq g_*W$ and hence $(\sigma^{-1})_*V=\tau_* W$ where
$\tau=\sigma^{-1}\circ g$.  There exists an automorphism $\sigma_1$
such that $\tau=g\circ \sigma_1$, and $\sigma_1$ is a
reparameterization of order $k$. We can see this in coordinates by
solving the equation
$$\sigma_1(t)^{k-1} = \sigma^{-1} ( t^{k-1}) = t^{k-1} + a_1
t^{2(k-1)} + \dots $$ with
$$\sigma_1(t) = t + \frac{a_1}{k-1} t^k + \dots $$ Thus
$\sigma^*V\simeq g_*(\sigma_{1,*}W)$, and in particular, we obtain
$\sigma^*V\simeq V$, provided $W$ is invariant under
reparameterizations of order $k$.  In fact, we will show that both
$\sheaf{L}_{\eta}$ and $\sheaf{L}_{\chi}$ are invariant under any
reparameterization $\sigma_1$ of order $k\geq 2$, which will be
enough.
\par
For multiplicative characters, this amounts to saying that for $d$
coprime to $q$, the covering $\Spec(K(\pi^{-1/d}))\lra \Spec(K)$ is
invariant under $\sigma_1$, which is clear because if we write
$\sigma_1(\pi)=\pi+b\pi^2$ for some $b\in R$, we get
$$
\sigma_1(\pi)^{-1/d}=\pi^{-1/d}(1+b\pi)^{-1/d},
$$
and $(1+b\pi)^{-1/d}\in K$. For additive characters, this amounts to
proving that the Artin-Schreier covering with equation
$y^q-y=\pi^{-1}$ is invariant, and this follows because the equation
$$
z^q-z=\frac{1}{\sigma_1(\pi)}-\frac{1}{\pi}=-\frac{b}{1+b\pi}
$$
is solvable in $K$.
\par
(2) By~(\ref{eq-local-hk}), the local wild monodromy representation of
$\mcH_{k-1}$ at $\infty$ is isomorphic to
$$
[t\mapsto -t^{k-1}]_*(\mcL_{\psi((k-1)t)}).
$$
It is equivalent to study this after pulling back by any tame
cover. In particular, after pulling back along the map
$t \mapsto t^{k-1}$, we have to deal with
\begin{equation}\label{eq-hk}
\bigoplus_{\xi^{k-1}=1} \mcL_{\psi ( \xi (k-1) t)},
\end{equation}
which is indeed a sum of one-dimensional characters.  They are
pairwise non-isomorphic (if we have, say, an isomorphism
$\mcL_{\psi (\xi_1 (k-1)t )} \simeq \mcL_{\psi (\xi_2 (k-1)t )}$ as
representations of the wild inertia group, then
$\mcL_{\psi ((\xi_1-\xi_2) (k-1)t )}$ is tamely ramified, which means
that $\xi_1=\xi_2$ since otherwise
$\mcL_{\psi ((\xi_1-\xi_2) (k-1)t )}$ is a non-trivial additive
character sheaf with Swan conductor $1$).

Since $\mathcal H_{k-1}$ is an irreducible representation of the full
inertia group at infinity (Lemma~\ref{lm-hypergeo}), the tame inertia
group acts transitively by conjugation on the set of characters
in~(\ref{eq-hk}) (the direct sum of any subset of the characters that
is stable under the tame inertia group would define an
inertia-invariant subspace).
\par
(3) Let $L/\Fq$ be an algebraically closed extension. We use the same
notation $\mathcal{H}_{k-1}$ and $\mcL_{\psi}$ for the sheaves
base-changed to $L$, so that for instance $[\times\alpha]^*\mcH_{k-1}$
and $\mcL_{\psi(\beta t)}$ are well-defined for $\alpha$ and
$\beta\in L^{\times}$.
\par
Adding a multiplicative shift to the computation of (2), the pullback
along $[t\mapsto -t^{k-1}]$ of the local wild monodromy representation
of $[\times\alpha]^*\mcH_{k-1}$ at $\infty$ is isomorphic to
$$
\bigoplus_{\beta^{k-1}=\alpha}\mcL_{\psi((k-1)\beta t)}.
$$
If the local wild monodromy representations of
$[\times\alpha_1]^*\mcH_{k-1}$ and $[\times\alpha_2]^*\mcH_{k-1}$ at
$\infty$ have a common irreducible component, then one of the additive
characters appearing in one of the two sums must also appear in the
other, so there exists $\beta$ such that
$\alpha_1=\beta^{k-1}=\alpha_2$.
\end{proof}

The following lemma is quite standard but we include a proof for lack
of a suitable reference.

\begin{lemma}\label{lm-weight-computation} 
  \emph{(1)} Let $U_{\Fq}$ be a dense open subset of a smooth projective
  curve $C_{\Fq}$ and let $\sheaf{F}$ be an $\ell$-adic sheaf on
  $C$. Assume that $\sheaf{F}$ is lisse and pure of weight $w$ on $U$,
  that it has no punctual sections, and that, viewed as a
  representation of the geometric fundamental group of $U$, it has no
  trivial subrepresentation.

  Then the subspace of weight $<w+1$ of $H^1(C_{\bFq}, \sheaf{F})$ is
  equal to
$$
\bigoplus_{x\in C-U}(\sheaf{F}_{\ov\eta}^{I_x}/\sheaf{F}_{\ov x}),
$$
where $I_x$ is the inertia group at $x$ and $\sheaf{F}_{\ov\eta}$ is
the stalk at the geometric generic point of $\sheaf{F}$.
\par
\emph{(2)} Let $\pi: C \to X$ be a smooth projective morphism of
relative dimension $1$ over $\Fq$, and let $\sheaf{F}$ be an
$\ell$-adic sheaf on $C$. Assume that $\sheaf{F}$ is lisse and pure of
weight $w$ on a dense open subset $U\subset C$. Assume that for all
$x\in X$ in some dense open subset, $\sheaf{F}|C_x$ has no punctual
sections and that, when $\sheaf{F}|C_x$ is viewed as a representation
of the geometric fundamental group of $U_x$, it has no trivial
subrepresentation.

On the dense open set where $R^1 \pi_* \sheaf{F}$ is lisse, let
$(R^1 \pi_* \sheaf{F})^{<w+1}$ be the maximal lisse subsheaf of
$R^1\pi_*\sheaf{F}$ of weight $<w+1$. Then for any point $x$ in the
dense open subset where $R^1 \pi_* \sheaf{F}$ is lisse, we have an
isomorphism
$$
(R^1 \pi_* \sheaf{F})^{<w+1}_x= \bigoplus_{y\in
  C_x-U_x}((\sheaf{F}|C_x)_{\ov\eta}^{I_y}/(\sheaf{F}|C_x)_{\ov y})
$$
where $(\sheaf{F}|C_x)_{\ov\eta}$ is the stalk at the geometric
generic point of the restriction of $\sheaf{F}$ to $C_x$.
\end{lemma}

\begin{proof}
  (1) Let $j\,:\, U\lra C$ denote the open immersion.  Because
  $\sheaf{F}$ has no punctual sections, the natural adjunction map
  $\sheaf{F} \to j_* j^* \sheaf{F} $ is injective. Let $\sheaf{G}$ be
  its cokernel.  Then we have a long exact sequence
\begin{equation}\label{eq-exact}
\cdots\lra H^i(C_{\bFq}, \sheaf{F}) \to H^i(C_{\bFq}, j_* j^* \sheaf{F})
\to H^i (C_{\bFq}, \sheaf{G})\lra \cdots
\end{equation}
By assumption on $\sheaf{F}$, we have
$$
H^0(C_{\bFq}, j_* j^* \sheaf{F})= H^0 (U_{\bFq}, j^* \sheaf{F})=0.
$$
Since $\sheaf{G}$ is supported on $C - U$, its cohomology vanishes in
degree above $1$, and hence we deduce a short  exact sequence
$$
0 \to H^0 (C_{\bFq}, \sheaf{G}) \to H^1(C_{\bFq}, \sheaf{F}) \to H^1(C,
j_* j^* \sheaf{F}) \to 0.
$$
\par
Because $j_* j^* \sheaf{F}$ is the middle extension of a lisse sheaf
pure of weight $w$, a result of Deligne implies that its cohomology
group $H^1(C_{\bFq}, j_* j^* \sheaf{F})$ is pure of weight
$w+1$ 
(see~\cite[Exemple 6.2.5(c) and Proposition 6.2.6]{WeilII}). Therefore
the weight $<w+1$ part of $H^1(C_{\bFq}, \sheaf{F})$ is the same as the
weight $<w+1$ part of $H^0(C_{\bFq}, \sheaf{G})$. Since the sheaf
$\sheaf{G}$ is punctual, we have
$$
H^0(C_{\bFq},\sheaf{G})=\bigoplus_{x\in C-U}\sheaf{G}_{\ov x}= \bigoplus_{x\in
  C-U}(j_*j^*\sheaf{F})_{\ov x}/\sheaf{F}_{\ov x}
$$ 
(by definition of $\sheaf{G}$). We also have
$$
(j_*j^*\sheaf{F})_{\ov x}=\sheaf{F}_{\ov\eta}^{I_x},
$$
and~\cite[Lemma 1.8.1]{WeilII} shows that this space is of weight
$\leq w$, so that all of $H^0(C_{\bFq},\sheaf{G})$ is the weight $<w+1$
part of $H^1(C_{\bFq},\sheaf{F})$, as claimed.



(2) Denote again by $j\,:\, U\lra C$ the open embedding. We want to
apply (1) fiber by fiber.  First (since pushforward does not commute
with arbitrary base change), we let $U_1$ denote a dense open subset
of $X$ such that the adjunction map
$$
\sheaf{F}|\pi^{-1}(U_1)\to j_*j^*\sheaf{F}|\pi^{-1}(U_1)
$$
is injective (the existence of such a dense open set follows
from~\cite[Th. Finitude, Th\'eor\`eme 1.9(2)]{sga4h}, applied to the
morphism $j$ over the base $X$).  Let $\sheaf{G}$ be the quotient
sheaf. Then we again take the long exact sequence
$$
\cdots\lra R^{i} \pi_* \sheaf{F} \to R^i \pi_* j_* j^* \sheaf{F} \to
R^i \pi_* \sheaf{G}\lra \cdots
$$
The fiber over any $x\in U_1$ of this exact sequence is the same as
the exact sequence~(\ref{eq-exact}) for the fiber curve $C_x$, again
using~\cite[Th. Finitude, Th\'eor\`eme 1.9(2)]{sga4h}. In particular,
for any point $x' \in U_1$ (closed or not), we have
$$
\bigoplus_{y\in
  C_{x'}-U_{x'}}((\sheaf{F}|C_{x'})_{\ov\eta}^{I_y}/(\sheaf{F}|C_{x'})_{\ov y})
=(R^0 \pi_* \sheaf{G})_{x'}.
$$
Thus (1) shows over any closed point $x' \in U_1$ that the weight
$<w+1$ part of $(R^1\pi_* \sheaf{F})_{x'}$ is the image of
$(R^0 \pi_* \sheaf{G})_{x'}$ in $(R^1\pi_* \sheaf{F})_{x'}$. Over a
possibly smaller dense open set $U_2\subset U_1$ where
$R^0 \pi_* \sheaf{G}$ and $R^1\pi_* \sheaf{F}$ are both lisse, this
implies that the maximal weight $<w+1$ lisse subsheaf of
$R^1\pi_* \sheaf{F}$ is $R^0 \pi_* \sheaf{G}$.  Then for an arbitrary
$x \in U_2$, we have
$$
(R^1 \pi_* \sheaf{F})^{<w+1}_x = (R^0 \pi_* \sheaf{G})_x=
\bigoplus_{y\in
  C_x-U_x}((\sheaf{F}|C_x)_{\ov\eta}^{I_y}/(\sheaf{F}|C_x)_{\ov y}).
$$
\par
If $x$ is the generic point, it is necessarily contained in the dense
open subset $U_2$. If not, we can replace $X$ by the closure of $x$ in
$X$ and apply the same argument, obtaining the same identity (because
the direct sum
$$
\bigoplus_{y\in
  C_x-U_x}((\sheaf{F}|C_x)_{\ov\eta}^{I_y}/(\sheaf{F}|C_x)_{\ov y})
$$
depends only on the fiber over $x$, and the same for
$(R^1 \pi_* \sheaf{F})^{<w+1}_x $, since taking the weight $<{w+1}$
part commutes with restriction to a closed subscheme.)
\end{proof}

The next lemma will be useful to bound in terms of $q$ the degree of
the subvariety $\mcV^{bad}$ for $\lambda=0$, by showing that this
variety is defined over $\Zz[1/\ell]$.

\begin{lemma}\label{lm-surprising-integrality} Let $X$ and $Y$ be
  separated varieties of finite type over $\Zz[1/\ell]$. Let $f: X\to Y$ and
  $g: X\to \Aa^1$ be morphisms.  Let 
  $p_2 : \Gm \times X \to X$ be the second projection.
\par
There exists an $\ell$-adic complex $K$ on $Y$ such that, for any
prime $q\neq \ell$ and any additive character $\psi$ of $\Fq$, we have
$$ 
R (f \circ p_2)_!  \mcL_\psi ( t g) = K | Y_{\Fq} .
$$
\end{lemma}

\begin{proof}
  We denote by $t$ a coordinate on $\Gg_m$.  As
  $R (f \circ p_2)_! = Rf_! R p_{2,!}$, it is sufficient to prove that
  there exists a complex $K'$ on $X$ with
$$
R p_{2,!} \mcL_{\psi(t g)} = K' | X_{\Fq},
$$ 
for all $q\not=\ell$ and all $\psi$, as we can then take $K=Rf_!K'$.
\par
Let $p_2'$ denote the second projection
$\Gg_m \times \Aa^1 \to \Aa^1$.  By the proper base change
theorem, we have
$$
R p_{2,!} \mcL_{\psi(t g)} = g^* R p'_{2,!} \mcL_{\psi ( tx)},
$$
for any $q\not=\ell$ and $\psi$, so it is sufficient to find a complex
$K^*$ on $\Aa^1_{\Zz[1/\ell]}$ with
$$
R p'_{2!} \mcL_{\psi (tx)} = K^* | \Aa^1_{\Fq}
$$
for all $q\not=\ell$ and all $\psi$, and to define $K'=g^*K^*$.
\par
By the above reduction we may assume that $X=\Aa^1_{\Zz[1/\ell]}$ and
write $p_2$ for $p'_2$. Let $j: \Gm \to \Aa^1$ be the open immersion
and $i: \{0\} \to \Aa^1$ the complementary closed immersion. Then
$R p_{2,!} \mcL_{\psi (tx)} $ is the Fourier transform of $j_! \bQl$
(as extension by zero commutes with pullback and tensor product). The
existence of an $\ell$-adic complex on $\Aa^1_{\Zz[1/\ell]}$ that
specializes to $\mathrm{FT}_\psi j_! \bQl=Rp_{2,!}\mcL_{\psi(tx)}$ in
each positive characteristic $q\not=\ell$ is a special case of
Laumon's homogeneous Fourier transform
(see~\cite[Th. 2.2]{laumonhomog}). In this special case, L. Fu
(see~\cite[Lemma 3.2]{Fu2}) showed that we can take the complex to be
$j_* \bQl$.
\end{proof}

Finally, we can already prove the last part of
Theorem~\ref{thmirreducibility}.

\begin{proposition}\label{pr-betti-bound}
  For all
  $\lambda_1,\lambda_2\in\Fq, \bfb_1,\bfb_2\not\in\mcV^\Delta(\Fq)$,
  the dimensions of the stalks of the sheaf $\mcR$ and the dimensions
  of the cohomology groups
$$
H^i_c(\Aa_{\bFq}, \mcR_{\lambda_1,\bfb_1}),\quad\quad H^i_c (\Aa_{\bFq},
\mcR_{\lambda_1,\bfb_1} \otimes \mcR_{\lambda_2,\bfb_2})
$$
are all bounded in terms of $k$ only.
\end{proposition}

\begin{proof}
  We deal with the second of these cohomology groups.  Fix
  $\lambda_1,\lambda_2$ in $\Fq$,
  $\bfb_1,\bfb_2\not\in\mcV^\Delta(\Fq)$.
By construction of $\mcR$ and by interpreting
sheaf-theoretically the definition of the hyper-Kloosterman sums,
there exists an affine variety $V_{\Zz}$ and maps
$f\,:\,V\to \Aa^1_{\Zz}$ and $g\,:\, V\to \Aa^1_{\Zz}$ such that, for any
prime $q$, we have
$$
(\mcR_{\lambda_1,\bfb_1}\otimes\mcR_{\lambda_2,\bfb_2})|\Aa^1_{\Fq}=
Rf_!g^*\sheaf{L}_{\psi}[2]
$$
(see also Lemma~\ref{lm-r-integrality} below for this construction).
\par
By the Leray spectral sequence, it is enough to bound the sum of Betti
numbers
$$
\sum_{i} \dim H^i_c(\tilde{V}_{\bFq},\sheaf{L}_{\psi(g)})
$$
where $\tilde{V}$ is the inverse image in $V$ of either a line or a
plane.  Since $(V,f,g)$ are defined over $\Zz$, a suitable bound is
given by the estimates of Bombieri and Katz for sums of Betti numbers
(see the version of Katz in~\cite[Theorem 12]{KatzBetti}).
\par
A similar argument applies to $H^i_c(\Aa^1_{\bFq},\mcR_{\lambda,\bfb})$.
\end{proof}

Finally, we need a lemma on inertia groups that is probably well-known
but for which we do not know a convenient reference.

\begin{lemma}\label{lm-surj-inertia}
  Let $\pi\colon \Aa^5_{\bFq}\to \Aa^4_{\bFq}$ be the projection on
  the first four coordinates. Let
  $\bar{\pi}\colon \Pp^4\times\Pp^1\to \Pp^4$ be the analogue
  projection. For any divisor $D$ in $\Pp^4$, the induced homomorphism
  from the inertia group at the generic point of $\bar{\pi}^{-1}(D)$
  to the inertia group at the generic point of $D$ is surjective.
\end{lemma}

\begin{proof}
  Let $R$ (resp. $R'$) be the étale local ring of $D$ (resp. of
  $\bar{\pi}^{-1}(D)$) at its generic point, $K$ (resp. $K'$) its
  field of fractions. Then the inertia group $I_D$ of $D$ is the
  Galois group of $K$ and the inertia group $I_{\bar{\pi}^{-1}(D)}$ of
  $\bar{\pi}^{-1}(D)$ is the Galois group of $K'$.  If the
  homomorphism $I_{\bar{\pi}^{-1}(D)}\to I_D$ of profinite groups is
  not surjective, then its image is contained in some proper open
  subgroup of $I_D$. By the Galois correspondence, this means that
  there exists some finite \'{e}tale covering $E \to K$ without a
  section whose pullback to $K'$ admits a section.
  
  We will show that every finite \'{e}tale covering $E\to \Spec(K)$
  whose pullback to $K'$ admits a section already has a section over
  $K$, implying by contradiction that the homomorphism is surjective,
  as claimed.
  
  Let $E\to \Spec(K)$ be such a covering, and $s'$ a section of the
  pullback to $K'$.  The section $s'$ is defined over
  $K'= R'[t^{-1}]$, where $t$ is a uniformizer of $R$ (and hence also
  a uniformizer of $R'$). Because $R'$ is the \'{e}tale local ring of
  the generic point of $\bar{\pi}^{-1}(D)$, it is the \'{e}tale local
  ring of the generic point $\eta$ of the special fiber $\Aa^1_R$.
  Because the section $s'$ is necessarily defined over some finitely
  generated subring of $R'[t^{-1}]$, and $R'$ is the limit of the
  rings of functions on all \'{e}tale neighborhoods of $\eta$, the
  section $s'$ is defined over the ring of functions on some \'{e}tale
  neighborhood $X \to \Aa^1_R$ of $\eta$, after inverting $t$.  The
  equations for $s'$ over this ring describe a map
  $s\colon X -\{t =0 \} \to E$ over $\Spec(R)$.
  
  The image of $X$ in $\Aa^1_R$ contains a Zariski neighborhood of
  $\eta$, which contains all but finitely many closed points of the
  special fiber. Hence it contains the image of some section of the
  projection $\pi: \Aa^1_R \to \Spec(R)$. Let $Y$ be the pullback of
  $X$ along that section. Then there is a morphism $Y-\{t=0\} \to E$,
  and $Y$ is an \'{e}tale cover of $\Spec(R)$, so it contains a copy
  of $\Spec(R)$, hence there is a map $\Spec(R) - \{t=0\} \to E$,
  which means that $E$ admits a section over
  $\Spec(K)= \Spec(R) -\{t=0\}$.

\end{proof}

\subsection{Irreducibility of sum-product sheaves for $\lambda=0$}
\label{sec-zero}

We now begin the study of sum-product sheaves in the case
$\lambda=0$. \emph{We always assume that $q>k$.}
\par
We denote by $\mcR_{\lambda=0}$ (resp. $\mcR^*_{\lambda=0}$) the
pullback $i^*\mcR$ (resp. $i^*\mcR^*$) for the inclusion $i$ of
$\Aa^5$ in $\Aa^6$ such that $i(r,\bfb)=(r,0,\bfb)$, and similarly we
define $\mcK_{\lambda=0}$.

The main result of this Section establishes that for $q$ large enough,
and for generic values of $\bfb$ (ie. outside some proper subvariety
$\mcV^{bad}\subset\Aa^4_\Fq$), the sheaf $\mcR^*_{\lambda=0,\bfb}$ is
geometrically irreducible.
\par
The strategy is as follows:
\begin{enumerate}
\item\label{Zkey} A key observation (Lemma \ref{lm-r-integrality}) is
  that, by homogeneity, $\mcR^*_{\lambda=0}$ is defined over
  $\Zz[1/\ell]$. In particular this implies that for $q$ large enough,
  the sheaf $\mcR^*_{\lambda=0}$ is not wildly ramified.
\item In Proposition~\ref{pr-new-irreducibility}, we use this fact
  together with the diophantine criterion of irreducibility
  (Lemma~\ref{diophantine}) and the the mean square average asymptotic
  formula of Proposition~\ref{corRrbcorrelation} to prove that
  $\mcR^*_{\lambda=0,\bfb}$ is generically geometrically irreducible.
\item In Proposition \ref{pr-degree-bound}, we conclude and show,
  using \eqref{Zkey}, that $\mcV^{bad}$ is in fact defined over
  $\Zz[1/\ell]$.
\end{enumerate}

\begin{lemma}\label{lm-mcr-lisse} 
Let $Z$ be the subvariety of $\Aa^5\times\Aa^4$ defined by 
\begin{equation}\label{eq-var-z}
\{(r,\bfb,x_1,x_2,x_3,x_4)\in\Aa^5\times\Aa^4\,\mid\, x_i^k=(r+b_i),\
i=1,\cdots,4\}
\end{equation}
and let $\tilde{Z}\subset\Aa^5$ be the image of the subvariety of $Z$
defined by the equation $x_1+x_2-x_3-x_4=0$ under the projection
$$
(r,\bfb,x_1,x_2,x_3,x_4)\in Z\mapsto (r,\bfb)\in\Aa^5.
$$
\par
\emph{(1)} The image $\tilde{Z}$ is closed and irreducible.
\par
\emph{(2)} Let $U$ be the dense open complement of the union of
$\tilde{Z}$ and of the divisors given by the equations $r=-b_i$ in
$\Aa^5$.  The sheaf $\mcR_{\lambda=0}$ is lisse on
$U$.
\par
\emph{(3)} On any dense open subset $V\subset U$ where
$\mcR^*_{\lambda=0}$ is lisse, the monodromy representation of
$\mcR^*_{\lambda=0}$ factors through $\pi_1(U)$.
\end{lemma}

\begin{proof} 
  (1) The projection $Z\to \Aa^5$ is finite because $Z$ is defined by
  adjoining the coordinates $x_1,x_2,x_3,x_4$ to $\Aa^5$, and each
  satisfies a monic polynomial equation. Thus the closed subvariety of
  $Z$ defined by the equation $x_1+x_2-x_3-x_4=0$ is also finite over
  $\Aa^5$ and hence its image $\tilde{Z}$ is closed.  Moreover,
  $\tilde{Z}$ is the projection of the subvariety of $\Aa^9$ with
  equations
$$
\begin{cases}
  x_i^k = r+b_i,\quad 1\leq i\leq 4\\
  x_1+x_2-x_3-x_4=0,
\end{cases}
$$
and hence this subvariety is isomorphic to the divisor in $\Aa^5$ with
coordinates $(x_1,x_2,x_3,x_4,r)$ given by the equation
$x_1+x_2-x_3-x_4=0$. In particular, it is irreducible, and therefore
its image $\tilde{Z}$ is also irreducible.

(2) This will use Deligne's semicontinuity theorem~\cite{LaumonSMF}.
Precisely, as we already observed, the sheaf $\mcK_{\lambda=0}$ is
lisse on the complement of the divisors given by the equations
$r=-b_i$ and $s=0$ in $\Aa^6$.  We compactify the $s$-coordinate by
$\Pp^1$ and work on
 $$
 X=(\Aa^1\times\Pp^1\times \Aa^4)\cap \{(r,s,\bfb)\,\mid\, (r,\bfb)\in
 U\}.
$$
\par
By extending by $0$, we view $\mcK_{\lambda=0}$ as a sheaf on $X$
which is lisse on the complement in $X$ of the divisors $s=0$ and
$s=\infty$ (because $U$ is contained in the complement of the divisors
$r=-b_i$ and thus $X$ is as well). Let
$$
\pi^{(2)} \,:\, X\lra U
$$
denote the projection $(r,s,\bfb)\mapsto (r,\bfb)$. Then $\pi^{(2)}$
is proper and smooth of relative dimension $1$ and
$\mcR_{\lambda=0}|U=R^1\pi^{(2)}_*\mcK$.
\par
Since the restrictions of $\mcK_{\lambda=0}$ to the divisors
$s=\infty$ and $s=0$ are zero, this sheaf is the extension by zero
from the complement of those divisors to the whole space of a lisse
sheaf. Deligne's semicontinuity theorem \cite[Corollary
2.1.2]{LaumonSMF}
implies that the sheaf $\mcR_{\lambda=0}$ is lisse on $U$ if the Swan
conductor is constant on each of these two divisors.

On $s=0$, the Kloosterman sheaf has tame ramification, hence any
tensor product of Kloosterman sheaves has tame ramification. Thus
$\mcK_{\lambda=0}$ has tame ramification at $s=0$, and in particular
the Swan conductor is zero, which is constant.
\par
On the other hand, Lemma \ref{lm-local-product} gives a formula for
the local monodromy representation at $s=\infty$ as a sum of
pushforward representations from the tame covering $x\mapsto x^k$.
Since the Swan conductor is additive and since the Swan conductor is
invariant under pushforward by a tame covering (see,
e.g.,~\cite[1.13.2]{GKM}), it follows that
\begin{multline*}
  \swan_{\infty}(\mcK_{r,\lambda=0,\bfb})=\\
  \sum_{\zeta_2,\zeta_3,\zeta_4\in\mmu_k}
  \swan_{\infty}\Bigl(\mcL_\psi \left( \left((r+b_1)^{1/k}+ \zeta_2
      (r+b_2)^{1/k} - \zeta_3 (r+b_3)^{1/k} - \zeta_4
      (r+b_4)^{1/k}\right) t \right) \Bigr)= k^3,
\end{multline*}
by definition of $U$, since the Swan conductor of $\mcL_{\psi(a t)}$
is $1$ for $a\not=0$. This is constant and therefore we deduce that
$\mcR_{\lambda=0}$ is lisse on $U$.
\par
(3) The restriction $\mcR^*_{\lambda=0}|V$ is a quotient of $\mcR|V$,
and both sheaves are lisse on $V$; since the monodromy representation
of $\mcR|V$ factors through $\pi_1(U)$, the same holds for $\mcR^*|V$.
\end{proof}

We can now deduce the main result of this section. We first show that
$\mcR_{\lambda=0}$ is defined over $\Zz$.  Intuitively, this is
because its trace function is independent of the choice of additive
character $\psi$ used in the definition of the Kloosterman
sheaf. Indeed, let $\psi'(x)=\psi(\lambda x)$, for some
$\lambda\in\Fqt$, be any non-trivial additive character of $\Fq$ and
let
$$
\Kl_{k,\psi'}(x)=q^{-\frac{k-1}2}({\psi'}\star \cdots\star
{\psi'})(x)=\Kl_{k}(\lambda^k x)
$$ 
be the Kloosterman sums defined using $\psi'$ instead of $\psi$. We
have then, with obvious notation, the equality
\begin{align*}
  \bfR_{\psi'}(r,0,\bfb)
  &=\sum_{s}\prod_{i=1}^2 \Kl_{k,\psi'}( s
    (r+b_i )) \overline{\Kl_{k,\psi'}(s (r+b_{i+2}))}\\
  &=\sum_{s}\prod_{i=1}^2 \Kl_{k}( \lambda^ks
    (r+b_i )) \overline{\Kl_{k}(\lambda^k s (r+b_{i+2}))}=\bfR(r,0,\bfb)	
\end{align*}
by making the change of variable $s\mapsto \lambda^ks$, so that
$(r,\bfb)\mapsto \bfR(r,0,\bfb)$ does not depend on the choice of
$\psi$.
\par
The following lemma is a geometric incarnation of this simple
computation:

\begin{lemma}\label{lm-r-integrality} 
  For any prime $\ell$, there exists an $\ell$-adic sheaf
  $\mcR^{univ}$ on $\Aa^5_{\Zz[1/\ell]}$ such that, for any prime
  $q\not=\ell$, 
  we have
$$
 \mcR^{univ} | \Aa^5_{\Fq} = \mcR_{\lambda=0},
$$
where $\mcR_{\lambda=0}$ is defined using the Kloosterman sheaf
$\HYPK_{\psi,k}$ for \emph{any} non-trivial additive character $\psi$
of $\Fq$.
\end{lemma}

\begin{proof} 
  Let $X_1\subset \Gg_m^{k+1}$ be the subvariety over $\Zz$ with
  equation
$$
x_1\cdots x_k=t
$$
and
$$
f_1\,:\, X_1\lra \Aa^1
$$ 
the projection $(x_1,\ldots, x_k,t)\mapsto t$. For any prime
$q\not=\ell$ and any $\psi$, we then have an isomorphism
$$
\HYPK_{k,\psi}\simeq Rf^{k-1}_{1,!}\sheaf{L}_{\psi}(x_1+\cdots +x_k)
$$
of sheaves on $\Aa^1_{\Fq}$ (up to a Tate twist). Let $X_2$ be the
variety in $\Gg_m^{4k}\times\Aa^6$ (over $\Zz[1/\ell]$) defined by the
equations
$$
\prod_{j=1}^kx_{i,j} = s(r+b_i),\quad\quad 
1\leq i\leq 4,
$$
and $f_2\,:\, X_2\lra \Aa^5$ the projection 
$$
f_2(x_{1,1},\ldots, x_{4,k},r,s,\bfb)= (r,\bfb).
$$
By definition, it follows that for all $q\not=\ell$, we have
$$
\mcR_{\lambda=0}=Rf^{4k-3}_{2,!}\sheaf{L}_{\psi}
\Bigl(\sum_{j=1}^k(x_{1,j}+x_{2,j}-x_{3,j}-x_{4,j})\Bigr).  
$$
\par
Let $X\subset \Gg_m^{4k-1}\times\Aa^6$ be the variety over
$\Zz[1/\ell]$ with equations
\begin{align*}
  \alpha_{1,2}\cdots \alpha_{1,k}&=\beta (r+b_1)\\
  \alpha_{2,1}\cdots \alpha_{2,k}&=\beta (r+b_2)\\
  \alpha_{3,1}\cdots \alpha_{3,k}&=\beta (r+b_3)\\
  \alpha_{4,1}\cdots \alpha_{4,k}&=\beta (r+b_4).
\end{align*}
The morphism $X_2\lra \Gg_m\times X $ given by
\begin{gather*}
(x_{1,1},\ldots,x_{4,k},r,s,\bfb)\mapsto
\Bigl(x_{1,1}, \Bigl(\frac{x_{1,2}}{x_{1,1}},\ldots,
\frac{x_{4,k}}{x_{1,1}},r,\frac{s}{x_{1,1}^k},\bfb\Bigr)\Bigr)
\end{gather*}
is an isomorphism over $\Zz[1/\ell]$. In coordinates $(x_{1,1},x)$ of
$\Gg_m\times X$, we have
$$
\sum_{j=1}^k(x_{1,j}+x_{2,j}-x_{3,j}-x_{4,j}) =x_{1,1}g(x)
$$
where $g\,:\, X\lra \Aa^1$ is the morphism
$$
(\alpha_{1,2},\ldots,\alpha_{4,k},r,s,\bfb)\mapsto 1+\sum_{j=2}^k
\alpha_{1,j}+ \sum_{j=1}^k (\alpha_{2,j}-\alpha_{3,j}-\alpha_{4,j}).
$$
Similarly, the projection $f_2$ is $f\circ p_2$ in the coordinates of
$\Gg_m\times X$ where $f:X\lra\Aa^5$ is the projection $(\alpha_{1,2},\cdots,\alpha_{4,k},r,s,\bfb)\lra (r,\bfb)$ and $p_2$ is the second
projection $\Gg_m\times X\lra X$. Thus we get
$$
\mcR_{\lambda=0}\simeq R^{4k-3}(f\circ p_2)_!\mcL_{\psi}(tg(x))
$$
on $\Aa^5_{\Fq}$, for all $q\not=\ell$.
\par
We can now apply Lemma~\ref{lm-surprising-integrality} to the variety
$X$, to $Y=\Aa^5$ and to $f\,:\, X\lra Y$. We deduce the existence of
a complex $K$ on $\Aa^5_{\Zz[1/\ell]}$ such that, for $q \not=\ell$,
$$
R(f\circ p_2)_!\sheaf{L}_{\psi}(tg(x))=
K|\Aa^5_{\Fq}
$$
so
$$
\mcR_{\lambda=0}|\Aa^5_{\Fq}=R^{4k-3}(f\circ p_2)_!\sheaf{L}_{\psi}(tg(x))=
\mathcal H^{4k-3}(K)|\Aa^5_{\Fq}
$$
and we can take $\mcR^{univ} = \mathcal H^{4k-3}(K)$.

\end{proof}

\begin{proposition}\label{pr-new-irreducibility} 
  For any sufficiently large prime $q$, the specialized $\ell$-adic
  sum-product sheaf $\mcR_{\lambda=0,\bfb}^*$ is geometrically
  irreducible for all $\bfb$ in an open dense subset of
  $\Aa^4_{\Ff_q}$.
\end{proposition}

\begin{proof} 
  We will show that $\mcR_{\lambda=0,\bfb}^*$ is geometrically
  irreducible at the generic point. By Pink's Specialization
  Theorem~\cite[Th. 8.18.2]{ESDE}, this will imply the result on an
  open dense subset. We compactify $\Aa^4$ (resp. $\Aa^5$) in $\Pp^4$
  (resp. $\Pp^4\times\Pp^1$), and we compactify the projection
  $\pi\colon \Aa^5\to\Aa^4$ using the analogue projection
  $\bar{\pi}\colon \Pp^4\times\Pp^1\to \Pp^4$.
\par
Let $W$ be the stalk of $\mcR_{\lambda=0}^*$ at the generic point of
$\Aa^5$, and $\rho\colon G\lra \GL(W)$ the corresponding
representation of the Galois group
$$
G=\Gal( \overline{\Ff_q(\bfb,r)} / \overline{\Ff}_q(\bfb,r)).
$$
This representation is irreducible since the sheaf
$\mcR_{\lambda=0}^*$ on $\Aa^5$ is geometrically irreducible by an
application of Lemma~\ref{diophantine} and
Proposition~\ref{corRrbcorrelation}.
\par
It is then enough to prove that the restriction of $\rho$ to the
normal subgroup
$$
G_1=\Gal( \overline{\Ff_q(\bfb,r)} / \overline{\Ff_q(\bfb)} (r))
$$ 
is also irreducible, since this will show that the fiber of
$\mcR_{\lambda=0}^*$ over the generic point of $\Aa^4$ is
geometrically irreducible. Note that
$G/G_1=\Gal( \overline{\Ff_q(\bfb)} / \overline{\Ff}_q(\bfb))$.
\par
The quotient $G/G_1$ acts on the set $\mathcal{W}$ of $G_1$-invariant
subspaces of $W$. Assume that the action of $G_1$ on $W$ is \emph{not}
irreducible. Then there is some nonzero proper $G_1$-invariant
subspace, which cannot be $G$-invariant, so the action of $G/G_1$ on
$\mathcal{W}$ is not a trivial action.
\par
Since the tame geometric fundamental group of $\Aa^4$ is trivial (see,
e.g.,~\cite[Th. 5.1]{orgogozo} -- using the fact that the tame
fundamental group is independent of the choice of compactification, as
explained in loc. cit., and the fact that the tame fundamental group
of $\Aa^1$ is trivial), the action of $G/G_1$ on $\mathcal{W}$ must be
ramified at some codimension $1$ point of $\Aa^4$ or wildly ramified
at $\infty$, in the sense that the inertia group (resp. wild inertia
group) at such a point acts non-trivially.
\par
Let $D$ be divisor in $\Aa^4$ where the action of $G/G_1$ is
ramified. Denote by $I_D$ the corresponding inertia group and by
$I_{\bar{\pi}^{-1}(D)}$ the inertia group of the divisor
$\bar{\pi}^{-1}(D)$. We have the commutative diagram
$$
\begin{tikzcd}
I_{\bar{\pi}^{-1}(D)}\arrow{d}\arrow{r} & G\arrow{d}\arrow{r} &
\GL(W)\arrow{d}\\
I_D\arrow{r} & G/G_1 \arrow{r} & \mathrm{Sym}(\mathcal{W}).
\end{tikzcd}
$$
By Lemma~\ref{lm-surj-inertia}, the homomorphism on the left is
surjective. Since $I_D$ acts non-trivially on $\mathcal{W}$, it
follows that $I_{\bar{\pi}^{-1}(D)}$ acts non-trivially on $W$. Hence
$\mcR_{\lambda=0}^*$ is ramified at the pullback of some codimension
$1$ point of $\Aa^4$, or wildly ramified at $\infty$. By
Lemma~\ref{lm-mcr-lisse} (3), the monodromy action of
$\mcR_{\lambda=0}^*$ on some dense open set $V$ where it is lisse
factors through $\pi_1(U)$. Since $\mcR_{\lambda=0}^*|V$ is a quotient
of $\mcR_{\lambda=0}|V$, it follows that $\mcR_{\lambda=0}$ is either
ramified at the pullback in $\Aa^5$ of some codimension $1$ point of
$\Aa^4$ or wildly ramified at $\infty$.

However, if $q$ is sufficiently large, the sheaf $\mcR_{\lambda=0}$ is
not wildly ramified at $\infty$, because it is defined over $\Zz$ (by
Lemma \ref{lm-r-integrality}) and hence can only have wild
ramification at finitely many primes (as can be seen by applying
Abhyankar's Lemma~\cite[Exposé XIII, \S 5]{sga1} as in~\cite[Th. 4.7.1
(i)]{Sommes}).

Furthermore, by Lemma \ref{lm-mcr-lisse} (2), the sheaf
$\mcR_{\lambda=0}$ is lisse outside the complement of the union of the
subvariety $\tilde{Z}$ defined in that lemma and the divisors given by
the equations $r=-b_i$ in $\Aa^5$. So the only codimension $1$ points
where the sheaf is ramified are the generic points of these
divisors. The divisors with equation $r=-b_i$ are clearly irreducible,
and the same is true of $\tilde{Z}$ by Lemma \ref{lm-mcr-lisse} (1),
so they each contain a single codimension $1$ point, thus we will
obtain a contradiction if we show that none of these divisors is a
pullback from $\Aa^4$ under the map $(r,\bfb)\mapsto \bfb$.
\par
It is clear that the divisors with equation $r+b_i=0$ are not
pullbacks from $\Aa^4$. Recall that the divisor $\tilde{Z}$ was
defined as the (closed) projection of the subvariety with equation
$x_1+x_2-x_3-x_4$ of the subvariety $Z$ of $\Aa^9$ given
by~(\ref{eq-var-z}), and (from Lemma~\ref{lm-mcr-lisse} (1)) that it
is irreducible. This means we will be done if we check that
$\tilde{Z}$ is not a pullback from $\Aa^4$ when $q$ is sufficiently
large.  For instance, note that $(r,\bfb)=(0,1,1,(-1)^k,3^k)$ is in
$\tilde{Z}$, as the image of $(1,1,-1,3,0,1,1,(-1)^k,3^k)$; if
$\tilde{Z}$ is a pullback from $\Aa^4$, we must have also
$(-1,\bfb)=(-1,1,1,(-1)^k,3^k)\in \tilde{Z}$, but this is not the case
since the corresponding equations for $(x_1,\ldots,x_4)$ to be in $Z$
impose
$$
\begin{cases}
  x_1=x_2=0\\
  x_3^k=-1+(-1)^k\\
  x_4^k=-1+3^k,
\end{cases}
$$
and to be in $\tilde{Z}$ we should have a solution with $x_3=-x_4$,
hence
$$
(-1+(-1)^k)=(-1)^k(-1+3^k)\in\Ff_q.
$$
This equation holds only for finitely many primes $q$.
\end{proof}

\begin{proposition}\label{pr-degree-bound}
  Fix a prime $\ell$.  There exists a hypersurface
  $\mcV^{bad} \subseteq \Aa^4_{\Zz[1/\ell]}$, containing
  $\mcV^{\Delta}$, which is stable under the automorphism
  $\bfb\mapsto \tilde{\bfb}=(b_3,b_4,b_1,b_2)$ of $\Aa^4$,
  and such that, for any sufficiently large prime $q$, the specialized
  $\ell$-adic sum-product sheaf $\mcR_{\lambda=0,\bfb}^*$ over $\Fq$
  is geometrically irreducible for all $\bfb$ outside $\mcV^{bad}$.
\end{proposition}

\begin{proof} 
  First we see by Proposition~\ref{pr-new-irreducibility} that, for a given $q$ sufficiently large, the sheaf $\mcR_{\lambda=0,\bfb}^*$ is
  geometrically irreducible for all $\bfb$ outside of some subvariety
  of codimension $\geq 1$ over $\Fq$. 
\par
To construct the exceptional subvariety over $\Zz[1/\ell]$, we denote
by $\sigma\,:\, \Aa^4_{\Zz}\lra \Aa^4_{\Zz}$ the automorphism
$(r,\bfb)\mapsto(r,\tilde{\bfb})$. We define the $\ell$-adic sheaf
$\sheaf{F}=\mcR^{univ} \otimes
(\mathrm{Id}\times\sigma)^*\mcR^{univ}$, where $\mcR^{univ}$ is the
sheaf on $\Zz[1/\ell]$ constructed in
Lemma~\ref{lm-r-integrality}. This is a constructible sheaf on
$\Aa^5_{\Zz[1/\ell]}$. Setting $\pi$ to be the projection $(r,\bfb)\lra\bfb$ we define $\sheaf{E}=R^2 \pi_!\sheaf{F}$, a
constructible $\ell$-adic sheaf on $\Aa^4_{\Zz[1/\ell]}$.

Let $U\subset \Aa^4_{\Zz[1/\ell]}$ be the maximal open subset where
$\sheaf{E}$ is lisse. Let $H\supset \Aa^4-U$ be any codimension $1$ closed subscheme of $\Aa^4_{\Zz[1/\ell]}$ containing the complement of $U$. Let
then
$$
\mcV^{bad}=\mcV^{\Delta}\cup H\cup \sigma(H).
$$
It is clear that $\mcV^{bad}$ is stable under $\sigma$.
We will now show that, for any $q>k$ distinct from $\ell$, the
specialized sheaf $\mcR^*_{\lambda=0, \bfb}$ over $\Fq$ is
geometrically irreducible for $\bfb$ outside $\mcV^{bad}$.

Let such a $q$ be given and fix
$\bfb\in\Ff_q^4\notin \mcV^{bad}(\Fq)$. We claim that the specialized
sheaf $\mcR^*_{\lambda=0, \bfb}$ is geometrically irreducible if and
only if the weight $4$ part of the stalk
$H^2_c(\Aa^1_{\bFq}, \sheaf{F}_{\bfb})$ of $\sheaf{E}|\Aa^4_{\Fq}$ at
$\bfb$ is one-dimensional.  If this is so, then we are done: since
mixed lisse sheaves are successive extensions of pure lisse sheaves,
the rank of the weight $4$ part of $\sheaf{E}$ on the open set where
it is lisse is constant. The first part of the argument has shown that
this weight $4$ part is of rank $1$ on some dense open set, so we know
it has rank $1$ on the open set where it is lisse.
\par
The proof of the claim is similar to the argument in the proof of
Theorem~\ref{thmRbounds} above. If $U_{\bfb}$ is a dense open subset
on which $\sheaf{F}_{\bfb}$ is lisse, we have
$$
H^2_c(\Aa^1_{\bFq}, \sheaf{F}_{\bfb})\simeq
(\sheaf{F}_{\bfb,\bar{\eta}})_{\pi_1(U_{\bfb})}(-1)
$$
so the weight $4$ part of the stalk is isomorphic to the weight $2$
part of the coinvariants of $\sheaf{F}_{\bfb,\bar{\eta}}$. This weight
$2$ part is isomorphic to the coinvariants of the maximal weight $2$
quotient of $\sheaf{F}$, which is
$\mcR^*_{\lambda=0} \otimes [\mathrm{Id}\times
\sigma]^*\mcR^*_{\lambda=0}$.  By Lemma~\ref{lm-dual-sheaf} (and the
geometric simplicity of the sheaves), we have a geometric isomorphism
$$
\mcR^{*\vee}_{\lambda=0, \bfb} \simeq \mcR^*_{\lambda=0, \tilde{\bfb}}
$$
on any dense open set where the sheaf is lisse.  So the weight $4$
part of $H^2_c(\Aa^1_{\bFq}, \sheaf{F}_{\bfb})$
is the same as the coinvariants of
$\mcR^*_{\lambda=0,\bfb} \otimes \mcR^{* \vee}_{\lambda=0,\bfb}$,
which is just the endomorphisms of $\mcR^*_{\lambda=0} $ as a
geometric monodromy representation. Since $\mcR^*_{\lambda=0}$ is
geometrically semisimple, the dimension of this space if $1$ if and
only if the representation is geometrically irreducible.
\end{proof}

\subsection{Irreducibility of sum-product sheaves for $\lambda\not=0$}
\label{sec-not-zero}

This section is devoted to the study of the irreducibility of
sum-product sheaves for $\lambda\not=0$. \textit{We always assume that
  $q>k\geq 2$.}
\par
Using Lemma~\ref{lemcriterion}, we want to show that if
$\bfb\notin\mcV^{\Delta}$, then $\mcR_{\lambda,\bfb}$ is geometrically
irreducible for \emph{all} $\lambda\not=0$. This is the most delicate
part of our argument. The strategy is as follows:
\par
\begin{enumerate}
\item We show that for $\bfb\notin\mcV^{\Delta}$, the sheaf $\mcR_{\bfb}^*$
  is geometrically irreducible on $\Aa^2$; this gives the first
  condition in Lemma~\ref{lemcriterion}.
\item Let $\bfO=(0,0,0,0)$; we compute explicitly the wild part of the
  monodromy at infinity of $\mcR_{\lambda,\bfO}$ for $\lambda\not=0$.
\item We show that the wild part of the monodromy at infinity of
  $\mcR_{\lambda,\bfb}$ is independent of $\bfb$ (for
  $\lambda\not=0$), and thus is known by the previous step; this
  should be understood intuitively from the fact that for any
  $\bfb=(b_1,b_2,b_3,b_4)$, the map
  $$r\mapsto (r,r,r,r)\text{ 
  approximates the map
  } r\mapsto (r+b_1,r+b_2,r+b_3,r+b_4)\text{ as } r\ra\infty.$$
\item We extend the computation to $\mcR^*_{\lambda,\bfb}$; this leads
  to a verification of the second condition of
  Lemma~\ref{lemcriterion}.
\item Finally, we check the last condition of this lemma.
\end{enumerate}

\emph{In all of this section, we fix a tuple
  $\bfb\not\in\mcV^{\Delta}(\Fq)$. 
}

\begin{lemma}\label{A2irred} 
  For any $\bfb\not\in\mcV^{\Delta}(\Fq)$, the sheaf $\mcR^*_{\bfb}$ on
  $\Aa^2$ is geometrically irreducible on the open subset where it is
  lisse.
\end{lemma}

\begin{proof}
  The result follows from Lemma \ref{diophantine} and
  \eqref{Rcorrelation}, as in the beginning of the proof of
  Proposition~\ref{pr-degree-bound}.
\end{proof}


\begin{lemma}\label{lm-lisseness} 
The following properties hold:
\par
\emph{(1)}  The sheaf $\mcR_{\bfb}$ is
lisse on the complement $U$ of the union of the divisor with equation
$\lambda=0$ and of the divisors with equations $r=-b_i$ for
$1\leq i\leq 4$.
\par
\emph{(2)} The generic rank of $\mcR_\bfb$ is $k^4$.
\par
\emph{(3)} The generic rank of $\mcR_{\lambda=0,\bfb}$ is at most $k^3$.
\end{lemma}

\begin{proof} 
  First we prove that $\mcR_{\bfb}$ is lisse on $U$.
Let
$$
i\,:\, U\injecte \Aa^1\times \Gg_m
$$
and
$$
j\,:\, \Aa^2\times\Gg_m\injecte \Aa^1\times\Pp^1\times\Gg_m
$$
be the canonical open immersions, and let
$\tilde{\mcK}_{\bfb}=j_!\mcK_{\bfb}$ be the extension by $0$ of
$\mcK_{\bfb}$. Write $\tilde{\pi}$ for the projection
$(r,s,\lambda)\mapsto (r,\lambda)$ on
$\Aa^1\times\Pp^1\times\Gg_m$. Let $\pi=\tilde\pi \circ j$,
$\tilde{W}=\tilde{\pi}^{-1}(U)$ and $W=\pi^{-1}(U)$ so that $W$ is the
preimage of $\tilde{W}$ under $j$. Denote also by $\pi_W$ the
restriction of $\pi$ to $W$.
\par
We note that $\mcK_{\bfb}$ is lisse on the complement of $s=0$ in $W$,
and vanishes on the divisor $s=0$. Similarly, $\tilde{\mcK}_{\bfb}$ is
lisse on the complement of the smooth divisor
$\{s=0\}\cup \{s=\infty\}$ in $\tilde{W}$.  Moreover, we have
$$
\mcR_{\bfb}|U=R^1\pi_{W!}(\mcK_{\bfb})=R^1\overline{\pi}_{W!}(\tilde{\mcK}_{\bfb}|W),
$$
where the point is that we write the restriction of $\mcR_{\bfb}$ to $U$ as a
higher direct image of the restriction of a sheaf lisse outside a
smooth divisor.
\par
We next claim that the Swan conductor of $\tilde{\mcK}_{\bfb}$ is constant
along the two divisors $s=0$ and $s=\infty$. Indeed, recall that
$$
\mcK_{\bfb}=\sheaf{L}_{\psi(\lambda s)} \otimes\bigotimes_{i=1}^2
(f_i^*\HYPK_k\otimes f_{i+2}^*\HYPK_k^{\vee})= \sheaf{L}_{\psi(\lambda
  s)} \otimes\sheaf{G},
$$
say. Along the divisor $s=0$, the pullbacks $f_i^*\HYPK_k$ and
$f_{i+2}^*\HYPK_k^{\vee}$ are tamely ramified since the Kloosterman
sheaf $\HYPK_k$ is tamely ramified at $0$ and $s\mapsto (r+b_i)s$
fixes $0$ and $\infty$. Since $\mcL_{\psi(\lambda s)}$ is unramified
along $s=0$, we see that $\tilde{\mcK}_{\bfb}$ is tamely ramified, and in
particular has constant Swan conductor equal to zero.
\par
On the other hand, the Kloosterman sheaf $\HYPK_k$ is wildly ramified
at $\infty$ with unique break $1/k$, so the tensor product $\sheaf{G}$
above has all breaks at most $1/k$ at $\infty$ (again because $f_i$
fixes $\infty$ as a function of $s$). Since $k\geq 2$ and the single
sheaf $\sheaf{L}_{\psi(\lambda s)}$ has rank $1$ and is wildly
ramified at $\infty$ with break $1$ (recall that $\lambda\not=0$ in
this argument), the sheaf $\tilde{\mcK}_{\bfb}$ has unique break $1$ at
$\infty$. Since the rank of $\tilde{\mcK}_{\bfb}$ is $k^4$, the Swan
conductor at $s=\infty$ of $\tilde{\mcK}_{\bfb}$ is the constant $k^4$. This
establishes our claim.
\par
It follows from the above and Deligne's semicontinuity theorem,
\cite[Corollary 2.1.2]{LaumonSMF} that the sheaf $R^1\pi_{W!}(\mcK_{\bfb})$
is lisse on $U$.  As we observed, this is the same as the restriction
of $\mcR_{\bfb}$ to $U$ and hence $ \mcR_{\bfb}$ is lisse on $U$.
\par
Now we consider the rank estimates. By the propre base change theorem,
the stalk of $\mcR$ over
$x=(r,\lambda,\bfb)\in\Aa^2\times (\Aa^4\setminus \mcV^{\Delta})$ is
$H^1_c(\Aa^1_{\bFq}, \sheaf{F})$
where
$$
\sheaf{F}=\sheaf{L}_{\psi(s\lambda)}\otimes
\bigotimes_{i=1}^{2}[\times (r+b_i)]^*\HYPK_k\otimes [\times
(r+b_{i+2})]^*\HYPK_k^{\vee}))
$$
We recall from Lemma~\ref{lm-support} that the $0$-th and $2$-nd
cohomology groups of $\sheaf{F}$ vanish, so that the rank of the stalk
of $\mcR$ at $x$ is minus the Euler-Poincaré characteristic of the
sheaf whose cohomology we consider.  The Euler-Poincaré formula for a
constructible sheaf on $\Aa^1$ gives
$$
\chi(\Aa^1_{\bFq},\sheaf{F})= \rank(\sheaf{F})-\sum_{x\in
  \Pp^1}\swan_x(\sheaf{F}) -\sum_{x\in\Aa^1}\Drop_x(\sheaf{F}),
$$
where $\Drop_x(\sheaf{F})$ is the generic rank of $\sheaf{F}$ minus
the dimension of the stalk at $x$ (see,
e.g.,~\cite[p. 67]{katz-mellin} and the references there,
or~\cite[Cor. 10.2.7]{fu-book}).

Since we normalized the Kloosterman sheaf $\HYPK_{k}$ to have stalk
$0$ at $0$, so does $\sheaf{F}$, and the above formula becomes
$$
\chi(\Aa^1_{\bFq},\sheaf{F})= -\sum_{x\in \Pp^1}\swan_x(\sheaf{F})
-\sum_{x\in\Gg_m}\Drop_x(\sheaf{F}).
$$
\par
(2) In the generic case $\lambda\not=0$ and $r+b_i\not=0$, the rank is
equal to $k^4$ 
since the sheaf $\sheaf{F}$ is then lisse on $\Gg_m$, tame at $0$, and
has unique break $1$ with multiplicity $k^4$ at infinity.
\par
(3) If $\lambda=0$ (and $\bfb$ generic), then we get generic rank
$\leq \swan_{\infty}(\sheaf{F})\leq k^3$, since $\sheaf{F}$ (for
$\lambda=0$) has all breaks $\leq 1/k$ at $\infty$ and rank $k^4$.
\end{proof}


We now consider the local monodromy of $\mcR_{\bfb}$ in terms of the
$r$ variable for $\lambda\not=0$. First we deal with the singularity
$r=-b_i$.

\begin{lemma}\label{tameness}  
  On the open set where $\lambda \neq 0$, the sheaf $\mcR_{\bfb}$ has
  tame ramification around the divisors $r=-b_i$ for $1\leq i\leq 4$.
\end{lemma} 

\begin{proof} 
  Let $\mathcal{O}$ be the ring of integers in a finite extension of
  $\Qq_{\ell}$ such that the sheaves $\HYPK_k$ and
  $\sheaf{L}_{\psi(\lambda s)}$ have a model over $\mathcal{O}$, in
  the sense of~\cite[Remark 1.10]{GKM}, and let $\varpi$ be a
  uniformizer of $\mcO$. Then $\mcR_{\bfb}$ has a model over
  $\mathcal{O}$ and we have
$$
\swan_{-b_i}(\mcR_{\bfb}) = \swan_{-b_i}(\mcR_{\bfb}/\varpi)
$$
for any $i$ (see, e.g.,~\cite[Remark 1.10]{GKM}). Thus we reduce to
$\ell$-torsion sheaves.
\par
We will show that the torsion sheaf $\mcR_{\bfb}/\varpi$ is trivialized
at $-b_i$ after pullback to a covering defined by adjoining $n$-th roots of $r+b_i$, for some $n$ coprime to
$q$. This implies that $\mcR_{\bfb}/\varpi$ is tame at $-b_i$, and hence
gives our claim.
\par
We fix $\lambda\not=0$ and we now view $f_i$ as a morphism
$\Aa^1\times\Aa^1\lra \Aa^1$ given by $(r,s)\mapsto s(r+b_i)$.  Over
the \'{e}tale local ring at $0$, the sheaf $\HYPK_k$ is isomorphic (by
Proposition~\ref{pr-kl} (4)) to the extension by zero of a lisse sheaf
$\mathcal U$ on $\Gg_m$ corresponding to a principal unipotent rank
$k$ representation of the tame fundamental group
$$
\lim_{\substack{\leftarrow\\(n,q)=1}}\mmu_n(\bFq)
$$
of $\Gg_m$.
Hence $f_i^* \HYPK_k$ and $f_i^* \mathcal U$ are isomorphic after
pullback in an \'{e}tale neighborhood of the divisor $D$ with equation
$s(r+b_i)=0$ in $\Aa^2$.
\par
The sheaf $\mathcal U/\varpi$ corresponds to a representation of the
monodromy group of an \'etale Kummer covering of $\Gg_m$, defined by
adjoining the $n$-th root of the coordinate for some $n$ coprime to
$q$. Therefore $f_i^* \mathcal U/\varpi$ corresponds to a covering of
$\Aa^2$, ramified over $D$, obtained by adjoining the $n$-th root of
$s(r+b_i)$. It follows that, if we adjoin the $n$-th root of $r+b_i$,
the cover defining $f_i^*\mathcal U/\varpi$ becomes isomorphic to the
cover obtained by adjoining the $n$-th root of $s$ of order coprime to
$q$.  Consider the map
$$
g\,:\, \Aa^2 \lra \Aa^2
$$
with $g(t,s)=(t^n-b_i,s)$.  Then because of this isomorphism of
covers, $g^*f_i^* \mathcal U/\varpi$ is locally isomorphic to
$g^*[(r,s)\mapsto s]^* \mathcal U/\varpi=[(t,s)\mapsto
s]^*\mathcal{U}/\varpi$. From now, on we will write
$s^*\sheaf{U}/\varpi$ for $[(t,s)\mapsto s]^*\mathcal{U}/\varpi$.

The sheaf $g^* f_i^* \HYPK_k$ is lisse on $\Aa^2$ away from the lines
$s=0$ and $t=0$. We claim that $g^* f_i^* \HYPK_k$, restricted to the
open set $t\neq 0$, may be extended to $\Aa^2$ to a sheaf $\HYPK_k'$
in such a way that $\HYPK_k'$ is lisse away from the line $s=0$, and
isomorphic to $s^* \mathcal U/\varpi$ on the line $t=0$. This is an
\'{e}tale-local condition, and may be checked in an \'{e}tale
neighborhood of the line $t=0$. In fact, since it depends only on the
restriction to the open set $t \neq 0$, it may be checked on the
complement of the line $t=0$ in an \'{e}tale neighborhood of
itself. In such a neighborhood, we have the two aforementioned
isomorphisms
$g^* f_i^* \HYPK_k /\varpi \cong g^* f_i^* \mathcal U/\varpi \cong s^*
\mathcal U/\varpi$. The existence of the desired extension is obvious
for $s^* \mathcal U/\varpi$, hence holds for $g^*f_i^*\HYPK_k$. We
next denote by $\HYPK_k^0$ the extension by zero to $\Aa^2$ of the
restriction of $\HYPK_k'$ to the complement of the line $s=0$ in
$\Aa^2$.

We have
$$
g^*\mcK_{\bfb}/\varpi= g^* \sheaf{L}_{\psi(\lambda s)}/\varpi \otimes
g^* f_i^* \HYPK_k/\varpi\otimes \bigotimes_{j\not=i} 
g^* f_j^* \HYPK_k/\varpi.
$$
Let
$$
\mcK_{\bfb}^0 =g^* \sheaf{L}_{\psi(\lambda s)}/\varpi \otimes
\HYPK^0_k\otimes \bigotimes_{j\not=i} g^* f_j^* \HYPK_k/\varpi
$$
be the same tensor product but with the $g^* f_i^* \HYPK_k/\varpi$
term replaced with $\HYPK_k^0$.
Then $\mcK_{\bfb}^0$ is lisse on $\Aa^2$ away from the line $s=0$ and
the lines $t^n-b_i=-b_j$ for $j \neq i$.

The sheaf $R^1 \pi_! \mcK_{\bfb}^0$ is lisse in an étale neighborhood
of $0$, by a proof similar to the proof in Lemma \ref{lm-lisseness}
that $\mcK$ is lisse.  Indeed, $ \mcK_{\bfb}^0$ is lisse near $t=0$
away from $s=0$ and $s=\infty$, and tamely ramified at $0$, so by
Deligne's semicontinuity theorem \cite[Corollary 2.1.2]{LaumonSMF} it
suffices to check that the Swan conductor of $ \mcK_{\bfb}^0$ at
$\infty$ is constant. The three Kloosterman sheaves all have breaks at
$\infty$ strictly less than $1$, and the same is true of $\mcK^0$
because for $t\neq 0$ it is a Kloosterman sheaf and at $t=0$ it is
unipotent and tame. Thus tensoring with $\sheaf{L}_{\psi(\lambda s)}$,
all the breaks become $1$ and the Swan conductor is constant.

So the local monodromy at $t=0$ of $R^1 \pi_! \mcK_{\bfb}^0$ is
trivial. But, by construction, the sheaf $\mcK_{\bfb}^0$ is isomorphic
to $g^* \mcK_{\bfb}/\varpi$ away from $t=0$, so the local monodromy of
$$
R^1 \pi_! g^* \mcK_{\bfb}/\varpi= [t\mapsto t^n-b_i]^* R^1 \pi_!
\mcK_{\bfb}/\varpi= [t\mapsto t^n-b_i]^* \mcR_{\bfb}/\varpi
$$
around $t=0$ is also trivial.  Thus $\mcR_{\bfb}/\varpi$ has trivial
local monodromy after adjoining the $n$-th roots of the uniformizer,
and is tamely ramified, as desired.
\end{proof}

It remains to compute the local monodromy at $\infty$. For this
purpose, we will use the theory of nearby and vanishing cycles. Since
this theory is likely to be unfamiliar to analytic number theorists,
Appendix~\ref{app-nearby} gives a short introduction, with some
explanation of its relevance for our purposes.


\begin{lemma}\label{lm-nearby} 
  Let $\lambda\not=0$ be fixed in a field extension (possibly
  transcendental) of $\Ff_q$. Let $X$ be the blowup of
  $\Pp^1 \times \Pp^1$ at the point $(r,s)=(\infty,0)$. Consider the
  projection $X\lra\Pp^1$ given by $(r,s)\mapsto r$. Let $\sheaf{F}$
  be the extension by zero of the sheaf $\mcK_{\lambda,\bfb}$ on
  $\Aa^2$ to $X$, and let $\sheaf{G}$ be the extension by zero of
  $\mcK_{\lambda,\bfO}$ on $\Aa^2$.
\par
\emph{(1)} The nearby cycles sheaves of $\sheaf{F}$ and $\sheaf{G}$
over $r=\infty$ are locally isomorphic at all $s\not=\infty$ in
$\Pp^1$ and at each point of the exceptional divisor of $X$.
\par
\emph{(2)} The nearby cycles sheaves of $\sheaf{F}$ and $\sheaf{G}$
over $r=\infty$ have the property that the stalk of $R \Psi \sheaf{F}$
at $s=\infty$, as a representation of the wild inertia group, can be
split into summands
$$
\rho_1,\ldots,\rho_m,
$$
and the stalk of $R \Psi \sheaf{G}$ at $s=\infty$, as a representation
of the wild inertia group, can be split into summands
$$
\rho'_1,\ldots,\rho'_m,
$$
such that, for all $i$, the representations $\rho'_i$ and $\rho_i$ of
the wild inertia group are isomorphic up to order $2$
reparameterizations, in the sense of
Definition~\ref{def-reparameterization}.
\end{lemma}

\begin{remark}
  We use the blowup $X$ instead of $\Pp^1\times\Pp^1$ because the
  argument below would not apply to $\Pp^1\times\Pp^1$: for
  $(r,s)=(\infty,0)$, the function $1/(rs)$ does not belong to the
  maximal ideal. See, e.g.,~\cite[p. 28--29]{hartshorne} for a quick
  description of blowups.
\end{remark}

\begin{proof} 
  (1) Since
\begin{gather*}
  \mcK_{\lambda,\bfb}=\mcL_{\psi(\lambda
    s)}\otimes\bigotimes_{i=1}^{2} \Bigl([s\mapsto (r+b_i)s]^*\HYPK_k
  \Bigr) \otimes \Bigl([s\mapsto (r+b_{i+2})s]^*\HYPK_k^\vee
  \Bigr),\\
  \mcK_{\lambda,\bfO}=\mcL_{\psi(\lambda
    s)}\otimes\bigotimes_{i=1}^{2} \Bigl([s\mapsto rs]^*\HYPK_k\Bigr)
  \otimes\Bigl( [s\mapsto rs]^*\HYPK_k^\vee \Bigr),
\end{gather*}
on $\Aa^2$, the \'etale-local nature of nearby cycles shows that it is
enough to prove that, for $1\leq i\leq 4$, the sheaf
$[s\mapsto (r+b_i)s]^* \HYPK_k$ is locally isomorphic to
$[s\mapsto rs]^* \HYPK_k$ on $\Aa^1-\{0\}\subset \Pp^1$ (with
coordinate $s$) and on the exceptional divisor $D$ of the blowup.
\par
For points not on the exceptional divisor, we apply
Lemma~\ref{lm-kloosterman-invariance} (2) to the strict henselization
$R$ of the local ring at $(\infty,s)\in X$, with $a=rs$ and $b= b_is$,
where $r$ and $s$ are now viewed as elements of the field of fractions
of $R$.  Note that $r^{-1}$ belongs then to the maximal ideal
$\mathfrak{m}$ of $R$ (since we are considering the situation at
$r=\infty$) and $s$ is a unit (since we are outside the exceptional
divisor), hence $a^{-1}$ also belongs to $\mathfrak{m}$. Moreover $b\in R$,
and therefore we obtain
$$
(a+b)^*\HYPK_k\simeq a^*\HYPK_k,
$$
which is the desired conclusion.
\par
The exceptional divisor $D$ is isomorphic to $\Pp^1$ by the map
$(r,s)\mapsto s/r^{-1}=rs$. 
Hence, for all points $x$ on $D$ except the point mapping to $\infty$
under this isomorphism, the function $rs$ is a function in the local
ring at $x$, and we may apply Lemma~\ref{lm-kloosterman-invariance}
(1) to the strict henselization $R$ of the local ring at that point,
with $a=rs$ and $b=b_i/r$. The function $1/(rs)$ vanishes at the point
mapping to $\infty$, thus is in the maximal ideal, so we may again use
Lemma~\ref{lm-kloosterman-invariance} (2) with $a=rs$ and $b=b_is$.
\par
(2) We denote again by $R$ the strict henselization of the local ring
at $(\infty,\infty)\in\Pp^1\times\Pp^1$ and by $\mathfrak{m}$ its
maximal ideal.  We also denote by $R_1$ the strict henselization of
the local ring at $\infty$ of $\Pp^1$ (with coordinate $r$), and by
$\mathfrak{m}_1$ its maximal ideal. Then $1/r$ is a uniformizer or
$R_1$. Let $R_0$ be the extension of $R_1$ generated by a $k$-th root
$1/\rho$ of $1/r$. Let $U = \Spec R_0[\rho]$. Since
$1+b_i/r\equiv 1\mods{\mathfrak{m}}$, there exists
$y_i\in R_1\subset R$ with $y_i^k=1+b_i/r$ and
$y_i\equiv 1\mods{\mathfrak{m}_1}$.  We can apply
Lemma~\ref{lm-local-product} to $U$, where $f$ is the projection to
$\Spec R_1-\{\infty\}$ composed with the inclusion
$\Spec R_1-\{\infty\} \to \Aa^1-\{-b_1,\dots,b_4\}$ and
$r_i = \rho y_i$. We observe that $r_1r_2r_3r_4=\rho^4 y_1y_2y_3y_4$
is a perfect square, as $y_1,y_2,y_3,y_4$ are all units in $R_1$,
hence squares in $R_1$ and thus squares in $R_0$.
Hence, by Lemma~\ref{lm-local-product}, we have an
isomorphism of local monodromy representations
\begin{multline*}
[f\times\mathrm{Id}]^*\Bigl(\mcL_{\psi(\lambda
  s)}\otimes\bigotimes_{i=1}^{2} f_i^* \HYPK_k \otimes f_{i+2}^*
\HYPK_k^\vee\Bigr)\\
 \simeq \bigoplus_{ \zeta_2,\zeta_3, \zeta_4 \in
  \mmu_k}\mcL_{\psi(\lambda s)} \otimes \mcL_{\tilde{\psi}} \Bigl(
s^{1/k} \rho \Bigl( y_1+\zeta_2 y_2-\zeta_3
y_3-\zeta_4y_4\Bigr)\Bigr),
\end{multline*}
where $\tilde{\psi}(x)=\psi(kx)$ as before.
  
The nearby cycles is preserved by this pullback to a $k$-th power
covering, as is the action of the wild inertia subgroup (because the
action of the full inertia group is restricted to the inertia group of
the covering, which contains the wild inertia group).

Since the nearby cycle functor is additive, we have a local
isomorphism
$$
[f\times\mathrm{Id}]^*R\Psi \mcK_{\lambda,\bfb} \simeq \bigoplus_{
  \zeta_2,\zeta_3, \zeta_4 \in \mmu_k} R\Psi\Bigl( \mcL_{\psi(\lambda
  s)} \otimes \mcL_{\tilde{\psi}} \Bigl( s^{1/k} \rho \Bigl(
y_1+\zeta_2 y_2-\zeta_3 y_3-\zeta_4y_4\Bigr)\Bigr) \Bigr)
$$
and we handle each term in the sum separately.  We will show that, for
each $(\zeta_2, \zeta_3,\zeta_4)$, either the corresponding component
has no nearby cycles for \emph{any} $\bfb\in\Aa^4$ (not only for
$\bfb\not\in\mcV^{\Delta}$), or that its nearby cycles, with the
action of the wild inertia group, are independent of $\bfb\in\Aa^4$,
up to reparameterizations of order $2$. We consider two cases.
\par
\textbf{Case 1.} Assume that $1+\zeta_2=\zeta_3 + \zeta_4$.
\par
In that case, the element
$$ 
y_1 + \zeta_2 y_2 - \zeta_3 y_3 -\zeta_4 y_4
$$
of $R$ belongs to the maximal ideal. 
Since $\rho^{-1}$ is a uniformizer of $R_0$, the element
$$
\rho (y_1 + \zeta_2 y_2 - \zeta_3 y_3 - \zeta_4 y_4 )
$$
belongs to $R_0$. Thus the sheaves
$$
\mcL_{\psi(\lambda s)},\quad\quad
\mcL_{\tilde{\psi}} \left( s^{1/k} \rho \left(
    y_1 + \zeta_2 y_2 - \zeta_3 y_3 - \zeta_4 y_4 \right)\right)
$$
both extend to lisse sheaves in an \'etale neighborhood of
$(\infty,\infty)$ away from the line $s=\infty$. 
\par
To check that their tensor product has no vanishing cycles, it
suffices (by Deligne's semicontinuity theorem once more \cite[Théorème
5.1.1]{LaumonSMF}) to check that the Swan conductor is constant. But
the breaks at infinity (in terms of $s$) of
$$
\mcL_{\tilde{\psi}} \left( s^{1/k} \rho \left( y_1 + \zeta_2 y_2 -
    \zeta_3 y_3 - \zeta_4 y_4 \right)\right)
$$
are all $\leq 1/k$, while $\mcL_{\psi( \lambda s)}$ has break $1$, so
the tensor product has all breaks equal to $1$, and we are done.
\par
\medskip
\par
\par
\textbf{Case 2.} Assume that $1 + \zeta_2 \not= \zeta_3 + \zeta_4$.
Then we have
$$
y_1+y_2\zeta_2-y_3\zeta_3-y_4\zeta_4= (1+\zeta_2-\zeta_3-\zeta_4)d
$$
where $d\in R_1$ satisfies $d\equiv 1\mods{\mathfrak{m}_1}$.
Let $\mu=\rho d$ and $u = rd^k = \mu^k$. 
Then we have
$$
\rho (y_1+y_2\zeta_2-y_3\zeta_3-y_4\zeta_4)=
\mu (1+\zeta_2-\zeta_3-\zeta_4)
$$
So, after pulling back to $U$ (which is also the cover defined by
adjoining $\mu$), we are dealing with the sheaf
$$
\mcL_{\psi(\lambda s)}\otimes
\mcL_{\tilde{\psi}}\Bigl(s^{1/k} \mu (1+\zeta_2-\zeta_3-\zeta_4)\Bigr).
$$
\par
The wild inertia action on the nearby cycles of this sheaf, in terms
of the variable $u$, can be computed on the pullback to the cover
defined by $\mu$ with $\mu^k=u$, and thus is independent of
$\bfb\in\Aa^4$, because this formula for the pullback is independent
of $\bfb$ and the cover is also independent of $\bfb$.

Since $1/r$ and $1/u$ are uniformizers of $R_1$, there is a unique
automorphism $\sigma$ of $R_1$ sending $r$ to $u$. Since
$d\equiv 1\mods{\mathfrak{m}_1}$, it follows that
$$
\frac{1}{u}\equiv \frac{1}{r}\mods{(1/r)^2},
$$
and hence $\sigma$ is a reparameterization of order $2$ (see
Definition~\ref{def-reparameterization}). This is the desired result.
\end{proof}

We will describe the wild part of the local monodromy at $r=\infty$ of
$\mcR_{\lambda,\bfb}$ using the following data.

\begin{definition} 
\label{def-sk}
Let $k\geq 2$ and let $q$ be a prime with $q\nmid k$.  We denote by
$S_k$ the multiset of \emph{non-zero} elements of $\bFq$ of the form
$$
( 1 + \zeta_2 -\zeta_3 - \zeta_4)^k
$$
where $\zeta_2$, $\zeta_3$ and $\zeta_4$ range over $\mmu_k(\bFq)$.
\end{definition}

We first use this definition to treat the local monodromy for
$\mcR_{\lambda,\bfO}$. 

\begin{lemma}\label{lm-local-fourier}
  Let $\lambda\not=0$ be fixed in a field extension (possibly
  transcendental) of $\Ff_q$. The local monodromy representation
  of $\mcR_{\lambda,\bfO}$ at $r=\infty$ is isomorphic to that of the
  sheaf
$$
\bigoplus_{\alpha \in S_k} [\times \alpha\lambda^{-1}]^* \mcH_{k-1},
$$
where $\mcH_{k-1}$ is the sheaf defined in
Definition~\ref{def-hypergeo}, plus a tamely ramified representation.
\end{lemma}

The meaning of the direct sum over the multiset $S_k$ is 
$$
\bigoplus_{\substack{\zeta_2,\zeta_3,\zeta_4\in
    \mmu_k\\1+\zeta_2-\zeta_3-\zeta_4\not=0}} [\times
((1+\zeta_2-\zeta_3-\zeta_4)^k\lambda)^{-1}]^* \mcH_{k-1},
$$
and similarly below.

\begin{proof} 
  Note that every representation of the inertia group is a sum of a
  wildly ramified representation and a tamely ramified representation,
  as the wild part is a $q$-group, so has semisimple $\ell$-adic
  representation theory, hence every representation of the wild inertia
  group splits canonically into trivial and nontrivial parts. Thus,
  because $\mcH_{k-1}$ is totally wild at $\infty$, we concern
  ourselves only with the wild summand.

The change of variable
$$
(r,s)\mapsto (\lambda/r,rs)
$$
is an isomorphism $\Gg_m\times \Aa^1\lra \Gg_m\times\Aa^1$ (with
inverse $(\xi,x)\mapsto (\lambda/\xi,x\xi/\lambda)$). In terms of the
variables $(\xi,x)$, the sheaf $\mcR_{\lambda,\bfO}$ becomes the
Fourier transform with respect to $\psi$ of the sheaf
$$
\sheaf{F}=\bigotimes_{i=1}^2 \HYPK_k \otimes \overline{\HYPK_k},
$$
on $\Aa^1$ with coordinate $x$, reflecting the trace function identity
$$
\sum_{s \in \Fq} \psi(\lambda s)\prod_{i=1}^2 \Kl_k( rs)
\overline{\Kl_k(rs)} = \sum_{ x \in \Fq} \psi ( x \xi) \prod_{i=1}^2
\Kl_k(x) \overline{\Kl_k(x)}.
$$
\par
We now need to compute the local monodromy at $\xi=0$ of this Fourier
transform, which we can do using Laumon's local Fourier transform
functors. Laumon's results (see, e.g.,~\cite[Th. 7.4.3,
Cor. 7.4.3.1]{ESDE}) give an isomorphism
$$
\mcR_{\lambda,\bfO}/ (\mcR_{\lambda,\bfO})_0\simeq
\ft_{\psi}\text{loc}(\infty,0)
(\sheaf{F}(\infty))
$$
of representations of the inertia group at $0$, where
$(\mcR_{\lambda,\bfO})_0$ is the stalk at $0$ and $\sheaf{F}(\infty)$
is the local monodromy representation of $\sheaf{F}$ at
$\infty$. Since the stalk at $0$ is a trivial representation of the
inertia group, this implies that the wild summand of the local monodromy
is the same as that of
$\ft_{\psi}\text{loc}(\infty,0)(\sheaf{F}(\infty))$.
\par
Using Lemma~\ref{lm-kl-infty} as in Lemma \ref{lm-local-product}, the
local monodromy at $\infty$ of $\sheaf{F}$ is isomorphic to that of
$$
\bigoplus_{\zeta_2,\zeta_3,\zeta_4 \in \mmu_k} \mcL_{\tilde{\psi}}
\left( x^{1/k}(1+ \zeta_2 - \zeta_3 - \zeta_4) \right)=
\bigoplus_{\zeta_2,\zeta_3,\zeta_4 \in \mmu_k} \mcL_{\tilde{\psi}}
\Bigl(\Bigl( (1+ \zeta_2 - \zeta_3 - \zeta_4)^k x\Bigr)^{1/k} \Bigr)
$$
where $\tilde{\psi}(x)=\psi(kx)$.  All triples
$(\zeta_2,\zeta_3,\zeta_4)$ with $1+ \zeta_2 - \zeta_3 - \zeta_4=0$
give tamely ramified local monodromy, whose local Fourier transform at
$0$ is also tamely ramified (see, e.g.,~\cite[Th. 7.4.4 (3)]{ESDE}),
so do not contribute to the wild part of the local monodromy.
\par
Otherwise, if $\alpha=(1+ \zeta_2 - \zeta_3 - \zeta_4)^k\not=0$ is an
element of $S_k$, then we have the following isomorphisms of local
monodromy representations at $0$ in $\Aa^1$ with (Fourier) coordinate
$\xi$, using the definition of the sheaf $\mcH_{k-1}$:
\begin{align*}
  \ft_{\psi}\text{loc}(\infty,0)
  (\mcL_{\tilde{\psi}}
  ( (\alpha x)^{1/k} ))
  &=
    [\times \alpha^{-1}]^*
    R\Phi_{\eta_0}\ft_\psi
    (\mcL_{\tilde{\psi}}(x^{1/k}))\\
  &\simeq [\times \alpha^{-1}]^*
    R\Phi_{\eta_0}([\xi\mapsto \xi^{-1}]^*\mcH_{k-1})\\
  &\simeq [\times \alpha^{-1}]^*
    [\xi\mapsto \xi^{-1}]^*\mcH_{k-1,\eta_\infty}\\
  &\simeq [\xi\mapsto (\alpha/\xi)]^*\mcH_{k-1,\eta_\infty}.
\end{align*}
(It is important to note that when composing pullbacks, one applies
the leftmost functions first, since this is the opposite order from
the usual composition of functions, where the rightmost is applied
first.  So $\xi$ is sent to $\alpha^{-1} \xi$ which is sent to
$(\alpha^{-1}\xi)^{-1}= \alpha/\xi$.) Since $\xi=\lambda/r$, this
concludes the proof.
\end{proof}

We can finally conclude:

\begin{cor}\label{cor-hardest-local-monodromy}
  Let $\lambda\not=0$ be fixed in a field extension (possibly
  transcendental) of $\Ff_q$. The wild inertia representation of
  $\mcR_{\lambda,\bfb}$ at $r=\infty$ is the same as that of the sheaf
  $$
  \bigoplus_{\alpha \in S_k} [\times \alpha\lambda^{-1}]^* \mcH_{k-1}.
$$
plus a trivial representation.
\end{cor}

For the proof, we use the same notation as in
Lemma~\ref{lm-nearby}. Thus, $X$ denotes the blowup of
$\Pp^1\times\Pp^1$ at $(\infty,0)$ and $\sheaf{F}$ and $\sheaf{G}$ on
$X$ are the extensions by $0$ from $\Aa^1 \times \Aa^1$ to $X$ of
$\mcK_{\lambda,\bfb}$ and $\mcK_{\lambda,\bfO}$ respectively. Let
$\pi$ be the proper map
$$
X\lra \Pp^1\times\Pp^1\lra \Pp^1
$$
where the second map is the proper projection $(r,s)\mapsto r$.

We need to compute the wild inertia representations at $\infty$ of
$R\pi_* \sheaf{F}$ and $R\pi_* \sheaf{G}$.  To do that, we use the
nearby cycles $R \Psi \sheaf{F}$ and $R \Psi \sheaf{G}$ relative to
$\pi$. These are complexes of sheaves with an inertia group action on
the fiber over $\infty$ over $X$. We know by Lemma~\ref{lm-nearby}
that $R\Psi \sheaf{F}$ and $R\Psi \sheaf{G}$ are locally isomorphic
away from the point $(\infty,\infty)$.

The key step is the following sub-lemma:

\begin{lemma}
  Away from $(\infty,\infty)$, the wild inertia group acts trivially
  on $R \Psi \sheaf{G}$.
\end{lemma}

\begin{proof}
  Let $f_1(r,s)=rs$. By definition, we have
$$
\mcK_{\lambda,\bfO}= \mcL_{\psi(\lambda s)} \otimes
f_1^*\HYPK_k^{\otimes 2} \otimes
(f_1^*\overline{\HYPK_k}^{\vee})^{\otimes 2}.
$$ 
Because we are verifying a local condition away from the line
$s=\infty$, we may ignore the factor $\mcL_{\psi(\lambda s) }$ and
consider only the nearby cycles of
$$
\sheaf{K}=f_1^*\HYPK_k^{\otimes 2} \otimes
(f_1^*\overline{\HYPK_k}^{\vee})^{\otimes 2}.
$$
For any $\alpha\in\Gg_m$, let $s_{\alpha}$ be the map
$(r,s)\mapsto (\alpha r,\alpha^{-1} s)$. We have
$f_1\circ s_{\alpha}=f_1$, hence
$s_{\alpha}^*\sheaf{K}\simeq\sheaf{K}$. The action of $s_\alpha$
extends to the blow-up $X$ and to the fiber of $X$ over $\infty$, so
it extends by functoriality to the nearby cycles complex
$R \Psi\sheaf{K}$. Since $s_\alpha$ acts by scaling on the coordinate
$r$ of the base local ring, the induced isomorphism
$s_{\alpha}^* R \Psi \sheaf{K} \simeq R \Psi \sheaf{K}$ sends the
Galois action on the nearby cycles complex to its multiplicative
translate by $\alpha$. Since the nearby cycles sheaf is constructible
\cite[Th. Finitude, Theorem 3.2]{sga4h}, only finitely many different
irreducible representations of the inertia group can appear in the
stalks of $R \Psi \sheaf{K}$ as Jordan-Hölder factors anywhere on the
fiber over $\infty$ (on each open set where $R \Psi \sheaf{K}$ is
lisse, there is a single representation with finitely many
Jordan-Hölder factors, and at each other point there is another
representation, again with finitely many Jordan-Hölder factors). By
symmetry, if any irreducible inertia representation appears in the
stalks, its multiplicative translates by $\alpha$ must also
appear. But by \cite[4.1.6]{GKM}, any non-trivial wildly ramified
representation has infinitely many non-isomorphic multiplicative
translates as $\alpha$ varies, so the wild inertia group must act
trivially on the stalks.

Let $I_1$ be the wild inertia group. There is an $I_1$-invariants
functor from $\ell$-adic sheaves with an action of $I_1$ to
$\ell$-adic sheaves, and an adjoint functor that views $\ell$-adic
sheaves as $\ell$-adic sheaves with a trivial action of $I_1$, giving
a natural adjunction map
$(R \Psi \sheaf{G})^{I_1} \to R \Psi \sheaf{G}$.  Because $I_1$ is a
pro-$q$ group, the $I_1$-invariants functor on $\ell$-adic sheaves has
no higher cohomology. Because the stalks are $I_1$-invariant, this map
is an isomorphism on stalks away from $(\infty,\infty)$, hence an
isomorphism away from $(\infty,\infty)$, so the wild inertia group
acts trivially on $R \Psi \sheaf{G}$ away from $(\infty,\infty)$.
\end{proof}

\begin{proof}[Proof of Corollary~\ref{cor-hardest-local-monodromy}]
  It follows from the last lemma and from Lemma \ref{lm-nearby}(1)
  that the wild inertia group acts trivially on $R \Psi \sheaf{F}$
  away from $(\infty,\infty)$. 

  Let $Z=\{(\infty,\infty)\}$ and $U$ the open complement. Let $i$ be
  the closed immersion of $Z$ and $j$ the open immersion of $U$.  We
  have distinguished triangles
$$
R\pi_*j_!R\Psi\sheaf{F}|U\to R\pi_*R\Psi\sheaf{F}\to
R\pi_*i_*R\Psi\sheaf{F}|Z\to,
$$
and
$$
R\pi_*j_!R\Psi\sheaf{G}|U\to R\pi_*R\Psi\sheaf{G}\to
R\pi_*i_*R\Psi\sheaf{G}|Z\to
$$

The middle terms are the local monodromy representations of
$R\pi_* \sheaf{F}$ and $R \pi_* \sheaf{G}$ at $\infty$, which we want
to compute. The third terms are the stalks of $R \Psi \sheaf{F}$ and
$R \Psi \sheaf{G}$ at $(\infty,\infty)$. The left-hand terms, by the
above, have trivial wild inertia action at $\infty$.


Since the representations of the wild inertia group are semisimple (as
it is a pro-$q$-group acting on an $\ell$-adic vector space) this
implies that the nontrivial part of the wild inertia representation on
the local monodromy of $R\pi_*\sheaf{F}$ and of $R\pi_*\sheaf{G}$ are
each equal to the nontrivial parts of the wild inertia representation
on the stalks of $R \Psi \sheaf{F}$ and $R \Psi \sheaf{G}$ at
$(\infty,\infty)$. By Lemma \ref{lm-nearby}(2), the stalks of
$R \Psi \sheaf{F}$ and $R \Psi \sheaf{G}$ at $(\infty,\infty)$ can be
split into summands which are isomorphic as representations of the
wild inertia group up to order $2$ reparameterizations, so the
nontrivial parts of the wild inertia representations on
$R\pi_* \sheaf{F}$ and $R \pi_* \sheaf{G}$ can be split into summands
that are equal up to order $2$ reparameterizations.

Finally, Lemma~\ref{lm-local-fourier} shows that the wild inertia
representation at $\infty$ of $R\pi_*\sheaf{G}$ is exactly as claimed
in the statement. Since, by Lemma~\ref{lm-hypergeo-2}, any
summand of the local monodromy at $\infty$ of $R\pi_*\sheaf{G}$
(i.e., of $\mcR_{\lambda,\bfO}$) is preserved by reparameterizations
of order $2$, we obtain in fact the same decomposition for
$\mcR_{\lambda,\bfb}$ also.
\end{proof}


\begin{cor}\label{cor-adjusted-local-monodromy}
  Let $\lambda\not=0$ be fixed in a field extension (possibly
  transcendental) of $\Ff_q$. The wild inertia representation of
  $\mcR_{\lambda,\bfb}^*$ at $r=\infty$ is isomorphic to that of
$$
\bigoplus_{\alpha \in S_k}  [\times \alpha/\lambda]^* \mcH_{k-1}.
$$ 
plus a trivial representation.
\end{cor}

\begin{proof} 
  In view of Corollary~\ref{cor-hardest-local-monodromy} and of the
  definition of $\mcR^*$, it suffices to prove that the weight $<1$
  part of $\mcR_{\bfb}$ is tamely ramified at $r=\infty$. To do this
  we will study the action of the decomposition group at $\infty$ on
  the stalk of the weight $<1$ part of $\mcR_{\lambda \bfb}$ at a
  generic point of the $r$-line. 
\par
We apply Lemma~\ref{lm-weight-computation} (2) to
$C=\Aa^1\times \Pp^1 \times \Gg_m$, with coordinates
$(r, s, \lambda)$, with its dense open subset
$U=\Aa^1\times\Aa^1\times \Gg_m$ (with open embedding $j$), to the
morphism $\pi\,:\, C\lra \Aa^1 \times \Gg_m$ given by
$\pi(r,s, \lambda)=(r,\lambda)$ and to the sheaf
$\sheaf{F}=j_!\mcK_{\bfb}$ on $C$.  The assumptions of
Lemma~\ref{lm-weight-computation} are easily verified using
Lemma~\ref{lm-support} (2) and (3).
\par
Taking $x=(r,\lambda)$ for a generic value of $r$, the lemma implies
that the part of weight $<1$ of
$$
(R^1\pi_*\sheaf{F})_{x}=H^1(\pi^{-1}(\bar{x}), \sheaf{F})=(\mcR_{\bfb})_{x}
$$
is isomorphic to
$$
\mcK_{x,\bfb}^{I(0)}/(\mcK_{x,\bfb})_0\oplus
\mcK_{x,\bfb}^{I(\infty)}/(\mcK_{x,\bfb})_{\infty}
=\mcK_{x,\bfb}^{I(0)},
$$
since $\mcK_{x,\bfb}$ is totally wildly ramified at $s=\infty$
and has stalk $0$ at $s=0$.
\par
Recall that the local monodromy representation of $\HYPK_k$ at $0$ is
unipotent. Let $K$ be an algebraically closed field extension of
$\Ff_q$ containing $\lambda$, so that over $K$ the decomposition group
representation of $\HYPK_k$ at $0$ is unipotent.  Hence the
decomposition group representation of
$$
[(r,s)\mapsto s(r+b_i)]^*\HYPK_k
$$ 
at a point where $s=0$ is unipotent (still over $K$). The
decomposition group representation of $\mcL_{\psi(\lambda s)}$ is
trivial at a point where $s=0$. Hence we conclude that the
decomposition group over $K$ also acts unipotently on the tensor
product $\mcK_{x,\bfb}$. Hence the inertia invariants
$\mcK_{x,\bfb}^{I(0)}$ is a unipotent representation of Galois
group of the residue field of the generic point $x$. In particular,
the inertia group at $r=\infty$ acts unipotently. Because it is
unipotent, it must factor through a pro-$\ell$ group and hence be
tame.

\end{proof}

We need some last elementary geometric considerations to isolate
features of the local monodromy at $\infty$ that will allow us to
deduce the irreducibility and disjointness of the sheaves
$\mcR_{\lambda,\bfb}$.

\begin{lemma}\label{lm-sk-geometry} 
Let $k\geq 2$ be given.
\par
\emph{(1)} If $q$ is sufficiently large, then the multiset $S_k$
contains an element with multiplicity $1$.
\par
\emph{(2)} If $q$ is sufficiently large, then the group of
$\mu\in\bFq^{\times}$ such that $\mu S_k=S_k$ is trivial if $k$ is
even, and is reduced to $\{\pm 1\}$ if $k$ is odd.
\end{lemma}

\begin{proof} 
  We denote by $\tilde{S}_k\subset \Cc$ the analogue of $S_k$ defined
  using $\mmu_k(\bar{\Qq})$.  We observe that the set of non-zero
  numbers $\zeta_1 +\zeta_2 - \zeta_3 -\zeta_4$, where $\zeta_i$ runs
  over $\mmu_k(\bFq)$, is the set of $k$-th roots of the elements of
  $S_k$, and similarly for $\tilde{S}_k$ and
  $\zeta_i\in \mmu_k(\bar{\Qq})$.  Moreover, non-zero element of this
  form has the same multiplicity as its $k$-th power as an element of
  $\tilde{S}_k$. Indeed, there is a bijection from the set of
  representations
$$
 \alpha = \zeta_1 + \zeta_2 - \zeta_3 -\zeta_4
$$
to those of
$$
\alpha^k=(1+\zeta'_2-\zeta'_3-\zeta'_4)^k,
$$
given by
$(\zeta_1,\ldots,\zeta_4)\mapsto (\zeta_2/\zeta_1,\zeta_3/\zeta_1,
\zeta_4/\zeta_1)$
with inverse
$(\zeta'_2,\zeta'_3,\zeta'_4)\mapsto
(\zeta,\zeta\zeta'_2,\zeta\zeta'_3,\zeta\zeta'_4)$,
where $\zeta$ is such that
$\alpha=\zeta(1+\zeta'_2-\zeta'_3-\zeta'_4)$.
\par
(1) Since any two distinct elements of $\tilde{S}_k$ are equal modulo
$q$ for finitely many primes $q$, it is enough to check that the set
$\tilde{S}_k$ contains an element of multiplicity one in $\Cc$.  To
find an element of $\tilde{S}_k$ with multiplicity one, it is
sufficient to find an $\Rr$-linear map $\Cc\lra \Rr$ with a unique
maximum and minimum on $\mmu_k$.  Clearly a generic linear function
has this property (e.g., if $k$ is even, we may take the real part).
\par
(2) We first show the corresponding property for $\tilde{S}_k$. Let
$T_k$ be the multiset of numbers $\zeta_1+\zeta_2-\zeta_3-\zeta_4$. By
the description above, it is enough to show that the group of complex
numbers $\mu$ such that $\mu T_k=T_k$ is equal to $\mmu_k$ if $k$ is
even and to $\mmu_{2k}$ if $k$ is odd.
\par
Consider the convex hull of $T_k$. It is the difference of two copies
of twice the convex hull of the $k$-th roots of unity. Since the
convex hull of $\mmu_k$ in $\Cc$ is a $k$-sided regular polygon, the
convex hull of $T_k$ is a $k$-sided regular polygon if $k$ is even,
and a $2k$-sided regular polygon if $k$ is odd. The result is then
clear.

To reduce the case of $S_k$ to the complex case, we note that an
arbitrary non-empty finite set $S\subset \Cc$ or $S\subset \bFq$ may
only be equal to its multiplicative translate by $\mu$ if $\mu$ is a
root of unity. Moreover, $\mu S=S$, where $\mu$ is a primitive $n$-th
roots of unity, if and only if the coefficients of a monic polynomial
whose roots are $S$ vanish in degrees coprime to $n$. When reducing a
polynomial with algebraic coefficients modulo a prime $q$ large
enough, the set of degrees which are zero modulo $q$ is the same as
the same which are zero in $\Cc$. Hence, for $q$ large enough, the
same roots of unity stabilize $S_k$ as $\tilde{S}_k$.
\end{proof}

Finally we can conclude the basic irreducibility statement for
sum-product sheaves when $\lambda$ is non-zero:

\begin{proposition}\label{pr-nonzero-irreducibility}
  For $q$ large enough in terms of $k$, the sheaf
  $\mcR_{\lambda,\bfb}^*$ is geometrically irreducible whenever
  $\lambda\not=0$ is fixed in a field extension (possibly
  transcendental) of $\Ff_q$ and $\bfb\not\in \mcV^{\Delta}$.
\end{proposition}

\begin{proof} 
  We apply Lemma~\ref{lemcriterion} (b) to $Y=\Gg_m \times \Pp^1$
  where the coordinate of $\Gg_m$ is $\lambda$ and the coordinate of
  $\Pp^1$ is $r$ and to the first projection $f\,:\, Y\lra X=\Gg_m$.
  We consider the sheaf on $Y$ which is the extension by zero of the
  sheaf $\mcR^*_{\bfb}$ on $\Gg_m\times\Aa^1$. The divisor $D$ is the
  union of the divisors $\{r=-b_i\}$ and $\{r=\infty\}$. If the three
  conditions of Lemma~\ref{lemcriterion} hold, then we obtain our
  desired conclusion.
\par
By Lemma \ref{diophantine} and \eqref{Rcorrelation} the sheaf
$\mcR_\bfb^*$ is geometrically irreducible on $Y-D$, so that the first
condition holds. It is also pure on $Y-D$ by definition.
\par
Next, we will show that the second condition holds by showing that
there exists an irreducible component of multiplicity one in the local
monodromy at $\infty$ of the restriction of $\mcR^*_{\bfb}$ to the
fiber of $f$ over a geometric generic point of $\Gg_m$ whose
isomorphism class is Galois-invariant.

By Corollary~\ref{cor-adjusted-local-monodromy}, the wild inertia
representation of $\mcR_{\lambda,\bfb}^*$ at $r=\infty$ is isomorphic
to that of
$\bigoplus_{\alpha \in S_k} [\times \alpha/\lambda]^* \mcH_{k-1}$ plus
a trivial representation. By Lemma \ref{lm-sk-geometry}(1), assuming
$q$ is large enough, some $\alpha$ appears with multiplicity $1$ in
$S_k$. Take such an $\alpha$. Let $V$ be the subspace of that local
monodromy representation that is sent to
$[\times \alpha/\lambda]^* \mcH_{k-1}$ under this isomorphism.
  
By Lemma~\ref{lm-hypergeo-2} (3), the irreducible components of the
summands $[\times \alpha/\lambda]^* \mcH_{k-1}$ as representations of
the wild inertia group are disjoint.  So we may characterize $V$ as
the subspace generated by all representations of the wild inertia
group that are isomorphic to wild inertia representations that appear
in $[\times \alpha/\lambda]^* \mcH_{k-1}$. Because
$[\times \alpha/\lambda]^* \mcH_{k-1}$ is a representation of the full
decomposition group, that set of isomorphism classes is stable under
the action of the decomposition group, so $V$ is a subrepresentation
of the local monodromy representation as a representation of the full
decomposition group. (Here we work over a large enough finite field so
that all of $S_k$, including $\alpha$, is contained in the base
field.)

We will show that $V$, restricted to the inertia group, is
irreducible. Restricted to the wild inertia group, it is isomorphic to
$[\times \alpha/\lambda]^* \mcH_{k-1}$. By Lemma~\ref{lm-hypergeo-2}
(2), the action by conjugation of the tame inertia group on the
irreducible wild inertia subrepresentations of
$[\times \alpha/\lambda]^* \mcH_{k-1}$ is transitive. Thus any
subspace would be a sum of wild inertia characters and would be
invariant under the tame inertia subgroup. So it must contain all the
characters or none, and therefore $V$ is indeed irreducible.

Then the irreducible representation $V$ occurs with multiplicity $1$
because each wild inertia component in it occurs with multiplicity
$1$, and its isomorphism class is invariant under conjugation by the
Galois group because it extends to a representation of the full
decomposition group.
\par
For the third condition of Lemma~\ref{lemcriterion}, it is enough to
show that the functions
$$
\lambda\mapsto \sw_r ( \mcR_{\lambda,\bfb}^* \otimes
\mcR_{\lambda,\bfb}^{*\vee})
$$
are locally constant on the divisors $r=-b_i$ and $r= \infty$.  By
Lemma~\ref{tameness}, this function is constant (equal to $0$) on the
divisors $r=-b_i$ for $1\leq i\leq 4$. The Swan conductor is
determined by the restriction to the wild inertia subgroup.  By
Corollary \ref{cor-hardest-local-monodromy}, the restriction of
$\mcR_{\lambda,\bfb}$ to the wild inertia subgroup is a sum of terms
of the form $[\times \alpha/\lambda]^* \mcH_{k-1}$ plus a trivial
representation. Hence the restriction of
$\mcR_{\lambda,\bfb}^* \otimes \mcR_{\lambda,\bfb}^{*\vee}$ to the
wild inertia subgroup is a sum of representations of the form
$[\times \alpha/\lambda]^*\mcH_{k-1} \otimes [\times
\beta/\lambda]^*\mcH_{k-1}^{\vee}$,  representations of the
forms $[\times \alpha/\lambda]^*\mcH_{k-1}$ and
$[\times \beta/\lambda]^*\mcH_{k-1}^{\vee}$, and a trivial
representation.

Therefore, on the divisor $r=\infty$, it suffices to check that the
Swan conductor of
$$
[\times \alpha/\lambda]^*\mcH_{k-1} \otimes [\times
\beta/\lambda]^*\mcH_{k-1}^{\vee} = [\times \alpha/\lambda]^*
(\mcH_{k-1} \otimes [\times \beta/\alpha]^* \mcH_{k-1}^\vee)
$$
depends only on $(\alpha,\beta)$ but is independent of
$\lambda\in\Gg_m$, and the same property for a single hypergeometric
sheaf $[\times \alpha/\lambda]^* \mcH_{k-1}$. But scalar
multiplication does not affect Swan conductors (since it is just an
automorphism of the local field and hence preserves the wild
ramification filtration) and hence these Swan conductors are equal to
the Swan conductors of
$\mcH_{k-1} \otimes [\times \beta/\alpha]^* \mcH_{k-1}^\vee$ and
$ \mcH_{k-1}$ respectively, and thus are independent of $\lambda$.
\end{proof}

\subsection{Final steps}

In this final section, we compare different specialized sum-product
sheaves $\mcR^*_{\lambda,\bfb}$.

We now show distinctness of specialized sum-product sheaves for
distinct $\lambda$. We recall that the subvariety $\mcV^{bad}$ has
been defined in Proposition~\ref{pr-degree-bound}. It is defined over
$\Zz[1/\ell]$ and stable under
$\bfb\mapsto \tilde{\bfb}=(b_3,b_4,b_1,b_2)$.

\begin{lemma}\label{lm-partial-distinct} 
  For $\bfb$ not contained in $\mcV^{bad}(\Fq)$, and for
  $\lambda_1\not=\lambda_2$ in $\Fq$, the sheaves
  $\mcR_{\lambda_1,\bfb}^*$ and $\mcR_{\lambda_2,\bfb}^*$ are not
  geometrically isomorphic.
\end{lemma}

\begin{proof} 
  Let us recall first that, by definition, $\mcV^{bad}$ contains
  $\mcV^{\Delta}$, and therefore the sheaves $\mcR^*_{\lambda,\bfb}$ are
  geometrically irreducible for $\bfb\not\in\mcV^{bad}$.

  First assume that $\lambda_1=0$ and $\lambda_2 \neq 0$ (the case
  $\lambda_2=0$ and $\lambda_1\not=0$ is of course similar). We will
  show that the generic ranks of the two sheaves $\mcR^*_{0,\bfb}$ and
  $\mcR^*_{\lambda_2,\bfb}$ are different, which of course implies
  that they are not geometrically isomorphic. By
  Lemma~\ref{lm-lisseness} (2), (3), we have
$$
\rank \mcR_{\lambda_2,\bfb}=k^4>k^3= \rank \mcR_{0,\bfb}.
$$
Applying Lemma~\ref{lm-weight-computation} (2) exactly as in the proof
of Corollary~\ref{cor-adjusted-local-monodromy}, we see that the part
of weight $<1$ of $\mcR_{\lambda_2,\bfb}$ has rank
$$
\dim \mcK_{\lambda=\lambda_2, \eta,\bfb}^{I(0)}
$$
while (by the same argument), the  part of weight $<1$ of
$\mcR_{0,\bfb}$ has rank
$$
\dim
\mcK_{\lambda=0,\eta,\bfb}^{I(0)}+\dim\mcK_{\lambda=0,\eta,\bfb}^{I(\infty)}
\geq \dim \mcK_{\lambda=0,\eta,\bfb}^{I(0)}.
$$
However the local monodromy representation of
$\mcK_{\lambda,\eta,\bfb}$ at $0$ is independent of $\lambda$, because
$\mcK$ is a tensor product of Kloosterman sheaves defined
independently of $\lambda$ with $\sheaf{L}_{\psi(\lambda s)}$, which
is lisse at $0$. So the rank of the inertia invariants is independent
of $\lambda$ also.

Hence there is a larger ``drop'' in the generic rank when passing from
$\mcR_{\lambda,\bfb}$ to $\mcR^*_{\lambda,\bfb}$ when $\lambda$ is
$0$, and we deduce that
$$
\rank \mcR^*_{\lambda_2,\bfb}>\rank \mcR^*_{0,\bfb}.
$$


  Now assume that $\lambda_1$ and $\lambda_2$ are non-zero and
  distinct. 
  Corollary~\ref{cor-hardest-local-monodromy} shows that the wild
  inertia representation of $\mcR_{\lambda_1,\bfb}^*$ at $\infty$ is
  the multiplicative translate by $\lambda_2/\lambda_1$ of the wild
  inertia representation of $\mcR_{\lambda_2,\bfb}^*$, which is itself
  isomorphic to the wild inertia representation of
$$
\bigoplus_{\alpha\in S_k}[\times \alpha\lambda_2^{-1}]^*\mcH_{k-1}.
$$
Since the wild inertia representation of $\mcH_{k-1}$ is not
isomorphic to any non-trivial multiplicative translate of itself by
Lemma \ref{lm-hypergeo-2}, (3), these local monodromy representations
are therefore isomorphic only if $S_k=(\lambda_2/\lambda_1 )S_k$.  By
Lemma~\ref{lm-sk-geometry}, (2), this is only possible if
$\lambda_2=\lambda_1$ or if $\lambda_2=-\lambda_1$, and that second
case occurs only if $k$ is odd.

Thus it only remains to deal with the case when $k$ is odd,
$\lambda_1=-\lambda_2$, and both are non-zero. We assume that we have
a geometric isomorphism
\begin{equation}\label{eq-one-iso}
  \mcR_{\lambda_1,\bfb}^*\simeq \mcR_{-\lambda_1,\bfb}^*
\end{equation}
for some $\lambda_1\not=0$, and proceed to derive a
contradiction. This isomorphism, and the fact that
$\mcR_{\lambda_1,\bfb}^*$ and $\mcR_{-\lambda_1,\bfb}^*$ are
geometrically irreducible, implies that
$H^2_c( \Aa^1_{\bFq} -\{-\bfb\}, \mcR_{\lambda_1, \bfb}^* \otimes
\mcR_{-\lambda_1, \bfb}^{* \vee})$ is one-dimensional, where we use
$\{-\bfb\}$ to denote the closed set $\{-b_1,-b_2,-b_3,-b_4\}$.  This
cohomology group is the stalk at $\lambda_1$ of the constructible
$\ell$-adic sheaf
$$
\sheaf{G}=R^2p_!(\mcR^*_{\bfb}\otimes g^*\mcR_{\bfb}^{*\vee})(1)
$$
where $p\,:\, (\Aa^1-\{-\bfb\}) \times\Gg_m\to \Gg_m$ is the
projection $(r,\lambda)\mapsto \lambda$ and $g$ is the automorphism
$(r,\lambda)\mapsto (r,-\lambda)$ of $(\Aa^1-\{-\bfb\}) \times\Gg_m$.
\par
By Deligne's semicontinuity theorem~\cite{LaumonSMF}, the sheaf
$\sheaf{G}$ is lisse on $\Gg_m$: indeed, the Swan conductors are
constant functions of $\lambda$ on the ramification divisors, by an
argument similar to that at the end of the proof of
Proposition~\ref{pr-nonzero-irreducibility}.  Hence, since the stalk
of $\sheaf{G}$ at $\lambda_1\in\Gg_m$ is one-dimensional, the sheaf
$\sheaf{G}$ is lisse of rank $1$ on $\Gg_m$.
\par
By Verdier duality (see, e.g.,~\cite[\S 1, (1.1.3)]{katz-laumon} and
the references there, and the fact that the dual of a (shifted) lisse
sheaf is the shifted dual lisse sheaf), the dual of the sheaf
$\sheaf{G}$ is isomorphic to
$$
R^0p_*(\mcR^{*,\vee}_{\bfb}\otimes g^*\mcR^*_{\bfb}) \simeq
p_*(\mathcal{H}om(\mcR^{*}_{\bfb},g^*\mcR^*_{\bfb}))
$$
and the latter is therefore lisse on $\Gg_m$. We have a natural
adjunction morphism
$$
p^*p_*\mathcal{H}om(\mcR^{*}_{\bfb},g^*\mcR^*_{\bfb})\to
\mathcal{H}om(\mcR^{*}_{\bfb},g^*\mcR^*_{\bfb}).
$$
We tensor with $\mcR_{\bfb}^*$, and compose with the canonical
morphism $\mcR_{\bfb}^*\otimes
\mathcal{H}om(\mcR^{*}_{\bfb},g^*\mcR^*_{\bfb})\to g^*\mcR^*_{\bfb}$
to deduce a morphism
$$
\phi\colon \mcR^*_{\bfb}\otimes p^*\sheaf{G}^{\vee}\to
g^*\mcR^*_{\bfb}.
$$
The restriction to the geometric generic fiber $p^{-1} (\bar{\eta})$
of $\Aa^1-\{-\bfb\} \times \Gg_m$ of $p^*\sheaf{G}^{\vee}$ is 
$$ 
(p^*p_*\mathcal{H}om(\mcR^{*}_{\bfb},g^*\mcR^*_{\bfb})) |
p^{-1}(\bar{\eta}) 
= p^* ( p_*\mathcal{H}om(\mcR^{*}_{\bfb},g^*\mcR^*_{\bfb}) |
\bar{\eta}),
$$ 
which is
$$ (p_*\mathcal{H}om(\mcR^{*}_{\bfb},g^*\mcR^*_{\bfb})) _{\bar{\eta}}
= \Gamma(p^{-1} (\bar{\eta}), \mathcal{H}om(R_{\bfb}^*, g^*
R_{\bfb}^*)| p^{-1} (\bar{\eta}))
$$ 
viewed as a constant sheaf.

The restriction to the geometric generic fiber of the previously described morphism is a natural homomorphism
$$
\mcR^*_{\bfb}| p^{-1} (\bar{\eta})\otimes \Gamma(p^{-1} (\bar{\eta}),
\mathcal{H}om(R_{\bfb}^*, g^* R_{\bfb}^*)| p^{-1} (\bar{\eta})) \to
g^*\mcR^*_{\bfb} | p^{-1} (\bar{\eta})
$$
Specifically, we can describe this as the map that sends a section
(say $s$) of $\mcR^*_{\bfb}$ over an open subset of
$p^{-1} (\bar{\eta})$ and a global section (say $f$) over
$p^{-1} (\bar{\eta})$ of the sheaf of homomorphisms from
$\mcR^*_{\bfb}$ to $g^* \mcR_{\bfb}^*$ to the image $f(s)$. This is
because the adjunction
$p^*p_*\mathcal{H}om(\mcR^{*}_{\bfb},g^*\mcR^*_{\bfb})\to
\mathcal{H}om(\mcR^{*}_{\bfb},g^*\mcR^*_{\bfb})$ and the natural
morphism
$ \mcR^*_{\bfb,\bar{\eta}}\otimes
\mathcal{H}om(\mcR^*_{\bfb,\bar{\eta}}, g^*\mcR^*_{\bfb,\bar{\eta}})
\to g^*\mcR^*_{\bfb,\bar{\eta}} $ commute with restriction to the
generic point, and at the generic point we can unpack the definition
to see that we get this natural morphism.

This morphism is nontrivial as long as
$\Gamma(p^{-1} (\bar{\eta}), \mathcal{H}om(\mcR_{\bfb}^*, g^*
\mcR_{\bfb}^*)| p^{-1} (\bar{\eta})) $ is nonzero, as any nonzero section
must correspond to a homomorphism that is nontrivial on some open
set. The space of global sections is indeed nontrivial because we saw
it is isomorphic to the stalk of $\sheaf{G}$ at $\bar{\eta}$, which
is one-dimensional.

Hence $\phi_{\bar{\eta}}$ is nonzero on the geometric generic fiber.
Because $ \mcR^*_{\bfb}$ and $g^*\mcR^*_{\bfb}$ are geometrically
irreducible lisse sheaves, and $p^*\sheaf{G}^{\vee}$ is
one-dimensional, this implies that $\phi_{\bar{\eta}}$ is an
isomorphism. Hence $\phi$ is a geometric isomorphism on any open dense
set $U$ on which $g^*\mcR^*_{\bfb}$, $p^* \sheaf{G}$, and
$\mcR^*_{\bfb}$ are lisse.
\par
We have seen that $\sheaf{G}$ is lisse on $\Gg_m$, and we know that
$\mcR^*_{\bfb}$ is lisse on the complement of the divisors $\lambda=0$
and $r=-b_i$, and the same holds for $g^* \mcR^*_{\bfb}$. So the
homomorphism $\phi$ is a geometric
isomorphism 
on the complement $U$ of these divisors.




Our next goal is to prove that $\sheaf{G}$ is in fact geometrically
trivial. For this, we now specialize the $r$ variable. For $r$ fixed
but generic, we deduce from the above that $\mcR^*_{r,\bfb}$ is
geometrically isomorphic to $(g^*\mcR^*)_{r,\bfb} \otimes \sheaf{G}$.
However, $\mcR_{r, \bfb}$ is the restriction to $\Gg_m$ of the Fourier
transform with respect to $\psi$ of the sheaf
$$
\sheaf{F}=\bigotimes_{1\leq i\leq 2} [s\mapsto
(r+b_i)s]^*\HYPK_k\otimes [s\mapsto (r+b_{i+2})s]^*\HYPK_k^{\vee}
$$
on $\Aa^1$ with variable $s$. The sheaf $\sheaf{F}$ is lisse on
$\Gg_m$, with unipotent tame local monodromy at $0$, and with all
breaks $\leq 1/k$ at $\infty$. By Fourier transform theory it follows
that $\mcR_{r,\bfb}$ is lisse on $\Gg_m$ (see~\cite[Lemma 7.3.9
(3)]{ESDE}), with unipotent tame local monodromy at $\infty$
(\cite[Th. 7.4.1 (1), Th. 7.4.4 (3)]{ESDE}) and with all breaks
$\leq 1/(k-1)$ at $0$ (see~\cite[Th. 7.5.4 (5)]{ESDE}; note the
integers $c$, $d$ in the assumption of that reference are not
necessarily coprime).
\par
Pulling-back by $g$, we see that the sheaf $g^*\mcR_{r,\bfb}$ has the
same ramification properties, and hence also
$g^*\mcR^*_{r,\bfb}$. From this and the isomorphism
$\mcR^*_{r,\bfb}\simeq (g^*\mcR^*)_{r,\bfb}\otimes \sheaf{G}$, it
follows that $\sheaf{G}$ must also be lisse on $\Gg_m$, tame with
unipotent monodromy at $\infty$, and with (unique) break
$\leq 1/(k-1)$ at $0$. But since a rank $1$ sheaf has an integral
break, this means that $\sheaf{G}$ is also tame at $0$, and since
unipotent monodromy in rank $1$ is trivial, this means that
$\sheaf{G}$ is lisse at $\infty$. However, a sheaf on $\Pp^1$ that is
lisse on $\Pp^1-\{0\}$ and tamely ramified at $0$ is geometrically
trivial, so $\sheaf{G}$ is geometrically trivial.
\par
We have therefore proved that $\mcR^*_{\bfb}$ and $ g^*\mcR^*_{\bfb}$
are geometrically isomorphic. But this is impossible, since this would
imply that
$$
\limsup_{d\ra+\infty}
\Bigl|
\frac{1}{q^d}\frac{1}{q^{2d}} \sumsum_{ \lambda, r\in \Fqd} \bfR ( r,
\lambda,\bfb, \Fqd) \overline{ \bfR ( r, -\lambda,\bfb; \Fqd) }\Bigr|
>0
$$
(since
$\bfR(r,\lambda,\bfb;\Fqd)=t_{\mcR^*}(r,\lambda,\bfb;\Fqd)+O(1)$ where
$\mcR^*_{\bfb}$ is lisse) and this contradicts the
estimate~(\ref{Rnoncorrelation}) for odd-rank Kloosterman sheaves.
\end{proof}

We can now finally recapitulate and prove
Theorem~\ref{thmirreducibility}.

\begin{proof}[Proof of Theorem \ref{thmirreducibility}]
  Let $\mcV^{bad}$ be the subvariety in
  Proposition~\ref{pr-degree-bound}. It is defined over
  $\Zz[1/\ell]$ and hence its degree is bounded independently of
  $q$. It is also stable under $\bfb\mapsto\tilde{\bfb}$ by
  construction. For $q$ large enough, let
  $\bfb\not\in\mcV^{bad}(\Fq)$. Then $\mcR^*_{0,\bfb}$ is
  geometrically irreducible (by Proposition~\ref{pr-degree-bound}) and
  $\mcR^*_{\lambda,\bfb}$ is geometrically irreducible for all
  $\lambda\not=0$ if $q$ is large enough by
  Proposition~\ref{pr-nonzero-irreducibility} since $\mcV^{bad}$ is
  defined to contain $\mcV^{\Delta}$.
\par
The second part of Theorem~\ref{thmirreducibility} is given by
Lemma~\ref{lm-partial-distinct}, and the third by
Proposition~\ref{pr-betti-bound}.
\end{proof}


\section{Functions of triple divisor type in arithmetic progressions
  to large moduli}\label{secternary}

In this section, we prove Theorem \ref{th-lfone}. Let $f$ be a
holomorphic primitive cusp form of level $1$ and weight $k$. We denote
by $\lambda_f(n)$ the Hecke eigenvalues, which are normalized so that
we have $|\lf(n)|\leq d_2(n)$.  The method will be very similar to
that used in~\cite{FKM3} and some technical details will be handled
rather quickly as they follow very closely the corresponding steps for
the triple divisor function.
\par
For any prime $q$ and integer $a$ coprime to $q$, we denote
$$
E(\lf\star 1,x;q,a):=\sum_{\stacksum{n\leq x}{n\equiv a\mods{q}}}
(\lf\star 1)(n)-\frac{1}{\varphi(q)} \sum_\stacksum{n\leq
  x}{(n,q)=1}(\lf\star 1)(n).
$$

\subsection{Preliminaries}

We first recall several useful results. We begin with stating the
estimates for linear and bilinear forms involving the
hyper-Kloosterman sums $\hypk_3(a;q)$.

\begin{proposition}\label{CombinedTheorem}
  Let $q$ a prime number,
  $M,N\in [1,q]$, $\mcN$ an interval of length $N$, and
  $(\alpha_m)_m$, $(\beta_n)_n$ two sequences supported respectively
  on $[1,M]$ and $\mcN$. Let $a$ be an integer coprime to $q$.
\par
Let $V$ and $W$ be smooth functions compactly supported in the
interval $[1,2]$ and satisfying
\begin{equation}\label{Wbound2}
  V^{(j)}(x), W^{(j)}(x)\ll_j Q^j
\end{equation}
for some $Q\geq 1$ and for all $j \geq 0$.
\par
 Let $\eps>0$ be given.
\par
\emph{(1)} There exists an absolute constant $C_1\geq 0$ such that we
have
\begin{equation}\label{PV}
  \sumsum_{m,n\geq 1}
  \lf(m)V\Bigl(\frac{m}{M}\Bigr)W\Bigl(\frac{n}{N}\Bigr)  
  \Kl_3(amn;q)\ll q^\eps Q^{C_1}
  MN\Bigl(\frac{1}q+\frac{q^{1/2}}N\Bigr).
\end{equation}
and 
\begin{equation}\label{FKM1}
  \sum_{m} \lf(m)V\Bigl(\frac{m}{M}\Bigr)   \Kl_3(am;q)\ll q^\eps Q^{C_1}
  M\Bigl(\frac1{q^{1/8}}+\frac{q^{3/8}}{M^{1/2}}\Bigr).
\end{equation}
\par
\emph{(3)} We have
\begin{equation}\label{typeIIPV}
  \sumsum_{m\leq M,\,n\in\mcN}{\alpha_m}\beta_n \Kl_3(amn;q)\ll q^{\eps}
  \| \alpha \|_2 \| \beta \|_2(MN)^{1/2}  
  \Bigl(\frac{1}{M^{1/2}}+\frac{q^{1/4}}{N^{1/2}}\Bigr).
\end{equation} 
\par
\emph{(3)} If
$$
1\leq M\leq N^2,\quad N<q,\quad MN\leq q^{3/2},
$$ 
we have
\begin{equation}\label{typIsmooth}
  \sumsum_{m,n\geq 1} \lf(m)V\Bigl(\frac{m}{M}\Bigr)W\Bigl(\frac{n}{N}\Bigr)   
  \Kl_3(amn;q)\ll q^\eps Q^{C_1}
  MN\Bigl(\frac{q^{1/4}}{M^{1/6}N^{5/12}}\Bigr).	
\end{equation}
\par
In all estimates, the implied constant depends only on $\eps$.
\end{proposition}

\begin{proof}
  The bound \eqref{PV} is an instance of the completion method
  and follows from an application of the Poisson summation formula to
  the sum over $n$, using the fact that
$$
\widehat{\Kl}_3(u)\ll 1,\quad \widehat{\Kl}_3(0)=-\frac{1}{p^{3/2}}
$$
(the former because $\Kl_3(\cdot;q)$ is the trace function of a
Fourier sheaf modulo $q$, and the latter by direct computation).
\par
The bounds \eqref{FKM1} and \eqref{typeIIPV} are special cases of
\cite[Thm. 1.2]{FKM1} and \cite[Thm. 1.17]{FKM2}. The bound
\eqref{typIsmooth} is a consequence of Theorem~\ref{thmtypeI} (for
$c=1$) after summation by parts.
\end{proof}

\begin{proposition}\label{fcteqn}
  Let $q$ be a prime number.  Let $V,W$ be two smooth functions
  compactly supported on $]0,+\infty[$, and let $K:\Zz\lra\Cc$ be any
  $q$-periodic arithmetic function.
\par 
We have
\begin{align*}
  \sumsum_{m,n\geq 1}K(mn)\lambda_{f}(m)V(m)W(n)=
  &\frac{\what K(0)}{q^{1/2}}\sumsum_{m,n\geq 1}\lambda_{f}(m)V(m)W(n)\\
  &+\Bigl( K(0)-\frac{\what K(0)}{q^{1/2}}\Bigr)
    \sumsum_{m,n\geq 1}\lambda_{f}(m)V(m)W({qn})\\
  &+\frac{1}{q^{3/2}}\sum_{m,n\geq 1}\tilde K(mn)
    \lambda_{f}(m)\check V\Bigl(\frac{m}{q^2}\Bigr)\hat W\Bigl(
    \frac{n}{q}\Bigr),
\end{align*}
where $\hat W$ denotes the Fourier transform of $W$, $\check{V}$ is
the weight $k$ Bessel transform given by
\begin{equation}\label{Besseldef}
  \check V(x)=2\pi i^k\int_0^\infty V(t)J_{k-1}(4\pi\sqrt{xt})dt,	
\end{equation}
and
$$
\widetilde K(m)=\frac{1}{q^{1/2}}\sum_{(u,q)=1}K(u)\Kl_3(mu;q).
$$
\end{proposition}

In particular, if $a$ is an integer coprime with $q$ and
$K(n)=\delta_{n=a\mods q}$, then we have
$$
K(0)=0,\quad \hat K(0)=\frac{1}{q^{1/2}},\quad \widetilde
K(m)=\frac{1}{q^{1/2}}\Kl_3(am;q)
$$
by direct computations.

\begin{proof}
  We split the sum  into
\begin{equation}\label{eq-split}
  \sum_{q|n}\Bigl(\sum_m
  \cdots\Bigr)+\sum_{(n,q)=1}\Bigl(\sum_m\cdots\Bigr).
\end{equation}
\par
The contribution of those $n$ divisible by $q$ is
$$
K(0)\sum_{m,n\geq 1}\lambda_{f}(m)V(m)W({qn}).
$$
\par
For those $n$ coprime to $q$, we apply the Fourier inversion formula
$$
K(mn)= \frac{\what
  K(0)}{q^{1/2}}+\frac{1}{q^{1/2}}\sum_\stacksum{u\mods
  q}{(u,q)=1}\what K(u)e\Bigl(-\frac{umn}{q}\Bigr).
$$
The contribution of the first term is
$$
\frac{\what K(0)}{q^{1/2}}\sum_\stacksum{m,n\geq
  1}{(n,q)=1}\lambda_{f}(m)V(m)W(n)=\frac{\what
  K(0)}{q^{1/2}}\Bigl(\sum_{m,n\geq 1}\lambda_{f}(m)V(m)W(n)-
\sum_{m,n\geq 1}\lambda_{f}(m)V(m)W(qn)\Bigr).
$$
\par
For the last term, we apply the Voronoi summation formula to the sum
over $m$: we have
$$
\sum_{m\geq 1}\lf(m)V(m) e\Bigl(-\frac{mnu}{q}\Bigr)=
\frac{1}{q}\sum_{m\geq 1}\lambda_f(m)\check V\Bigl(\frac{m}{q^2}\Bigr)
e\Bigl(\frac{\ov{nu}m}{q}\Bigr)
$$
for each $u$ (see, e.g.,~\cite[Lemma 2.2]{FGKM}).  Therefore, the
total contribution of the second term in~(\ref{eq-split}) equals
$$
\frac{1}{q}\sumsum_\stacksum{m,n\geq 1}{(n,q)=1}\lf(m)\check
V\Bigl(\frac{m}{q^2}\Bigr)W(n)\tilde K(m,n)
$$
with
$$
\tilde K(m,n)=\frac{1}{q^{1/2}}\sum_{(u,q)=1}\what
K(u)e\Bigl(\frac{m\ov{nu}}{q}\Bigr).
$$
\par
We finish by applying the Poisson summation formula to the sum over
$n$: we have
$$
\sum_{(n,q)=1}W(n)\tilde K(m,n) = \frac{1}{q} \sum_{n}\hat
W\Bigl(\frac{n}q\Bigr)
\sum_{(v,q)=1}e\Bigl(\frac{m\ov{uv}+nv}{q}\Bigr)=
\frac{1}{q^{1/2}}\sum_{n}\hat W\Bigl(\frac{n}q\Bigr) \Kl_2(mn\ov u;q)
$$
for each $m$, so that the total contribution becomes
$$
\frac{1}{q^{3/2}}\sumsum_\stacksum{m,n}{(n,q)=1}\tilde
K(mn)\lf(m)\check V\Bigl(\frac{m}{q^2}\Bigr)\hat
W\Bigl(\frac{n}q\Bigr)
$$
where
$$
\tilde{K}(m)=\frac{1}{q^{1/2}}\sum_{(u,q)=1}\hat K(u)\Kl_2(m\ov u;q)=
\frac{1}{q^{1/2}}\sum_{(u,q)=1}K(u)\Kl_3(mu;q).
$$
for any $m$. This gives the formula we stated.
\end{proof}

\subsection{Decomposition of $E(\lf\star 1,x;q,a)$}

Given any $A\geq 1$ as in Theorem \ref{th-lfone}, we fix some $B\geq 1$ sufficiently large (to depend on $A$).  Given $x\geq 2$, we
set
$$
\mcL:=\log x,\quad \Delta=1+\mcL^{-B}.
$$
Arguing as in \cite{FKM3}, we perform a partition of unity on the $m$
and $n$ variables and decompose
$E(\lf\star 1,x;q,a)$ into $O(\log^2 x)$ terms of the form
$$
\tilde{E}(V,W;q,a)=\sum_{mn=a\mods
  q}\lambda_{f}(m)V(m)W(n)-\frac{1}q\sum_{(mn,q)=
  1}\lambda_{f}(m)V(m)W(n)
$$
where $V,W$ are smooth functions satisfying 
\begin{gather*}
  \supp V\subset [M,\Delta M],\ \supp W\subset [N,\Delta N]\\
  x^jV^{(j)}(x),\ x^jW^{(j)}(x)\ll_j \mcL^{Bj}
\end{gather*}
and where
$$
x\mcL^{-C}\leq MN\leq x
$$
for some $C\geq 0$ large enough, depending on the value of the
parameter $A$ in Theorem \ref{th-lfone}.
\par
Applying Proposition \ref{fcteqn} to the first term, we obtain 
$$
\tilde{E}(V,W;q,a)= \frac{1}{q^{1/2}}\sum_{m,n\geq 1}\Kl_3(mn;q)
\lambda_{f}(m)\check V\Bigl(\frac{m}{q^2}\Bigr)\hat
W\Bigl(\frac{n}{q}\Bigr),
$$
and hence it only remains to prove that 
\begin{equation}\label{wishedbound}
  \frac{1}{q^{1/2}}\sum_{m,n\geq 1}\Kl_3(mn;q)\lambda_{f}(m)
  \check V\Bigl(\frac{m}{q^2}\Bigr)
  \hat W\Bigl(\frac{n}{q}\Bigr)\ll_{A} \frac {MN}q\mcL^{-A}.	
\end{equation}

The following standard lemma describes the decay of the Fourier and
Bessel transforms of $V$ and $W$.

\begin{lemma} 
  Let $V,W$ be as above and let $\check W$, $\hat W$ be their Bessel
  and Fourier transforms as defined in \eqref{Besseldef}. There exists
  a constant $D\geq 0$ such that for any $x>0$, any $E\geq 0$ and any
  $j\geq 0$, we have
\begin{gather}
\label{Vbound}
x^j\check V^{(j)}(x)\ll_{E,f,j} M\mcL^{Bj}\Bigl(\frac{\mcL^{Dj}}{1+xM}
\Bigr)^{E},\\
\label{Wbound} x^j\hat W^{(j)}(x)\ll_{E,j}
N\mcL^{Bj}\Bigl(\frac{\mcL^{Dj}}{1+xN}\Bigr)^{E}.
\end{gather}
\end{lemma} 

\begin{proof} 
  By the change of variable $u=4\pi\sqrt{xt}$, we find that
$$
\check V(x)=\frac{i^k}{8\pi}\int_0^{\infty}
\frac{u^2}{x}V\Bigl(\frac{u^2}{16\pi^2x}\Bigr) J_{k-1}(u)\frac{du}{u}=
\frac{i^k}{8\pi\sqrt{x}}\int_0^{\infty}\Bigl(\frac{u^2}{x}\Bigr)^{1/2}
V\Bigl(\frac{u^2}{16\pi^2 x}\Bigr)J_{k-1}(u)du.
$$
Since $J_{k-1}(u)\ll_f (1+u)^{-1/2}$,
we have
$$
\check V(x)\ll \frac{M}{(1+xM)^{1/2}}.
$$
On the other hand, applying \cite[Lem. 6.1]{KMV} we obtain the bound
$$
\check V(x)\ll_{f,j}\frac{M^{1/2}}{x^{1/2}}\frac{(1+|\log
  xM|)\mcL^{O(j)}}{(xM)^{\frac{j-1}2}}(xM)^{1/4}.
$$
In particular if $xM\geq 1$, then by taking $j$ large enough, we see
that $\check V(x)\ll_{f,E}M\mcL^{O(E)}(xM)^{-E}$, which concludes the
proof of \eqref{Vbound} when $j=0$. The general case is similar, and
the proof of \eqref{Wbound} follows similar lines (using easier
standard properties of the Fourier transform).
\end{proof}

Set
$$
M^*=q^2/M\text{ and }N^*=q/N.
$$
Then this lemma shows that, if $\eta>0$ is arbitrarily small, the
contribution to the sum \eqref{wishedbound} of the $(m,n)$ such that
$$
m\geq x^{\eta/2} M^*\text{ or }n\geq x^{\eta/2} N^*
$$
is negligible. Therefore, by \eqref{Vbound} and \eqref{Wbound}, and a
smooth dyadic partition of unity, we are reduced to estimating sums of
the type
$$
S(M',N')=\sumsum_{m,n\geq 1}\lf(m){\Kl_3(amn;q)}V^*(m)W^*(n)
$$
where
$$
1/2\leq M'\leq M^*x^{\eta/2},\quad 1/2\leq N'\leq N^*x^{\eta/2},
$$
and 
$V^*,W^*$ are smooth compactly supported functions with
\begin{gather*}
\supp (V^*)\subset [M',2 M'],\quad \supp(W^*)\subset [N',2 N']
\\
u^jV^*(j)(u),\quad u^jW^*(j)(u)\ll \mcL^{O(j)}
\end{gather*}
for any $j\geq 0$.
Precisely, it is enough to prove that
$$
S(M',N')\ll_A q\mcL^{-A}.
$$
Since the trivial bound for $S(M',N')$ is
$$
S(M',N')\ll M'N'\mcL,
$$
we may assume that
$$
q\mcL^{-A-1}\leq M'N'\leq q^3x^{-1+\eta}.
$$

Let us write
$$
x=q^{2-\delta},\quad M=q^\mu,\quad N=q^\nu,\quad M'=q^{\mu'},\quad
N'=q^{\nu'}
$$
so that
$$
M^*=q^{\mu^*},\quad N^*=q^{\nu^*}
$$
with
$$
\mu^*=2-\mu,\quad \nu^*=1-\nu,\quad \mu'\leq \mu^*+\eta/2,\quad
\nu'\leq \nu^*+\eta/2
$$
and 
$$
\mu+\nu=2-\delta+o(1).
$$
\par
Let us write
$$
S(M',N')=q^{\sigma(\mu',\nu')}.
$$
Then Proposition~\ref{CombinedTheorem} translates to the estimates
$$
\sigma(\mu',\nu')\leq \tau(\mu',\nu')+o(1)
$$
where
\begin{gather}
\label{eq-PV-lin}
\tau(\mu',\nu')\leq \mu'+\nu'+\max(-1,1/2-\nu')
\quad\text{(by~(\ref{PV}))}\\
\label{eq-FKM1-lin}
\tau(\mu',\nu')\leq \mu'+\nu'+\max(-1/8,3/8-\mu'/2)
\quad\text{(by~(\ref{FKM1}))}\\
\label{eq-typeIIPV1}
\tau(\mu',\nu')\leq \mu'+\nu'+\max(-\mu'/2,1/4-\nu'/2)
\quad\text{(by~(\ref{typeIIPV}))}\\
\label{eq-typeIIPV2}
\tau(\mu',\nu')\leq \mu'+\nu'+\max(-\nu'/2,1/4-\mu'/2)
\quad\text{(by~(\ref{typeIIPV}) with $M$, $N$ interchanged)}\\
\label{eq-typeI}
\tau(\mu',\nu')\leq \mu'+\nu'+1/4-\mu'/6-5\nu'/12
\quad\text{(by~(\ref{typIsmooth}), if $0\leq \mu'\leq 2\nu'$)}
\end{gather}
(indeed, note that the conditions $\nu'\leq 1$ and $\mu'+\nu'\leq 3/2$
also required in~(\ref{typIsmooth}) are always satisfied for $\eta$
small enough, since $\mu'+\nu'\leq 3+(2-\delta)(-1+\eta)<\frac32$ and
$\nu'=1-\nu\leq 1$).

We will prove that if $\delta<\frac{1}{26}$ and $\eta$ is small
enough, then we have $\sigma(\mu',\nu')\leq 1-\kappa$, where
$\kappa>0$ depends only on $\delta$ and $\eta$. This implies the
desired estimate. In the argument, we denote by $o(1)$ quantities
tending to $0$ as $\eta$ tends to $0$ or $q$ tends to infinity.

First, since 
$$
\mu'+\nu'-1\leq \mu'+\nu'-\frac1{8}\leq 1+\delta-\frac18+o(1)<1
$$
we may replace \eqref{eq-PV-lin} and \eqref{eq-FKM1-lin} by
\begin{gather}
\label{eq-PV-linbis}
\tau(\mu',\nu')\leq \mu'+\frac{1}{2}
\\
\label{eq-FKM1-linbis}
\tau(\mu',\nu')\leq \frac{\mu'+\nu'}{2}+\frac{\nu'}{2}+\frac{3}{8}.
\end{gather}


We now distinguish various cases:
\begin{itemize}
\item If $\mu'\leq \frac{1}{2}-\kappa$, then we obtain the bound
  by~(\ref{eq-PV-linbis});
\item If 
$$
\nu'>2(\delta+\kappa)\text{ and } \mu'>
\frac{1}{2}+(2\delta+\kappa),
$$
then we obtain the bound by~(\ref{eq-typeIIPV2});
\item If $\nu'\leq 2(\delta+\kappa)$, we obtain a suitable bound,
  provided $\kappa$ is small enough, by~(\ref{eq-FKM1-linbis}) since
  then
$$
\frac{\mu'+\nu'}{2}+\frac{\nu'}{2}+\frac{3}{8} \leq
\frac{1+\delta+o(1)}{2}+\delta+\kappa+\frac{3}{8}\leq
\frac{7}{8}+\frac{3\delta}{2}+\kappa+o(1);
$$
\item Finally, if $\mu'\leq (2\delta+\kappa)+\frac{1}{2}$, then from
  $\mu'+\nu'\geq 1$, we deduce that
$$
2\nu'\geq 1-4\delta -2\kappa\geq \frac{1}{2}+2\delta-4\kappa
\geq \mu'
$$
provided $\kappa$ is small enough, and so~(\ref{eq-typeI}) is
applicable and gives the desired bound since
\begin{align*}
  \mu'+\nu'+\frac{1}{4}-\frac{\mu'}{6}-\frac{5\nu'}{12} 
  &=
    \frac{7}{12}(\mu'+\nu')+\frac14+\frac{\mu'}{4}+o(1) \\
  &\leq
    \frac{7}{12}(1+\delta)+\frac14+
    \frac{1}{4}\Bigl(\frac12+2\delta+\kappa
    \Bigr)+o(1)
  \\
  &
    =1-\frac{13}{12}(1/26-\delta)+\frac{\kappa}{4}+o(1).
\end{align*}
\end{itemize}


\appendix

\section{Nearby and vanishing cycles}\label{app-nearby}

Let $R$ be a Henselian discrete valuation ring $R$ with fraction field
$K$. Let $S$ be the spectrum of $R$, and denote its generic point by
$\eta$ and its special point by $s$. Let $\bar{\eta}$ be a geometric
point over $\eta$ and $\bar{s}$ a geometric point over $s$. 
\par
For any proper scheme $f\,:\, X\lra S$, and any prime $\ell$
invertible on $S$, the nearby cycles function $R \Psi$ is a functor
from $\ell$-adic sheaves on $X_{\eta}$ to the derived category of
$\ell$-adic sheaves on $X_{\overline{s}}$ equipped with an action of
the absolute Galois group $G$ of $K$. (See, e.g.,~\cite[Exp. XIII]{SGA7} 
for the definition and further references.)
\begin{equation}\label{nearbycyclediagram}
\xymatrix{
X_s\ar[r]^-{i}\ar[d]&X\ar[d]_{f}&X_{\ov\eta}\ar[l]_-{j}\ar[d]\\
s\ar[r]^-{i_0}&S&\ov\eta\ar[l]_-{j_0}.
}
\end{equation}
\par
Given $\mcF$ a sheaf on $X$ and  $\mcF_s:=i^*\mcF$ and $\mcF_{\ov\eta}:=j^*\mcF$, the complex $R \Psi\mcF$ is defined as
$$R \Psi\mcF=i^*Rj_*\mcF_{\ov\eta}.$$
The mapping cone of the adjunction map
$i^*\mcF\ra  R \Psi\mcF$
is noted $R\Phi\mcF$ is is called the complex of vanishing cycles; one then has a cohomology exact sequence arising from the corresponding distinguished triangle
\begin{equation}\label{eq-vanishing-exact}\cdots\ra H^i(X_{\ov s},\mcF_s)\ra H^i(X_{\overline{s}}, R \Psi
\mathcal F)\ra H^i(X_{\overline{s}}, R \Phi
\mathcal F)\ra \cdots \end{equation}

The functor $R\Psi$ has several key properties that we use in this paper:
\par
(1) (See~\cite[(1.3.3.1)]{LaumonSMF}, \cite[(2.1.8.3)]{SGA7}) For any
$i\geq 0$, there is a natural isomorphism of $G$-representations
\begin{equation}\label{eq-nearby-local}
H^i (X_{\overline{\eta}}, \mathcal F) = H^i(X_{\overline{s}}, R \Psi
\mathcal F).
\end{equation}
Since the left-hand side of~(\ref{eq-nearby-local}) is, together with
its Galois action, the local monodromy representation of the
higher-direct image sheaf $R^if_*\sheaf{F}$ at $s$, the nearby cycle
complex will enable us to compute the local monodromy representation
at specific points of some global sheaves obtained by push-forward on
curves.
\par
(2) (See~\cite[Th 1.3.1.3]{LaumonSMF},~\cite[Th. Finitude,
Prop. 3.7]{sga4h}) The functor $R \Psi$ is defined \'{e}tale-locally:
if two pairs $(X\lra S,\sheaf{F})$ and $(X'\lra S,\sheaf{F}')$ are
given which are isomorphic in an \'{e}tale neighborhood of a point
$x\in X$, i.e., if there exist a scheme $U$ over $S$, a point
$\tilde{x}\in U$ and \'etale morphisms making the diagram
$$
\begin{array}{ccc}
U&\fleche{g'}&X'\\
g\ \downarrow&& \downarrow f'\\
X&\fleche{f} & S
\end{array}
$$
commute with $g(\tilde{x})=x$, $g'(\tilde{x})=x'$ (say), and if
$g^*\sheaf{F}\simeq (g')^*\sheaf{F}'$, then we have
$$
g^*R\Psi\sheaf{F}\simeq (g')^*R\Psi\sheaf{G}
$$
(i.e., the nearby cycles complexes are isomorphic in the same
\'{e}tale neighborhood.)

This will be useful to compare the local monodromy of a given sheaf on
a given curve to possibly simpler ones on other (also possibly
simpler) curves, which are étale-locally isomorphic and take advantage
of some existing computations of nearby cycles : for instance the
local acyclicity of smooth morphisms (which handles the case of a
lisse sheaf on a smooth scheme) and Laumon's local Fourier transform
which describes the nearby cycles that arise when computing the
Fourier transform of a sheaf (aka the stationary phase formula).


\end{document}